\documentclass[letterpaper, 11pt,  reqno]{amsart}
\usepackage[margin=1.2in,marginparwidth=1.5cm, marginparsep=0.5cm]{geometry}

\usepackage{amsmath,amssymb,amscd,amsthm,amsxtra, esint, xcolor}
\usepackage[implicit=true]{hyperref}

\usepackage[shortlabels]{enumitem}

\allowdisplaybreaks[2]

\usepackage{color}

\definecolor{gr}{rgb}   {0.,   0.69,   0.23 }
\definecolor{bl}{rgb}   {0.,   0.5,   1. }
\definecolor{mg}{rgb}   {0.85,  0.,    0.85}
\definecolor{yl}{rgb}   {0.8,  0.7,   0.}
\definecolor{or}{rgb}  {0.7,0.2,0.2}

\newtheorem{theorem}{Theorem} [section]

\newtheorem{lemma}[theorem]{Lemma}
\newtheorem{proposition}[theorem]{Proposition}
\newtheorem{remark}[theorem]{Remark}

\DeclareMathOperator*{\supp}{supp}

\newcommand{\I}{\hspace{0.5mm}\text{I}\hspace{0.5mm}}

\newcommand{\II}{\text{I \hspace{-2.8mm} I} }
\newcommand{\III}{\text{I \hspace{-2.9mm} I \hspace{-2.9mm} I}}

\newcommand{\noi}{\noindent}
\newcommand{\Z}{\mathbb{Z}}
\newcommand{\R}{\mathbb{R}}
\newcommand{\C}{\mathbb{C}}
\newcommand{\T}{\mathbb{T}}
\newcommand{\bul}{\bullet}

\newcommand{\hi}{\textup{hi}}
\newcommand{\HI}{\textup{HI}}
\newcommand{\lo}{\textup{lo}}
\newcommand{\LO}{\textup{LO}}

\newcommand{\W}{\mathcal{W}}

\let\Re=\undefined\DeclareMathOperator*{\Re}{Re}

\let\P= \undefined
\newcommand{\P}{\mathbf{P}}

\newcommand{\Q}{\mathbf{Q}}

\renewcommand{\L}{\mathcal{L}}

\newcommand{\J}{\mathcal{J}}

\newcommand{\GG}{\mathcal{G}}

\newcommand{\F}{\mathcal{F}}

\newcommand{\al}{\alpha}
\newcommand{\be}{\beta}
\newcommand{\dl}{\delta}

\newcommand{\eps}{\varepsilon}
\newcommand{\kk}{\kappa}
\newcommand{\g}{\gamma}

\newcommand{\ld}{\lambda}

\newcommand{\s}{\sigma}

\newcommand{\ft}{\widehat}
\newcommand{\Ft}{{\mathcal{F}}}
\newcommand{\wt}{\widetilde}
\newcommand{\cj}{\overline}
\newcommand{\dx}{\partial_x}

\newcommand{\dt}{\partial_t}

\newcommand{\embeds}{\hookrightarrow}
\newcommand{\LRA}{\Longrightarrow}

\newcommand{\ta}{\theta}

\renewcommand{\l}{\ell}

\newcommand{\les}{\lesssim}
\newcommand{\ges}{\gtrsim}

\newcommand{\jb}[1]
{\langle #1 \rangle}

\newcommand{\ind}{\mathbf 1}

\renewcommand{\S}{\mathcal{S}}

\newcommand{\N}{\mathbb{N}}
\newcommand{\NN}{\mathcal{N}}

\newtheorem*{ackno}{Acknowledgements}

\renewcommand{\H}{\mathcal{H}}

\def\sgn{\textup{sgn}}

\newcommand{\Pbhi}{\mathbf{P}_{\textup{hi}} }
\newcommand{\Pblo}{\mathbf{P}_{\textup{lo}} }
\newcommand{\Pbhip}{\mathbf{P}_{\textup{+,hi}} }
\newcommand{\PbHIp}{\mathbf{P}_{\textup{+,HI}} }
\newcommand{\PbHI}{\mathbf{P}_{\textup{HI}}}
\newcommand{\PbLO}{\mathbf{P}_{\textup{LO}}}

\newcommand{\TT}{\mathcal{T}}
\newcommand{\QQ}{\mathcal{Q}}

\newcommand{\Id}{\textup{Id}}

\newcommand{\lax}{\mathcal{L}}
\newcommand{\peter}{\mathcal{P}}
\newcommand{\op}{\textup{op}}

\newcommand{\Pih}{\Pi_{+,h}}

\numberwithin{equation}{section}
\numberwithin{theorem}{section}

\makeatletter
\@namedef{subjclassname@2020}{\textup{2020} Mathematics Subject Classification}
\makeatother

\begin{document}
\baselineskip = 14pt

\title[Well-posedness for INLS]{On the well-posedness of the intermediate nonlinear Schr\"{o}dinger equation on the line}

\author[A.~Chapouto, J.~Forlano, T.~Laurens]
{Andreia Chapouto, Justin Forlano, Thierry Laurens}

\address{Andreia Chapouto,
CNRS, and Laboratoire de math\'ematiques de Versailles, UVSQ, Universit\'e Paris-Saclay, CNRS, 45 avenue des \'Etats-Unis, 78035 Versailles Cedex, France}

\email{andreia.chapouto@uvsq.fr}

\address{Justin Forlano,
School of Mathematics, Monash University, VIC 3800, Australia}

\email{justin.forlano@monash.edu}

\address{Thierry Laurens,
Department of Mathematics, University of Wisconsin–Madison, WI, 53706, USA}
\email{laurens@math.wisc.edu}




\begin{abstract}
We consider a family of intermediate nonlinear Schr\"{o}dinger equations (INLS) on the real line, which includes the continuum Calogero-Moser models (CCM). We prove that INLS is locally well-posed in $H^{s}(\R)$ for any $s>\frac 14$, which improves upon the previous best result of $s>\frac 12$ by de Moura-Pilod (2008). This result is also new in the special case of CCM, as the initial condition is not required to lie in any Hardy space.

Our approach is based on a gauge transformation, exploiting the remarkable structure of the nonlinearity together with bilinear Strichartz estimates, 
which allows to recover some of the derivative loss. 
This turns out to be sufficient to establish our main results for CCM in the Hardy space. 
For INLS and CCM outside of the Hardy space, the main difficulty comes from the lack of the Hardy space assumption, which we overcome by implementing a refined decomposition of the solutions, which observes a nonlinear smoothing effect in part of the solution. 

We also discover a new Lax pair for INLS and use it to establish global well-posedness in $H^{s}(\R)$ for any $s>\frac 14$ under the additional assumption of small $L^2$-norm.
\end{abstract}

\maketitle

\tableofcontents

\section{Introduction}

We consider the Cauchy problem for the
intermediate nonlinear Schr\"{o}dinger equation (INLS):
\begin{equation}
\left\{
\begin{aligned}
  & \dt u+i\dx^2 u = \be u (1+i\TT_{h})\dx(|u|^2)+ i\g |u|^2 u,\\
   & u|_{t=0}=u_0,
\end{aligned}
 \right. \label{INLS}
\end{equation}
 where $u:\R\times \R\to \C$, $\be,\g\in \R$, and 
  $\TT_{h}$ is the singular integral operator with kernel 
  \begin{align}
\mathcal{T}_{h}f(x) = \frac{1}{2h} \text{p.v.} \int_{-\infty}^{\infty} \coth\bigg( \frac{\pi(x-y)}{2h} \bigg) f(y)dy, \quad 0<h<\infty,\label{tilbert}
\end{align}
 where $\text{p.v.}$ denotes the principal value. By taking the Fourier transform, we see that $\TT_{h}$ is a Fourier multiplier operator with multiplier 
\begin{align}
\F\{\TT_{h} f\}(\xi)  =-i\coth(h\xi) \ft f(\xi), 
\quad \xi \in \R\setminus\{0\}. 
\label{TTft}
\end{align}
The INLS equation \eqref{INLS} with $\g=0$ was derived by Pelinovsky~\cite{Pel1} as a model for the evolution of quasi-harmonic internal waves in a two fluid layer system, where the bottom fluid is of a finite depth $h>0$. In the same work, Pelinovsky demonstrated the existence of multi-soliton solutions, strongly indicating that \eqref{INLS} is completely integrable. This was verified by Pelinovsky-Grimshaw~\cite{Pel3} where they developed the inverse scattering transform, used it to explain the multi-solitons, and found an infinite sequence of conservation laws.  
The scattering transform (and Lax pair) in \cite{Pel3} involves $2\times 2$ matrix operators, which can be heuristically explained since in the shallow-depth limit $(h\to 0)$, and after a suitable change of variables, \eqref{INLS} formally converges to the cubic nonlinear Schr\"{o}dinger equation. 
Later, Pelinovsky-Grimshaw~\cite{Pel2} took into account higher order effects leading to \eqref{INLS} for $\g \in  \R$. 

Regarding the well-posedness theory for INLS \eqref{INLS}, little appears to be known. 
An important aspect of this is to determine scaling critical spaces, which suggest where the barrier to well-posedness lies. 
Given $\ld \geq 1$ and a smooth solution $u$ to \eqref{INLS}, the rescaled solution $u_{\ld}(t,x) = \ld^{-\frac 12} u(\ld^{-2}t, \ld^{-1}x) $
 satisfies
 \begin{align}
\dt u_{\ld} +i\dx^2 u_{\ld} = 2\be u_{\ld} \TT_{h\ld}\dx(|u_{\ld}|^2) +i\ld^{-1}\g |u_{\ld}|^2 u_{\ld}
\notag
\end{align}
with initial data $u_{0,\ld}(x) = u_0(\ld^{-1}x)$. When $\g=0$, the scaling is not an exact symmetry; rather, the family of equations \eqref{INLS} with depth parameters $0<h<\infty$ remains invariant under scaling.  
This scaling then reveals that $L^2(\R)$ is the scaling critical space.  When $\g \neq 0$, the contribution from the term $|u|^2 u$ scales favourably so we still expect criticality in $L^2 (\R)$.
In this direction, de Moura~\cite{demoura1} established the local well-posedness in $H^s (\R)$ for any $s\geq 1$ and for small initial data. Using a gauge transformation, de Moura-Pilod \cite{PMP} proved local well-posedness in $H^{s}(\R)$ for any $s>\frac 12$ and removed the small data restriction. Later, Barros-de Moura-Santos \cite{BdMS} proved local well-posedness for \eqref{INLS} with sufficiently small initial data in the Besov space $B^{\frac 12}_{2,1}(\R)$.
It was shown in \cite{PdM} that, if it exists, the solution map cannot be $C^{3}$ at the origin in $H^{s}(\R)$ for any $s<0$. 
One of the main goals of this 
article is to go beyond the results of \cite{PMP, BdMS} and make progress towards the scaling critical space $L^2(\R)$ without using complete integrability. This opens up handling perturbations of \eqref{INLS} such as the physically relevant case with $\g \neq 0$, which may not be completely integrable.

 For fixed $\xi\neq 0$, from \eqref{TTft}, we see that $-i\coth(h\xi) \to -i\sgn(\xi)$, which we recognise as the Fourier multiplier associated to the Hilbert transform $\H$. Thus, by formally taking $h\to \infty$ in \eqref{INLS} we arrive at our second equation of interest in this paper:
 \begin{equation}
   \dt u+i\dx^2 u = \be u (1+i\H)\dx(|u|^2)+ i\g |u|^2 u.
\label{CCM0}
\end{equation}
This discussion suggests defining $\TT_{\infty}=\H$, extending \eqref{INLS} to the case when $h=\infty$.
We rewrite \eqref{CCM0} into a more familiar form.  Given a set $A\subseteq \R$, we write $\ind_{A}$ to be the characteristic function of the set $A$. Then, with $\P_{\pm}$ defined as the Fourier multiplier operator with symbol $\ind_{\{ \pm \xi>0\}}$, we have
 \begin{align}
\P_{+}+\P_{-} =\text{Id}\label{pm1}
\end{align}
and 
\begin{align}
\H= -i \P_{+}+i\P_{-} \quad \text{so that} \quad 1+i\H = 2\P_+. \label{Hilbert}
\end{align}
Inserting the second identity in \eqref{Hilbert} into \eqref{CCM0}, we arrive at the equation:
 \begin{equation}
   \dt u+i\dx^2 u = 2\be u \P_{+} \dx(|u|^2)+ i\g |u|^2 u.
\label{CCMg0}
\end{equation}
When $\g=0$, we recognise \eqref{CCMg0} as the continuum Calogero-Moser equation (CCM):
 \begin{equation}
   \dt u+i\dx^2 u = 2\be u \P_{+} \dx(|u|^2).
\label{CCM}
\end{equation}
 The sign of $\be$ determines the type of CCM we consider: defocusing if $\be>0$ or focusing if $\be<0$. The defocusing and focusing varieties arise in unique physical contexts. The defocusing CCM \eqref{CCM} was derived by Pelinovsky \cite{Pel1} as the infinite depth limit of \eqref{INLS}, whereas the focusing CCM \eqref{CCM} was formally derived in \cite{Abanov} as a continuum limit of classical Calogero-Moser particle systems.

Remarkably, the infinite-depth model~\eqref{CCM} formally leaves the Hardy space $L^2_{+}(\R)$ invariant (see \eqref{hardy}).
More precisely, if $u_0 \in L^2_{+}(\R)\cap H^{\infty}(\R)$, then the same is true for the solution $u$. Indeed, by the Fourier convolution theorem, we have 
\begin{align}
\F\{ u\P_{+}\dx(|u|^2)\}(\xi)  =\int_{\xi=\xi_1-\xi_2+\xi_3} i(\xi_3-\xi_2) \ind_{\{\xi_3-\xi_2>0\}} \ft u(\xi_1) \cj{\ft u(\xi_2)} \ft u(\xi_3) d\xi_1 d\xi_2. 
\label{CCMnonlin}
\end{align}
If we assume that $\P_{+}u=u$, then the projection $\P_{+}$, the fact that $\xi_1>0$ and the hyperplane condition $\xi=\xi_1-\xi_2+\xi_3$ imply that $\xi>0$, showing that the nonlinearity in \eqref{CCM} preserves the Hardy space assumption. 
We point out that whilst both focusing and defocusing CCM \eqref{CCM} preserve the Hardy space, only the focusing variety was derived in the context of solutions in $L^2_{+}(\R)$ in \cite{Abanov}; Pelinovky's derivation of CCM in the defocusing case  \cite{Pel1} does not impose a Hardy space assumption.
More remarkably, CCM on $L^2_{+}(\R)$ is completely integrable, and a Lax pair was observed by G\'erard-Lenzmann~\cite{GL}.
In their influential paper, they established the local well-posedness in $H^s_+(\R)$ for $s>\frac12$ (see \eqref{hardy}), investigated the global well-posedness and construction of multi-solitons, and demonstrated that the focusing CCM displays turbulent behaviour. 
In particular, they observed that there is a threshold for global solutions to the focusing CCM~\eqref{CCM} in the Hardy space due to the unique static solution
\begin{align}
u(t,x) = \mathcal{R}(x) = \tfrac{\sqrt{2}}{x+i} \in H^{1}_{+}(\R) \label{static}
\end{align} 
which has mass $M(\mathcal{R})=\int_{\R}|\mathcal{R}|^2 dx = 2\pi$. Here $\mathcal{R}$ is the unique (up to symmetries) ground state for the energy $ \int_{\R} | \dx u  -i\P_{+}(|u|^2)u|^2 dx.$
The value $M(\mathcal{R})$ is then the discriminator for when the conservation laws fail to control the $H^s$-norms of solutions. For $s\geq 1$, \cite{GL} showed that solutions exists globally-in-time if $M(u_0)\leq M(\mathcal{R})$, with an emphasis on the equality case here.
This threshold has since been shown to be sharp \cite{HoganKowalski,KKK1}.
We mention that the inverse scattering theory for CCM in the defocusing \cite{Matsuno6} and focusing \cite{FrankRead} cases has also been studied, as well as for defocusing INLS \cite{Pel3}.
Much attention has also been given to special solutions of these models, including the defocusing \cite{Rana2,Matsuno1, Matsuno2, Matsuno3, Matsuno4, Matsuno5, Matsuno7} and focusing \cite{Rana2,Matsuno8} CCM equations, and the defocusing \cite{Pel1} and focusing \cite{Tutiya} INLS equations.

As for low-regularity well-posedness, 
the third author with Killip and Vi\c{s}an~\cite{KLV} proved global well-posedness of CCM \eqref{CCM} in the scaling critical space $L^2_{+}(\R)$ for the defocusing version of \eqref{CCM} and for the focusing version below $M(\mathcal{R})$.
Their argument crucially exploits the complete integrability of \eqref{CCM} and an explicit formula for solutions in the Hardy space akin to that first discovered by G\'{e}rard~\cite{Gerard} for the Benjamin-Ono equation (BO).
Other interesting recent developments include investigating the zero dispersion limit of CCM \cite{Rana3}  and well-posedness for the defocusing CCM on a constant background \cite{Chen}. 

 Following the construction of finite-time blow-up solutions in~\cite{KKK1}, Kim-Kwon \cite{KKK2} established a long-time resolution result for generic $H^1(\R)$ solutions with no Hardy space assumption.  Their result provides a list that contains all possible asymptotic behaviours for both global and finite-time blow-up solutions. See also \cite{JeongKim} for a construction of blow-up solutions for focusing CCM with a different rates.

On the circle $\T$, CCM~\eqref{CCM} has also been recently studied.
The global well-posedness  in $L^{2}_{+}(\T)$ was proved by Badreddine \cite{Rana1} for small (large) data in the (de)focusing setting. The method is based on an explicit formula for solutions in the Hardy space coming from the completely integrable structure.
In the companion paper \cite{CFL1}, we consider the circle setting for INLS \eqref{INLS} and prove local well-posedness in $H^{s}(\T)$ for any $s>\frac 12$,  establish the infinite depth limit $h\to \infty$, and show unconditional uniqueness in the energy space~$H^{1}(\T)$.

\subsection{Main results}

Our goal in this paper is to make the first steps in investigating the low-regularity well-posedness for INLS \eqref{INLS} beyond the $H^{\frac 12}(\R)$ result of de Moura-Pilod~\cite{PMP} and without relying on complete integrability. 
We make the first step towards bridging the gap in the well-posedness theory for \eqref{INLS}, towards the scaling critical space $L^2(\R)$, where the best previous result is only known for CCM \eqref{CCM} in the Hardy space \cite{KLV}.
We point out that an approach exploiting complete integrability appears to lose its effectiveness in this setting. Indeed, the conservation laws lose their coercivity outside of the Hardy space irrespective of the defocusing/focusing nature of the nonlinearity! 
For example, the momentum 
\begin{align}
\textstyle{P(u)  = \int_{\R} iu \dx \cj{u} +\frac{\be}{2}\int_{\R} |u|^4  dx}
 \label{momentum}
\end{align}
 is a conserved quantity for INLS~\eqref{INLS}. For CCM in $L^2_{+}(\R)$, $P(u)$ is coercive since $\int_{\R} iu\dx \cj{u} dx = \| u\|_{\dot{H}^{1/2}}^2$. 
 However, outside of the Hardy space or for INLS, this is no longer true. We instead take a Fourier analytic approach, which has the additional advantage of being applicable to perturbations of CCM and INLS, such as \eqref{INLS} with $\g\neq 0$.
We now state our first main~result.
 
 \begin{theorem}\label{THM:main}
 Let $s>\frac 14$. Then, for any $0<h\leq \infty$ and $\g,\be\in \R$, \eqref{INLS} is locally well-posed in $H^{s}(\R)$.
 More precisely, for any $u_0\in H^s(\R)$, there exist $0<\dl\ll 1$ and $T=T(\|u_0\|_{H^s})>0$ and a unique solution $u$ to \eqref{INLS} in the space 
 \begin{align}
 C([0,T];H^{s}(\R))\cap L^{4}_{T}W^{s,4}_{x}\cap ( X^{s-\frac 14-11\dl, \frac 12}_{T}+X^{s-1,1}_{T}),
 \label{uspace}
\end{align}
satisfying $u(0)=u_0$ and 
\begin{align}
v= \Pbhip( e^{i \be F[u]}u) \in X^{s,\frac 12+\dl}_{T} \quad \text{and} \quad \P_{-,\textup{hi}}u \in X^{s,\frac 12+\dl}_{T}, \label{vspace}
\end{align}
  where $F$ is a primitive of $|u|^2$ defined in \eqref{F}, the spaces $X^{s,b}_{T}$ are the usual Fourier restriction norm spaces defined in \eqref{Xsb}, and $\P_{\pm, \hi} = \P_{\pm}\P_\hi$ are as in \eqref{projs}.

 \end{theorem}

 Theorem~\ref{THM:main} extends the known local well-posedness of INLS \eqref{INLS} to any $s>\frac 14$. We find this somewhat surprising in view of (a) our Fourier analytic method of proof of Theorem~\ref{THM:main} and (b) the similarity of \eqref{INLS} to other physically relevant nonlinear Schr\"{o}dinger equations with a cubic derivative nonlinearity whose well-posedness study has been stubbornly restricted to $H^{\frac 12}(\R)$ and above when using Fourier analytic methods. Firstly, 
 Ozawa-Tsutsumi \cite{OT} studied the Cauchy problem for the following derivative NLS equation\footnote{The linear operator in \cite{OT} is $i\dt +\dx^2$. We have transformed the equation via $u\to \cj{u}$ for easier comparison with \eqref{INLS}.}:
 \begin{align}
\dt u +i\dx^2 u = 2 u\dx(|u|^2). \label{OTNLS}
\end{align}
They introduced the gauge transform:
\begin{align}
v (x) = e^{i \int_{-\infty}^{x} |u(y)|^2dy} u(x) \label{DNLSgauge}
\end{align}
which takes a solution $u$ to \eqref{OTNLS} to the gauged function $v$, which solves
\begin{align}
\dt v+ i\dx^2 v =  i|v|^4 v. \label{gDNLS}
\end{align}
Thus, the gauge has completely ameliorated the derivative nonlinear term! 
They then established local well-posedness for \eqref{gDNLS} and hence \eqref{OTNLS} in $H^{\frac 12}(\R)$. Moreover, they also observed that the bilinear form $\dx( \cj{f}g)$
has a certain null structure related to the bilinear Strichartz estimate:
\begin{align}
\big\| |\dx|^{\frac 12} ( \cj{e^{i t\dx^2} \phi} \cdot  e^{it \dx^2} \psi)\big\|_{L^{2}_{t,x}(\R\times \R)}\les \|\phi\|_{L^2_x} \|\psi\|_{L^2_x}. \label{bilinIntro}
\end{align}
\noi 
 This observation will be crucial to us for the case of \eqref{INLS} as we discuss in the next section. Later, Takaoka \cite{Takaoka} considered the derivative nonlinear Schr\"{o}dinger equation:
 \begin{align}
\dt u + i\dx^2 u  = 2\dx(| u|^2 u), \label{DNLS}
\end{align}
which is $L^2$-critical.
By distributing the derivative, one sees that \eqref{DNLS} has two kinds of nonlinear terms: $u^{2}\dx \cj{u}$ and  $|u|^2 \dx u$. As described in \cite{Takaoka}, the Fourier restriction norm method can handle the former term but is inapplicable for the latter one. Thus, by employing the gauge transform of \cite{OT}, the latter term can be removed and local well-posedness could be established for \eqref{DNLS} in $H^{\frac 12}(\R)$ using a contraction mapping argument. This barrier of regularity stood for over 20 years until the complete integrability of \eqref{DNLS} was used in \cite{HGKNV} to prove the global well-posedness in the critical space $L^2(\R)$.
We refer to Remark~\ref{RMK:DNLS} for a further comparison between \eqref{INLS} and our result in Theorem~\ref{THM:main} and \eqref{DNLS}. 

Lastly, we point out the connection of \eqref{INLS} to another interesting cubic derivative-type nonlinear Schr\"{o}dinger equation, which is the Kinetic DNLS (KDNLS):
\begin{align}
\dt u + i\dx^2 u =  \dx(|u|^2 u) - \be \dx[ \H(|u|^2)u], \label{KDNLS}
\end{align}
where $\be \in \R$. Note that \eqref{KDNLS} reduces to \eqref{DNLS} if $\be=0$. 
The nonlinearity in KDNLS looks very similar to that in \eqref{CCM0}, but there is a crucial difference: for \eqref{KDNLS}, the parameter $\be$ is real-valued, while \eqref{CCM0}  is \eqref{KDNLS} with $\be=-i$. Indeed, when $\be>0$, \eqref{KDNLS} has a dissipative structure as the $L^2$-norm of solutions is decreasing. This dissipative structure can be exploited to study \eqref{KDNLS} even below $H^{\frac 12}$. In particular, the effect is stronger on $\T$, where global well-posedness holds in $H^{s}(\T)$ for any $s>\frac 14$ \cite{Kishi-1, Kishi0}. On the line, it is more difficult to exploit the dissipativity. We refer to \cite{Kishi1} for a priori bounds in $H^{s}(\R)$ for any $s>\frac 14$ and asymptotic behaviour of solutions under a decay condition~\cite{Kishi2, Kishi3}.

 We point out that the space \eqref{uspace} where the solution lives in Theorem~\ref{THM:main} simplifies if we additionally assume that $u_0\in L^2_{+}(\R)$ when $h=\infty$ and $\g=0$.
 Moreover, the second property in \eqref{vspace} becomes trivial since $\P_{-,\text{hi}}u=0$. We discuss this further in the Subsection~\ref{SUBSEC:H12}.

We also discover a Lax pair for INLS \eqref{INLS} on $\R$ for $\g=0$ (see \eqref{Lax}), which differs from that in~\cite{Pel3} and appears to be new in the literature.
When $h=\infty$ and $u\in L^2_+$, our Lax pair reduces to the one that has been used for CCM in \cite{Rana1,Rana2,Rana3,FrankRead,GL,KLV}.
However, outside of the Hardy space, we are only aware of the papers \cite{ChenPel, KKK2} which provide Lax pairs for CCM.
In order to make sense of our Lax operator outside of the Hardy space, we need to assume that the potentials belong to $L^4(\R)$, which by Sobolev embedding leads to the restriction of potentials in $H^{\frac 14}(\R)$. 
We use the Lax pair to establish low regularity a priori bounds in a manner inspired by \cite{KLV}.
In order to close a bootstrap argument, we will need to assume that the initial data has small $L^2$-norm. We then can then globalise our solutions under such an assumption of small $L^2$-norm.

 \begin{theorem}\label{THM:GWP}
 Let $\frac 14<s\leq 1$, $0<h\leq \infty$, $\be\in \R$ and $\g=0$. Then, there exists $r>0$ such that for any $A\geq 0$, there exists $B>0$ so that all (global) $H^{\infty}(\R)$-solutions to \eqref{INLS} satisfy 
 \begin{align}
\|u(0)\|_{L^2}\leq r \quad \textup{and} \quad \|u(0)\|_{H^{s}}\leq A \quad \LRA \quad \sup_{t\in \R} \|u(t)\|_{H^s}\leq B, \label{apriori}
\end{align}
where $r$ and $B$ can be chosen uniformly in $1\leq h\leq \infty$.
Consequently, \eqref{INLS} is globally well-posed in $H^{s}(\R)\cap B_{r}(0)$, where $B_{r}(0)=\{ u\in L^2(\R)\, :\, \|u\|_{L^2}< r\}$. 
 \end{theorem}
 
 We point that the small $L^2$-norm assumption cannot be removed in the focusing case for CCM~\eqref{CCM}. Indeed, by applying the pseudo-conformal transformation to the static solution $\mathcal{R}$ in \eqref{static}, one obtains explicit blow-up solutions
 \begin{align*}
u_{\text{sing}}(t,x) = t^{-\frac 12}e^{\tfrac{ix^2}{4t}} \mathcal{R}(\tfrac{x}{t})
\end{align*}
for all $t>0$. Whilst the profile $\mathcal{R}\in L^2_+ (\R)$, $u_{\text{sing}}(t)\notin L^2_{+}(\R)$. Moreover, it can be shown that $u_{\text{sing}}(t)\in H^{s}(\R)$ for any $0\leq s<\frac 12$ with $M(\mathcal{R})=2\pi$. See also \cite{KKK2} for blow-up solutions with $M(\mathcal{R})>2\pi$ for focusing CCM in the Hardy space.
Above $H^1(\R)$, global well-posedness of INLS \eqref{INLS} (even with $\g\neq 0$) with large $L^2$-data in the defocusing case follows from conservation of the energy for a suitably gauged version of \eqref{INLS}. See \cite{demoura1}.

In Section~\ref{SEC:GWP}, we provide a further application of our Lax pair by proving that the ``polynomial" conservation laws for INLS \eqref{INLS} (with $\g=0$) converge to those for CCM~\eqref{CCM}, when $h\to\infty$. See Proposition~\ref{PROP:poly}.

By combining Theorem~\ref{THM:GWP} with the local-in-time convergence of solutions as $h\to\infty$ (even when $\g\neq 0$), which follows essentially as a consequence of our proof of Theorem~\ref{THM:main}, we obtain our third result on the global-in-time infinite depth limit. 

\begin{theorem} \label{THM:limit}
Fix $s>\frac 14$, $\be\in \R$, $\g=0$, and $r>0$ as in Theorem~\ref{THM:GWP}. Let $u_0\in H^{s}(\R)\cap B_{r}(0)$ and $\{u_{0,h}\}_{1 \le h <\infty} \subset H^s(\R) \cap B_r(0)$ a net such that $u_{0,h} \to u_0$ in $H^s(\R)$ as $h\to\infty$. Then, let $u_{\infty}$ and $u_h$ denote the global solutions to \eqref{CCM} and \eqref{INLS}, respectively, with $u_{\infty}\vert_{t=0}=u_0$ and $u_{h} \vert_{t=0} = u_{0,h}$ constructed in Theorem~\ref{THM:GWP}. Then, $u_{h}$ converges to $u_{\infty}$ as $h\to \infty$ in $C(\R;H^{s}(\R))$.
\end{theorem}

 A similar convergence result holds locally-in-time in the case of large $L^2$-norm. We opted to provide a global-in-time convergence statement in Theorem~\ref{THM:limit} for simplicity.

 \subsection{Going beyond \texorpdfstring{$H^{\frac12}$}{H\textonehalf}}
 \label{SUBSEC:H12}

We now discuss the ideas behind the proof of Theorem~\ref{THM:main}. We reiterate that the overall source of difficulty and the new ideas used to overcome this are due to the low-regularity below $H^{\frac 12}(\R)$.

The first step is to view INLS \eqref{INLS} as a perturbation of CCM \eqref{CCM}. 
By defining the operator
\begin{align}
\GG_h : = (\H- \mathcal{T}_{h})\dx. \label{Lh}
\end{align}
 we see that \eqref{INLS} becomes
\begin{align}
\dt u+i\dx^2 u = 2\be u \P_{+}\dx(|u|^2)-i\be u \GG_{h}(|u|^2)+i\g |u|^2 u. \label{INLS22}
\end{align}
To justify this rewriting, we note  that there is a strong smoothing effect in the difference $\H-\TT_{h}$ operator so that $\GG_{h}$ is $L^p(\R)$-bounded for any $1<p<\infty$. See Lemma~\ref{LEM:GGh}. For INLS, this observation goes back to \cite{PMP}. More generally, the regularising mapping properties of $\GG_{h}$ have been a key ingredient in the recent progress on the well-posedness for the intermediate long wave equation $\dt u +\TT_{h}\dx^2 u = \dx(u^2)$
 by viewing it as a perturbation of the BO equation:
 \begin{align}
\dt u  + \H \dx^2 u = \dx(u^2). 
\notag
\end{align}
See \cite{IS}, followed by \cite{Gli, CLOP, CFLOP-2, FLZ, GL}.

As our method relies on a Fourier analytic approach and not on complete integrability, we need to perform a gauge transformation to ameliorate the bad interactions in the nonlinear term $u \P_{+}\dx(|u|^2)$.  As discussed further in \cite{PMP}, the gauge transform for DNLS in \eqref{DNLSgauge} is not helpful here. Thus, following \cite{PMP}, we use a frequency localised version of \eqref{DNLSgauge}  given by the new variable $v$ in \eqref{vspace}. Frequency localised versions of these gauge transformations originally go back to the work of Tao~\cite{TAO04} for BO. 
As shown in Lemma~\ref{LEM:gaugeeqns}, $v$ satisfies
\begin{align}
\dt v + i\dx^2 v= -2\be \Pbhip[ v \P_{-}\dx(|u|^2)] + \text{l.o.t}. \label{veqintro}
\end{align}
 Like the case of BO and unlike the case of DNLS \eqref{DNLS}, the gauge equation \eqref{veqintro} is not closed; it still depends on $u$ and the derivative has not been completely removed.
Compare \eqref{veqintro} with the gauged DNLS equation in \eqref{gDNLS}.
  This means that one needs to run a bootstrap type argument (rather than a contraction mapping argument) to construct solutions, juggling both $u$ and $v$ at the same time, and in the case of INLS \eqref{INLS22}, also the third variable $w :=\P_{-,\text{hi}}u$. In this regard, we follow a similar overall strategy as in \cite{MP} for BO with some additional inputs from \cite{GLM}. 

 The gain in \eqref{veqintro} 
 comes from having a nonlinearity with a kind of null structure of the~type 
 \begin{align}
\P_{\pm }[ \P_{\pm}f \cdot \P_{\mp}\dx g] \label{signs}
\end{align}
 which tames the derivative to an extent:
the first input function $f$ has a higher frequency than the second input function $g$ and can thus always help to control the derivative $\dx$. In the case of \eqref{veqintro}, this means that, in the \textit{worst} case, $v$ and only one of the functions $\cj{u}$ or $u$ have high enough frequencies, which can be used to control the derivative. In the absence of any \textit{additional} smoothing  and regardless of the use of any Strichartz type $L^{p}_{t,x}$-space,
this means that $v$ and $u/\cj{u}$ must use their derivatives to control $\dx$. If they can take $s$-derivatives each, this forces $2s\geq 1$, i.e., $s\geq 1/2$, which roughly shows where the $H^{\frac 12}(\R)$ barrier appears. Notice that having the Hardy space assumption for the terms $\cj{u}/u$ is useless to remove these High-High-Low interactions.

One of the main new ingredients in this paper is the use of a bilinear Strichartz estimate as in \eqref{bilinIntro}, more precisely Lemma~\ref{LEM:bilin}, in order to gain additional smoothing to weaken the effect of the derivative. In theory, this gives a $1/2$ derivative gain which changes the above numerology related to \eqref{signs} to $2s\geq 1/2$, i.e., $s\geq 1/4$. In practice, we are (currently) limited by only being able to use this once with the gauged variable $v$, and it is not available for the functions $u$ and $\cj{u}$. Let us explain this further. One of the main steps in the approach of \cite{MP} (and in our case) is a  trilinear estimate for the nonlinear term $\Pbhip[ v \P_{-}\dx(|u|^2)] $ in \eqref{veqintro} in the setting of the Fourier restriction norm spaces $X^{s,\frac 12+}$ of Bourgain~\cite{BO93}; see Section~\ref{SEC:tri}. 
Here, the above heuristics about \eqref{signs}  are not quite indicative of reality since we also have to control the derivative $\jb{\dx}^{s}$ from the norm as well as the $\dx$ in the nonlinearity. Eventually though, one will have to rely on the additional smoothing coming from the multilinear dispersion through the phase function $\Phi(\cj{\xi})$ and the relation \eqref{mods}. This is standard fare in the Fourier restriction norm method, although we mention that the phase function here is not strongly non-resonant as is the case for the  gauged BO equation, so we rely on a careful case separation to avoid the resonant set.

The moral in the above discussion is that we need to understand $X^{s,b}$-information for the solutions $u$ to \eqref{INLS22} and \eqref{CCM} whose nonlinearities cannot be (entirely) controlled in $X_{T}^{s,\frac 12+}$. In \cite{MP}, Molinet-Pilod show that smooth solutions to BO belong to $X_{T}^{s-1,1}$, which by interpolation with $X_{T}^{s,0}$, gives control on, say, $X_{T}^{s-\frac 12-,\frac 12+}$. 
This $X_{T}^{s-1,1}$-property is proved by showing that the nonlinear term $\dx(u^2)$
belongs to $L^{2}_{T}H^{s-1}_x$. The choice $s-1$ is the minimum amount of smoothing needed to control the derivative $\dx$. For the bilinear estimates for BO, whilst one loses a lot of spatial regularity in placing $u$ into $X^{s-1,1}_{T}$, one only does this when $u$ has a large modulation which allows to counteract the loss in spatial regularity by a full gain of the phase function.

Here, we arrive at the first issue in using these methods to study CCM \eqref{CCM} and INLS \eqref{INLS}: these equations cannot be put into a conservative form:
the derivative is embedded into the nonlinearity. For CCM \eqref{CCM} with the Hardy space assumption, we see from \eqref{CCMnonlin} that the signs $\P_{+}$ on the derivative $\xi_3-\xi_2$ and, here is the crucial part, on the first function with frequency $\xi_1$, imply that 
\begin{align}
|\xi_1| + |\xi_3-\xi_2| = \xi = |\xi|, \label{CCMsigns1}
\end{align}
so that the output frequency can be used to control the derivative in the nonlinearity and measuring the nonlinearity in $L^2_{T}H^{s-1}_x$ is effective to show that $u\in X_{T}^{s-1,1}$. This happens in Lemma~\ref{LEM:CCMXsb}.

However, 
this discussion relies heavily on the Hardy space assumption to ensure \eqref{CCMsigns1}. We can no longer make this assumption when studying INLS \eqref{INLS22}. Consequently, whenever the first input function has a very high \textit{negative} frequency, \eqref{CCMsigns1} becomes $\xi + |\xi_1|=|\xi_3-\xi_2|$
and there are now dangerous interactions of the form 
\begin{align}
|\xi_1|\sim |\xi_3-\xi_2|\gg |\xi|. \label{CCMsigns2}
\end{align}
These are fatal to us below $H^{\frac 12}(\R)$ as we would have to rely on the smoothing from at most two input functions forcing $s\geq 1/2$ and no amount of weakening the topology to say $L^{2}_{T}H^{s-100}_{x}$ is going to help. Our way out of this issue is to decompose $u$ into two-parts: 
\begin{align}
u=u_{g}+u_{b}. \label{udecomp}
\end{align}
The part $u_{g}$ precisely avoids the interactions \eqref{CCMsigns2} and has the desired $X^{s-1,1}_{T}$ regularity. Conversely, $u_{b}$ contains these bad interactions in \eqref{CCMsigns2} and we instead control it in $X^{s-\frac 14-,\frac 12+}_{T}$. 
This allows us to use duality and the bilinear Strichartz estimate to recover about half of the derivative and consequently control $u_{b}$ even when $1/4 < s\leq 1/2$. Moreover, whilst this term has worse temporal regularity, we have gained\footnote{The gain is actually better at $\frac 12-s-$ but it is enough for us to use the least amount which is $\frac 12-\frac 14=\frac 14$.} about $1/4$ in spatial smoothing! We detail the precise version of the decomposition \eqref{udecomp} in Subsection~\ref{SEC:INLSdecomp}. The extra spatial smoothing counteracts the worse phase gain when we establish the trilinear estimate (Proposition~\ref{PROP:tri2}) in the setting of INLS \eqref{INLS22}. 

We find it interesting that we need extra smoothing from the bilinear Strichartz estimate for proving both the trilinear estimates \textit{and} in establishing the $X^{s,b}_{T}$-property for smooth solutions to the original equation.
We do not believe that the regularity restriction $s>1/4$ in Theorem~\ref{THM:main} is sharp.  It seems possible that by implementing additional ideas such as local smoothing estimates and finer  decompositions into our approach, the threshold may be lowered to some $0<s_0\leq 1/4$. However, such an approach then requires corresponding maximal function estimates, which are difficult to establish below $H^{\frac 12}(\R)$, and it is not currently clear to us if the full sub-critical range could be covered this way.
For the sake of global well-posedness, it is interesting that the threshold $1/4$ also appears in making sense of the Lax operator. See also Remark~\ref{RMK:DNLS} for another instance of this numerology.

The remaining of the paper is organised as follows.
In Section~\ref{SEC:Prelim}, we introduce relevant notation, and review properties of the operators $\TT_h$ and $\GG_h$ in \eqref{tilbert} and \eqref{Lh}, and related quantities. 
Section~\ref{SEC:main} introduces the gauged variables $v, w$ in \eqref{gauge}, establishes the regularity properties of $u,v,w$ for both CCM \eqref{CCM} and INLS \eqref{INLS}, as well as the finer decomposition needed for the latter (Subsection~\ref{SEC:INLSdecomp}). 
The crucial trilinear estimates to handle the nonlinearity are proven in Section~\ref{SEC:tri}. 
In Section~\ref{SEC:LWP} we present the proof of Theorem~\ref{THM:main} and Theorem~\ref{THM:limit} on well-posedness of INLS and CCM, and convergence in the limit $h\to\infty$. Lastly, Section~\ref{SEC:GWP} presents the new Lax pair for INLS \eqref{INLS}, which is used to obtain the long-time bounds and global well-posedness in Theorem~\ref{THM:GWP}, as well as study of convergence of the family of polynomial conservation laws of INLS \eqref{INLS} in the infinite-depth limit.

We conclude this section with some additional remarks.

 \begin{remark}\rm
Our result in Theorem~\ref{THM:main} also extends to suitable perturbations of \eqref{CCM}, such as for the following equation:
\begin{align}
 \dt u+i\dx^2 u = 2\be u\P_{+}\dx(|u|^2)+i\g \P_{+}( |u|^2 u) 
 \notag
\end{align}
where $\g\in \R$.  This equation preserves the Hardy space assumption for any $\g\in \R$, and when $\g\neq 0$, we do not know if it is completely integrable. Notice that if  $\be=0$, we get a dispersive version of the Szeg\"{o} equation \cite{GG}.  
\end{remark}

\begin{remark}\rm \label{RMK:DNLS}
In the following discussion, we compare at a more technical level, the good nonlinearity $v^2 \dx \cj{v}$ in the gauged version of DNLS \eqref{DNLS} and the derivative nonlinearity in the gauged equation \eqref{veqintro} for INLS \eqref{INLS22}. 
Consider the spatial multiplier 
that naturally appears in the $X^{s,b}$-analysis of  $v^2 \dx \cj{v}$:
\begin{align}\label{multiDNLS}
\frac{ \jb{\xi}^{s} |\xi_2|}{\jb{\xi_1}^s \jb{\xi_2}^{s} \jb{\xi_3}^s} \frac{1}{\jb{ (\xi_3-\xi_2)(\xi_1-\xi_2)}^{\frac 12}},
\end{align}
where the second factor is the gain of half of the phase function morally coming from the Fourier restriction norm method. In the case when $|\xi|\sim|\xi_1|\sim |\xi_2|\sim|\xi_3|$, the phase function is not helpful to control the first factor, which imposes the condition $s\geq \frac 12$ to control the numerator.
Now consider the nonlinear term $v\P_{-}\dx(|v|^2)$.
The analogue of Fourier multiplier in \eqref{multiDNLS} for this term is:
\begin{align}
\frac{ \jb{\xi}^{s} |\xi_3-\xi_2|}{\jb{\xi_1}^s \jb{\xi_2}^{s} \jb{\xi_3}^s} \frac{1}{\jb{ (\xi_3-\xi_2)(\xi_1-\xi_2)}^{\frac 12}}.
\label{multINLS}
\end{align}
Assuming that $|\xi_1-\xi_2|\ges 1$ and that the derivative is large, then \eqref{multINLS} simplifies to
 \begin{align}
\frac{ \jb{\xi}^{s} |\xi_3-\xi_2|^{\frac 12}}{\jb{\xi_1}^s \jb{\xi_2}^{s} \jb{\xi_3}^s} \frac{1}{\jb{ \xi_1-\xi_2}^{\frac 12}}  .
\notag
\end{align}
In the same nearly resonant situation where all frequencies are similar, we see that the first factor is now controlled precisely as soon as $s\geq \frac 14$. This heuristically explains the numerology in Theorem~\ref{THM:main}.
However, in practice, this analogy is not accurate
as the nonlinearity on the gauged side for INLS \eqref{INLS22} is actually $\Pbhip[ v \P_{-}\dx(|u|^2)]$ and we do not have the same $X^{s,b}$-properties for $v$ and the original solution $u$. Indeed, we only know that $u\in X^{s-\frac 12, \frac 12}$ and the above numerology does not work. 
Consequently, we cannot always gain derivatives through the phase function, and this is where the bilinear Strichartz estimate is helpful.  
\end{remark}

 \begin{remark}
\rm
Our proof of Theorem~\ref{THM:main} does not rely on the specific form of the linear operator $\GG_{h}$ in \eqref{Lh}. In particular, we only need $\GG_{h}$ to be $L^p(\R)\to L^p(\R)$ bounded for $2\leq p\leq 4$ and to preserve real-valuedness: if $f$ is real-valued, then $\GG_{h}f$ is also real-valued.
However, at this point, we need the specific form of  $\TT_{h}$ in \eqref{tilbert} to write a Lax pair for \eqref{INLS}, namely the fact that it satisfies a Cotlar-type identity; see \eqref{Tilb1}. It seems possible to simply use the Lax pair for CCM (obtained by putting $h=\infty$ in \eqref{Lax}) and to use a Gronwall argument to handle the perturbative term in \eqref{INLS2}, similar to \cite{CFLOP-2}. Such a bound would likely not be uniform in time, unlike \eqref{apriori}.
To obtain the convergence result in Theorem~\ref{THM:limit}, we also specialise $\GG_{h}$ as in \eqref{Lh}. 
 \end{remark}

\section{Preliminaries}\label{SEC:Prelim}

\subsection{Notation}

In this subsection, we introduce relevant notation, projections, and function spaces, which will be used throughout.

We use $A\les B$ to denote $A\leq C B$ for some constant $C>0$,
$A\ll B$ if there is a small $c>0$ such that $A\le cB$, and $A\sim B$ if both $A\les B$ and $B\les A$ hold. 
The notation $a-$ refers to $a-\eps$ for any $\eps>0$. 
Also, $a\land b$ and $a\lor b$ denote the minimum and the maximum between $a$ and $b$, respectively.

Given a function $f$ on $\R$, we use $\F f$ and $\ft f$ to denote its Fourier transform
\begin{align}
\ft f(\xi) = \frac{1}{\sqrt{2\pi}} \int_\R f(x) e^{-i\xi x} dx. 
    \notag
\end{align}
For space-time functions $u:\R\times\R \to \C$, we may use the notation $\F_t u$ and $\F_x u$ to indicate the Fourier transform with respect to the time and space variables. We omit this indexing, when clear from context.

Let $s\in \R$ and $1 \le p \le \infty$. We define the $L^p$-based Sobolev spaces $W^{s, p}(\R)$ by the norm:
\begin{align*}
    \| f\|_{W^{s, p}} 
    = \| J^s f \|_{L^p}
    = \big\| \Ft^{-1} \big( \jb{\xi}^s \ft f(\xi) \big) \big\|_{L^p}, 
\end{align*}
where $J^s$ denotes the Bessel potential with Fourier multiplier $\jb{\xi}^s$, where $\jb{x} = (1+|x|^2)$ and $\Ft^{-1}$ stands for the inverse Fourier transform. 
We also use $\dot{W}^{s,p}(\R)$ for the homogeneous Sobolev spaces with norm
\begin{align*}
\| f\|_{\dot{W}^{s, p}} 
    = \| D^s f \|_{L^p}
    = \big\| \Ft^{-1} \big( |\xi|^s \ft f(\xi) \big) \big\|_{L^p}
    ,
\end{align*}
where $D^s$ is the Riesz potential, with Fourier multiplier $|\xi|^s$.
When $p=2$, we write $W^{s,2}(\R) = H^s(\R)$ for the $L^2$-based Sobolev spaces, with norm
\begin{align*}
    \| f \|_{H^s} = \| \jb{\xi}^s \ft f (\xi) \|_{L^2_\xi}. 
\end{align*}

We define the inner product on $L^2(\R)$ by 
\begin{align*}
\jb{f,g} = \int_{\R} \cj{f} g dx, 
\end{align*}
and the dyadic $L^p$-spaces, for $1<p<\infty$, via the norm
\begin{align}
\| f\|_{ \wt{L^p_{t,x}}} = \bigg( \sum_{N\in \Z} \| \P_{N}f\|_{L^p_{t,x}}^{2} \bigg)^{\frac 12}.
\notag
\end{align}
By the Littlewood-Paley square function theorem and Minkowski's inequality, it holds that 
\begin{align}\label{Lpwt}
\|f\|_{L^{p}_{t,x}} \les \|f\|_{\wt{L^{p}_{t,x}}}
\end{align}
for any $2\leq p<\infty$.

Lastly, 
given $1\leq p\leq \infty$ and an operator $R:L^p(\R) \to L^{p}(\R)$, we use $\|R\|_{L^p \to L^p}$ to denote its operator norm on $L^p$.
When $p=2$, we will use the shorthand notation $\|R\|_{\text{op}}$ to denote the $L^2\to L^2$ operator norm of $R$. 
When working with space-time functions, given $T>0$, we often use the shorthand notation $L^p_T W^{s, q}_x$ for $L^p([0,T]; W^{s,q}(\R))$ and $L^p_T L^q_x$ for $L^p([0,T]; L^q(\R))$.

In the analysis of CCM \eqref{CCM}, we also use the Hardy-Sobolev space $H^s_+(\R)$, defined as 
\begin{align}\label{hardy}
    H^s_+ (\R) 
    =
    \big\{
    f \in H^s(\R) : \ \supp \ft f \subset [0,\infty)
    \big\} , 
\end{align}
and the Hardy space $L^2_+(\R)$ when $s=0$.

We now introduce notation to perform Litlewood-Paley decompositions. 
Let $\eta:\R\to [0,1]$ be a smooth function supported on $[-2,2]$ and equal to $1$ on $[-1,1]$.
Given $N\in 2^{\Z}$, let $\eta_{N}(\xi)=\eta(\frac{\xi}{N})$ and $\psi_N(\xi)=\eta(\frac{\xi}{N})-\eta(\frac{2\xi}{N})$. 
Note that
\begin{align*}
\sum_{N\geq 1} \psi_{N}(\xi) = 1- \eta_{\frac 12}(\xi)
\quad \text{when } \xi \in \R\setminus\{0\}
.
\end{align*}
Moreover, we use $\P_{\le N}$ and $\P_N$ to denote
 the Littlewood-Paley projectors defined by
 \begin{align*}
  \F \,(\P_{\leq N} f)  &= \eta_{N} \ft f 
, 
\\
\F(\P_1 f ) 
&= 
\eta_1  \ft f 
\quad
\text{and}
\quad
\F\, \P_N f 
= \F \,\P_{\leq N} f -\F\, 
\P_{\leq \frac{N}{2}} f 
=\psi_N \ft f, \quad \text{when } N\ge 2 ,
 \end{align*}
 and $\P_{>N}:= 1-\P_{\leq N}$. Note that
\begin{align}
\sum_{N\ge 1} \P_N f = f ,  
\notag
\end{align}
which we will often use in our estimates, where by abuse of notation, we assume to sum over dyadic numbers in $2^{\Z_{\ge0}}$. 
 We use $\wt\P_{N}$ for the wider projector with multiplier $\wt{\psi}_{N}(\xi) =\psi_{N}(\frac{\xi}{2}) + \psi_{N}(\xi) +\psi_{N}(2\xi)$.
 We then set
 \begin{equation}
\begin{alignedat}{3}
 \F\,( \P_{+} f)(\xi) &= \ind_{\xi> 0} \ft f(\xi), 
 &
 \qquad   \F\, ( \P_{-} f)(\xi)  &= \ind_{\xi< 0} \ft f(\xi),
 \\
\Pbhi & = \sum_{N\geq 2}\P_{N},  
&
\quad 
\PbHI & = \sum_{N\geq 8}\P_N, 
\\
\Pblo &= \text{Id}-\Pbhi, 
& \quad 
\PbLO &=\text{Id} - \PbHI.
\end{alignedat}
\label{projs}
\end{equation}
Moreover, we define $\P_{\ll 1} = \P_{\leq 2^{-2}}$ and $\P_{\ges 1} =\text{Id}-\P_{\ll 1}$, which satisfy $\P_{\ll 1}\Pblo =\P_{\ll 1}$. We also define the shorthand $\P_{\pm, \text{hi}}=\P_{\pm}\Pbhi$ and similarly for $\P_\HI, \P_\lo, \P_\LO$.

For space-time functions $u: \R\times \R \to \C$, we define frequency projectors on the space-time Fourier variables $(\tau, \xi)$: given $K\in 2^{\N}$, we set 
\begin{align}
\label{Qpro}
\begin{split}
\mathcal{F}_{t,x}\{ \Q_{\ll K}u\}(\tau,\xi) &= \eta_{10^{-10}K}(\tau +\xi^2)\ft u(\tau,\xi),\\
 \mathcal{F}_{t,x}\{ \Q_{\ges K}u\}(\tau,\xi) &=(1- \eta_{10^{-10}K})(\tau +\xi^2)\ft u(\tau,\xi).
\end{split}
\end{align}

\subsection{Product estimates}

In this subsection, we recall the fractional Leibniz rule and show relevant product estimates involving the function $e^{i\be \dx^{-1}|u|^2}$ which appears in the gauge transformation in Subsection~\ref{SEC:Gauge}.

We will extensively use the fractional Leibniz rule; see \cite{CW,GO,BL}.

\begin{lemma}[Fractional Leibniz rule]
\label{LEM:leib}
Let $s\ge0$ and $1<p_j,q_j,r\leq \infty$, $j=1,2$,
such that $\frac1r=\frac{1}{p_j}+\frac{1}{q_j}$. 
Then, we have
\begin{align*}
  \| J^s(fg)\|_{L^r (\R)} \les \|J^s f\|_{L^{p_1}(\R)} \|g\|_{L^{q_1}(\R)}  + \|f\|_{L^{p_2}(\R)} \| J^s  g\|_{L^{q_2}(\R)}.
\end{align*}
\end{lemma}
\noi

The gauge transformation from \cite{PMP}, which we recall in \eqref{gauge}, requires understanding the function $e^{i\be \dx^{-1}|u|^2}$. To make this precise, we define the primitive
\begin{align}
F=F[u] = \dx^{-1}(|u|^2) = \int_{-\infty}^{x} |u(t,y)|^2 dy \label{F}
\end{align}
and note that 
\begin{align*}
\dx F = |u|^2.
\end{align*}
We then consider the function $e^{i\be F}$, which satisfies $|e^{i\be F}|=1$ since $F$ is real-valued.
It follows that $e^{i\be F}\in L^{\infty}(\R)\setminus L^2(\R)$. Consequently, $e^{i\be F}$ is merely a tempered distribution. Nonetheless, as $u\in L^2(\R)$, $\dx e^{i\be F}\in L^1(\R)$, and we have that for almost every $\xi\in\R$, 
\begin{align*}
\mathcal{F}_x\{e^{i\be F}\}(\xi) = \frac{1}{i\xi} \int_{\R} e^{-ix\xi} \dx(e^{i\be F}) dx.
\end{align*}
Whilst $\Pbhi e^{i\be F}$, $\PbHI e^{i\be F}$, and $\Pbhip e^{i\be F}$ are well-defined and belong to $L^2(\R)$, due to the non-integrable singularity at the origin, the quantities $\P_{\pm}(e^{i\be F})$ are ill-defined. Moreover, we need to carefully define $\Pblo e^{i\be F}$ and $\PbLO e^{i\be F}$. Here, these are understood as 
\begin{align*}
\Pblo e^{i\be F} : = e^{i\be F} -\Pbhi e^{i\be F} \quad \text{and} \quad \PbLO e^{i\be F} : = e^{i\be F} -\PbHI e^{i\be F}
\end{align*}
and it follows that
\begin{align}
\PbHI \Pblo(e^{i\be F}) &= \PbHI e^{i\be F} - \PbHI e^{i\be F} = 0, 
\notag
\\
\dx \Pblo(e^{i\be F}) &= \Pblo \dx e^{i\be F}, \label{PloeiF}
\end{align}  where the $\Pblo$ appearing on the right-hand side of \eqref{PloeiF} agrees with an honest Littlewood-Payley projection to frequencies $\{|\xi|\les 1\}$.

Given $0<\al \leq 1$, we have that $D^{\al}_{x} e^{i\be F}$ belongs to the subspace of tempered distributions described in \cite[Definition 1.26]{BCD}. In particular, it follows from \cite[Proposition 2.14]{BCD} that we have the following equality in the sense of tempered distributions: 
\begin{align*}
\sum_{k\in \Z} \dot{\P}_k [D^{\al}_{x} e^{i\be F}] = D^{\al}_{x} e^{i\be F},
\end{align*}
where the $\{ \dot{\P}_k \}_{k}$ are homogeneous Littlewood-Payley projectors. In particular, this implies
 \begin{align}
 \begin{split}
D^{\al}_{x} \Pblo e^{i\be F} & = D^{\al}_{x} e^{i\be F} - D^{\al}_{x} \Pbhi e^{i\be F}  \\
&=\sum_{k\in \Z} \dot{\P}_k [D^{\al}_{x} e^{i\be F}] - \sum^{\infty}_{k= 2}D^{\al}_{x}\dot{\P}_k [ e^{i\be F} ]  \\
& = \sum_{k<2} \dot{\P}_k [D^{\al}_{x} e^{i\be F}] 
\end{split} \label{Dalcomp}
\end{align}
Then, 
\begin{align*}
\dot{\P}_k [D^{\al}_{x} e^{i\be F}] (x) = \int m_{k,\al}(x-y)e^{i\be F(y)}dy = 2^{k(1+\al)}\int m_{1,\al}(2^{k}(x-y))e^{i\be F(y)}dy,
\end{align*}
where $m_{1,\al} = \F^{-1} \big\{ |\cdot|^\al \psi \big\}$ and $m_{k,\al } = 2^{k(1+\al) } m_{1,\al}(2^k\cdot )$,
and thus
\begin{align}
\| D^{\al}_{x} \Pblo e^{i\be F}\|_{L^{\infty}_{x}} \leq  \sum_{k<2}  \| \dot{\P}_k [D^{\al}_{x} e^{i\be F}] \|_{L^{\infty}_{x}}  \leq \sum_{k<2} 2^{k \al} \| m_{1,\al}\|_{L_x^1} \les 1. \label{DeLinfty}
\end{align}
Similarly, we have the difference estimate
\begin{align}
\| D^{\al}_{x} \Pblo [e^{i F_1}-e^{iF_2}]\|_{L^{\infty}_{x}} \les \|e^{iF_1}-e^{iF_2}\|_{L^{\infty}_x}. \label{DeLinfty2}
\end{align}

We will need a product estimate involving products of functions with the exponential factors $e^{i\be F}$.

\begin{lemma}
Let $0\leq s\leq \frac23$. Given $F_j$, $j=1,2$,  two real-valued functions such that $\dx F_j= |f_j|^2$ with $f_j\in L^4(\R)$ and $g \in H^{s}(\R)\cap L^3(\R)$, it holds that
\begin{align}
\| J^{s} \Pbhi[ e^{iF_1}g]\|_{L^{2}_{x}} \les \|f_1\|_{L^4_x}^2 & \big\{ \|\P_{\ll 1}g\|_{L^{\infty}_x}+  \|g\|_{L_x^{3}}\big\}   +\|D^{s}\P_{\ges 1}g\|_{L_x^2} 
, 
  \label{L2F1} \\
\| J^{s} \Pbhi[ (e^{iF_1}-e^{iF_2})g]\|_{L^{2}_{x}} & \les \Big[
(\|f_1\|_{L^4_x}+\|f_2\|_{L^4_x})\|f_1-f_2\|_{L^4_x} \notag\\ 
& \hphantom{XX} +(1+\|f_2\|_{L^4_x}^2)\|F_1-F_2\|_{L^{\infty}_x} \Big]  \notag \\
& \hphantom{XX}\times (\|\P_{\ll1} g\|_{L^{\infty}_x} +\|D^{s}\P_{\ges 1}g\|_{L_x^2} +\|g\|_{L^{3}_x}). 
 \label{L2F2}
\end{align}
\end{lemma}

\begin{proof}
We begin with \eqref{L2F1}. We write $g=\P_{\ll 1} g+ \P_{\ges 1}g$ and consider the contribution coming from each of these parts. First, the projections imply $\Pbhi[ e^{iF_1} \P_{\ll 1}g]= \Pbhi[ \P_{\ges 1}(e^{iF_1})\P_{\ll 1}g]$ and thus by the fractional Leibniz rule and Bernstein's inequality,
\begin{align*}
\| J^{s} \Pbhi[ e^{iF_1} \P_{\ll1}g]\|_{L^{2}_{x}}& \sim \| D^{s} \Pbhi[ (\P_{\ges 1}e^{iF_1})\P_{\ll 1}g]\|_{L^{2}_{x}} \\
& \les \| D^{s} \P_{\ges 1} e^{iF_1}\|_{L^{2}_x} \|\P_{\ll 1} g\|_{L^{\infty}_x} + \|\P_{\ges 1}e^{iF_1}\|_{L^2} \|D^{s}\P_{\ll 1}g\|_{L^{\infty}_x} \\
& \les \| D^{s-1} \P_{\ges 1}( |f_1|^2 e^{iF_1})\|_{L^{2}_x}  \|\P_{\ll 1} g\|_{L^{\infty}_x}  + \|\P_{\ges 1}(|f_1|^2 e^{iF_1})\|_{L^2_x} \|\P_{\ll 1} g\|_{L^2_x}  \\
& \les \|f_1\|_{L^4_x}^2 \| \P_{\ll 1}g\|_{L^{\infty}_x}.
\end{align*} 
\noi
Now we consider the contribution from $\P_{\ges 1}g$. By the fractional Leibniz rule (Lemma~\ref{LEM:leib}),  and \eqref{DeLinfty},  we have 
\begin{align*}
 \| J^{s} \Pbhi[ \Pblo(e^{iF_1})\P_{\ges  1}g]\|_{L^{2}_{x}}  & \les \|\Pblo(e^{iF_1})\|_{L^{\infty}_x}\|D^s \P_{\ges 1}g\|_{L^2_x} + \|D^{s}\Pblo(e^{iF_1})\|_{L^{\infty}_x} \|\P_{\ges 1}g\|_{L^2_x} \\
 &\les \|D^s \P_{\ges1}g\|_{L^2_x}.
\end{align*}
Similarly, by fractional Leibniz, \eqref{DeLinfty}, and Sobolev's embedding,  we have
\begin{align*}
 \| J^{s} \Pbhi[ \Pbhi(e^{iF_1})\P_{\ges  1}g]\|_{L^{2}_{x}}  & \les \|\Pbhi(e^{iF_1})\|_{L^{\infty}_x}\|D^{s}\P_{\ges 1}g\|_{L^2_x} + \|D^{s}\Pbhi(e^{iF_1})\|_{L^{6}_x} \|\P_{\ges 1}g\|_{L^{3}_x}  \\
 &\les \|D^{s}\P_{\ges 1}g\|_{L^2_x} + \|D^{s-1}\Pbhi( |f_1|^2 e^{iF_1})\|_{L^{6}_x} \|g\|_{L^{3}_x} \\
  &\les \|D^{s}\P_{\ges 1}g\|_{L^2_x} + \|D^{s-\frac 23}\Pbhi( |f_1|^2 e^{iF_1})\|_{L^{2}_x} \|g\|_{L^{3}_x} \\
 & \les \|D^{s}\P_{\ges 1}g\|_{L^2_x}+\|f_1\|_{L^4_x}^2 \|g\|_{L^{3}_x}, 
\end{align*}
given that $s \le \frac23$. We obtain \eqref{L2F1} from combining the estimates above.

For \eqref{L2F2}, we apply the same argument as above but with \eqref{DeLinfty2} and the following estimate:
\begin{align*}
 \big\||f_1|^2 e^{iF_1} -|f_2|^2 e^{iF_2} \big\|_{L^{2}_x} &\les  \big\| |f_1|^2 -|f_2|^2 \big\|_{L^{2}_x} + \| f_2\|_{L^4_x}^2 \|e^{iF_1}-e^{iF_2}\|_{L^{\infty}_x} \\
 & \les (\|f_1\|_{L^4_x}+\|f_2\|_{L^4_x}) \|f_1-f_2\|_{L^4_x} + \|f_2\|_{L^4_x}^{2}  \|e^{iF_1}-e^{iF_2}\|_{L^{\infty}_x}.
\end{align*} 
By the mean value theorem, we have $\|e^{iF_1}-e^{iF_2}\|_{L^{\infty}_x}\les \|F_1-F_2\|_{L^{\infty}_x}$ and this completes the proof of \eqref{L2F2}.
\end{proof}

\subsection{The operators \texorpdfstring{$\TT_{h}$}{Th}
 and \texorpdfstring{$\GG_{h}$}{Gh}}

We recall some known facts about the operators $\TT_{h}$ in \eqref{tilbert} and $\GG_{h}$ in \eqref{Lh}.
The operator $\TT_{h}$ satisfies the following Cotlar-type identity: for sufficiently nice $f,g$, it holds that
\begin{align}
 \TT_{h}[ f\TT_{h} g +g \TT_{h}h] = \TT_{h} f \cdot \TT_{h}g- fg -M_{f}M_{g} ,
 \label{Tilb1}
\end{align}
where $M_{f}  : = \frac{1}{2h} \int_{\R} f  dx$. 

When $|\xi|\ll 1$, $\coth(\xi) \approx \frac{1}{2\xi}$, 
so the Fourier multiplier for $\TT_h$ in \eqref{TTft} behaves like an antiderivative but it is singular unless applied to functions whose Fourier transform is vanishing sufficiently fast near the origin. 
By subtracting this antiderivative term from $\TT_h$, we obtain a better behaved operator. More precisely, we consider the singular integral operator
\begin{align}
\mathcal{K}_{h}f(x) &= \frac{1}{2h} \text{p.v.} \int_{\R} \bigg[ \coth\bigg( \frac{\pi(x-y)}{2h} \bigg) - \sgn\bigg( \frac{\pi(x-y)}{2h} \bigg)\bigg] f(y)dy,   \label{Kh}
\end{align}
for which the following holds. 

\begin{lemma}\label{LEM:Kh}
Let $0<h<\infty$ and $1<p<\infty$.  Then, $\mathcal{K}_{h}$ is $L^p(\R)\to L^p(\R)$ bounded and the operator norm is uniformly bounded in $h$ for $1\leq h<\infty$.
\end{lemma}
\begin{proof}
The kernel $K_h$ of the integral operator $\mathcal{K}_h$ can be shown to be a Calder\'on--Zygmund convolution kernel (see \cite[Proposition 5.4.4]{Grafakos}) and thus the claimed $L^p(\R)$ boundedness follows. Moreover, it satisfies the scaling property
\begin{align*}
    K_{h}(x) = \tfrac 1h K_{1}(\tfrac{x}{h}) \quad \text{for all} \,\, h>0
\end{align*}
 and thus it satisfies the conditions of \cite[Proposition 5.4.4]{Grafakos}
uniformly over $1\leq h<\infty$. 
\end{proof}

For the operator $\GG_h$ in \eqref{Lh}, the presence of the derivative $\dx$ ameliorates the singularity at the origin. Thus, the symbol of $\GG_h$ acts like the identity near the origin and is exponentially decaying. By the Mikhlin-H\"{o}rmander multiplier theorem \cite[Theorem 6.2.7]{Grafakos}, we then obtain:

\begin{lemma}\label{LEM:GGh}
Let $0<h<\infty$ and $1<p<\infty$.  Then, $\mathcal{G}_{h}$ is $L^p(\R)\to L^p(\R)$ bounded and it holds that 
\begin{align}
    \| \GG_{h}\|_{L^{p}(\R)\to L^p(\R)} \les \tfrac{1}{h}
    ,
    \label{GGconv}
\end{align}
where the implicit constant is uniform in $1\leq h<\infty$.
\end{lemma}
\begin{proof}
The proof is similar to that of Lemma~\ref{LEM:Kh}. We simply note that the kernel $G_{h}$ of the integral operator $\GG_{h}$ satisfies the scaling relation:
\begin{align*}
    G_{h}(x) = \tfrac{1}{h^2}G_{1}(\tfrac{x}{h}) \quad \text{for all} \, \, h>0.
\end{align*}
The integral operator with kernel $\tfrac{1}{h^2}G_{1}(\tfrac{x}{h})$ is then a Calder\'{o}n-Zygmund operator and is $L^p(\R)\to L^p(\R)$ bounded, for $1<p<\infty$, uniformly in $1\leq h<\infty$. The extra factor of $h^{-1}$ then accounts for its appearance in \eqref{GGconv}.
\end{proof}

In INLS \eqref{INLS22}, we have two kinds of harmless cubic terms: the local one $|u|^2 u$ and the nonlocal one $u\GG_{h}(|u|^2)$. We can deal with these terms simultaneously by defining the operator
\begin{align}
\QQ_{h}: =  -i\be \GG_{h}+i\g \text{Id}. \label{Qh}
\end{align}
Then, \eqref{INLS22} can be written more succinctly as
\begin{align}
\dt u+i\dx^2 u = 2\be u \P_{+}\dx(|u|^2)+u \QQ_{h}(|u|^2). \label{INLS2}
\end{align}
It is clear that $\QQ_{h}$ has at least the same $L^{p}\to L^p$ mapping properties as $\L_{h}$.

\section{The gauge transform and properties of solutions}
\label{SEC:main}

\subsection{Fourier restriction norm spaces}

For $s,b\in \R$, we consider the Fourier restriction norm spaces $X^{s,b}(\R\times \R)$ as the completion of $\S(\R\times \R)$ under the norm \cite{BO93}:
\noi
\begin{align}
\| u\|_{X^{s, b}(\R\times \R)} &= \big\| \jb{\tau+\xi^2}^{b}\jb{\xi}^{s} \ft u(\tau, \xi)\big\|_{L^2_{\tau,\xi}}. \label{Xsb}
\end{align}
\noi
Given a time interval $I \subset \R$, we define localised in time versions of these spaces as follows: if $u: I \times \R \to \C$, then 
\begin{align}
\|u\|_{X^{s,b}_{I}}: =\inf\{ \|\wt{u}\|_{X^{s,b}} \, : \, \wt{u}:\R\times \R \to \C, \,\, \wt{u}\vert_{I\times \R} = u\}. \label{localspace}
\end{align}
When $I=[0,T]$ for some $T>0$, we use the notation $X^{s,b}_{I}=X^{s,b}_{T}$.
For any $b>\frac 12$, the following embedding holds:
\begin{align}
X^{s,b}_{T} \embeds C([0,T];H^{s}(\R)). \label{YsCTHs}
\end{align}
We recall the following linear estimates related to the Fourier restriction norm spaces. See~\cite{MP}, for example.

\begin{lemma}\label{LEM:linXsb}
Let $s, b\in \R$, $T>0$, and $\eta$ denote a smooth time-cutoff.

\noi
\textup{(i)} The following estimate holds
\begin{align}
\| \eta(t) S(t)f\|_{X^{s,b}} \les \|f\|_{H^{s}}, 
\notag
\end{align}
where $S(t)$ denotes the linear Schr\"odinger propagator $e^{-it \dx^2}$.

\noi
\textup{(ii)} Let $0<\dl<\frac 12$. Then, 
\begin{align}\label{lin2}
\bigg\| \eta(t) \int_{0}^{t} S(t-t') g(t')dt' \bigg\|_{X^{s,\frac 12+\dl}} & \les \|g\|_{X^{s, -\frac 12 +\dl}}
.
\end{align}

\noi
\textup{(iii)} Given $-\frac 12 < b'\leq b<\frac 12$, it holds that 
\begin{align}
\| u\|_{X^{s,b'}_{T}} \les T^{b-b'}\|u\|_{X^{s,b}_{T}} \quad \text{and} \quad \| \ind_{[0,T]}u\|_{X^{s,b'}} \les T^{b-b'}\|u\|_{X^{s,b}} 
.
\label{timeloc}
\end{align}
\textup{(iv)} Given $0\le \dl <\frac18$, the following estimate holds
\begin{align}
\|u\|_{L^{4}_{T,x}}& \les \|u\|_{\wt{L^{4}_{T,x}}} \les  T^{\frac 14-2\dl - } \|u\|_{X_{T}^{0,\frac 12-2\dl}}. 
 \label{L4}
\end{align}
\end{lemma}

\begin{proof}
The properties in (i)-(iii) are standard. See for example \cite{TAObook}. 

\noi To show \eqref{L4}, we  first recall the $L^6_{T,x}$-Strichartz estimate
\begin{align*}
\| S(t) f\|_{\wt{L^{6}_{T,x}}} \les \|f\|_{L^2_x},
\end{align*}
which by transference principle (see \cite[Lemma~2.9]{TAObook}) and \eqref{localspace} implies
\begin{align*}
\| u\|_{\wt{L^{6}_{T,x}}} \les \|u\|_{X^{0,b}_{T}}
\end{align*}
for any $b>\frac 12$. Interpolating with the trivial $L^{2}_{T,x}$-estimate gives 
\begin{align*}
\| u\|_{\wt{L^{4}_{T,x}}} \les \|u\|_{X^{0,\frac 14+}_{T}},
\end{align*}
at which point \eqref{L4} follows from applying \eqref{Lpwt} and \eqref{timeloc}.
\end{proof}

We now state the bilinear Strochartz estimate.

\begin{lemma}[Bilinear Strichartz estimate \cite{BO98, OT}] \label{LEM:bilin}
Let $N\in 2^{\N}$, $f,g\in L^2(\R)$, and $\eta$ be a real-valued smooth cutoff function. Then, 
\begin{align}
\| \P_{N}[ \eta(t) S(t) f \cdot \eta(t) \cj{S(t)g}]\|_{L^{2}_{t,x}} \les N^{-\frac 12} \|f\|_{L^2_x}\|g\|_{L^{2}_x}. \label{bilin}
\end{align}
Moreover, for any $0<\dl \ll 1$ sufficiently small, it holds that
\begin{align}
\| \P_{N}[ u \cdot \cj{v}]\|_{L^{2}_{t,x}} \les N^{-\frac 12+10\dl} \|u\|_{X^{0,\frac 12-2\dl}}\|v\|_{X^{0,\frac 12-2\dl}}.
 \label{bilin1}
\end{align}
\end{lemma}
\begin{proof}
We include a proof of \eqref{bilin} for the reader's convenience.  Taking the space-time Fourier transform, we have
\begin{align*}
\mathcal{F}_{t,x}\{  \eta(t) S(t) f \cdot \cj{\eta(t) S(t)g}\}(\tau,\xi) & = \int \ft f(\mu) \cj{\ft g(\mu-\xi)}  \,\mathcal{F}_t\{\eta^2\}(\tau -\mu^2+(\mu-\xi)^2) d\mu.
\end{align*}
Note that since $\eta\in C_{c}^{\infty}(\R)$, we have $\mathcal{F}_t\{ \eta^2\}\in \mathcal{S}(\R)$.
Then, by Plancherel's theorem and Cauchy-Schwarz inequality in $\mu$, we have 
\begin{align*}
[\text{LHS} \eqref{bilin}]^2 & \leq M_{N} \int |\ft f(\mu)|^2 |\ft g(\mu-\xi)|^2 \int |\mathcal{F}_t\{\eta^2\}(\tau -\mu^2+(\mu-\xi)^2)|d\tau  d\xi d\mu \\
& \les M_{N} \|f\|_{L^2_x}^{2} \|g\|^2_{L^2_x},
\end{align*}
where 
\begin{align*}
M_{N} &= \sup_{|\xi|\sim N, \tau\in \R} \int |\mathcal{F}_t\{\eta^2\}(\tau -\mu^2+(\mu-\xi)^2)| d\mu \\
&  =\sup_{|\xi|\sim N, \tau\in \R} \int |\mathcal{F}_t\{\eta^2\}(\tau +\xi^2-2\xi \mu)| d\mu\\
& \sim  N^{-1} \sup_{|\xi|\sim N, \tau\in \R} \int |\mathcal{F}_t\{\eta^2\}(\tau +\xi^2-\mu)| d\mu \les N^{-1}.
\end{align*}
This proves \eqref{bilin}. 
To obtain \eqref{bilin1}, on the one hand,  \eqref{bilin} and the transference principle (it is well-known that the linear version of this in \cite[Lemma~2.9]{TAObook} generalises with the same proof ideas to multilinear operators) imply
\begin{align}
\| \P_{N}[ u \cdot \cj{v}]\|_{L^{2}_{t,x}} \les N^{-\frac 12} \|u\|_{X^{0,\frac 12+}}\|v\|_{X^{0,\frac 12+}}. \label{bilin2}
\end{align}
On the other hand, by Bernstein's, H\"{o}lder's, and Sobolev inequalities,  we have 
\begin{align}
\| \P_{N}[ u \cdot \cj{v}]\|_{L^{2}_{t,x}}\les N^{\frac 12} \| u \cdot \cj{v}\|_{L^2_{t} L^{1}_x} \les N^{\frac 12} \|u\|_{L^4_{t} L^{2}_x} \|v\|_{L^{4}_{t}L^{2}_x} \les N^{\frac 12} \|u\|_{X^{0, \frac 14+}} \| v\|_{X^{0,\frac 14+}}. \label{bilin3}
\end{align}
Interpolating \eqref{bilin2} and \eqref{bilin3} yields \eqref{bilin1}.
\end{proof}

\begin{remark}\rm 
The conjugation and the projection $\P_{N}$ are important in \eqref{bilin}, and ensure that the constant on the right-hand side of \eqref{bilin} is essentially independent of the frequency supports of the functions $f$ and $g$. If we remove the projection $\P_{N}$, the estimate is now insensitive to any conjugations, and we need to impose a condition on the distance between the Fourier supports of the functions $f$ and $g$. Namely, suppose that $f$ and $g$ are compactly supported on the Fourier side taking values in the sets $S_1$ and $S_2$, respectively. Let $d(S_1,S_2)$ denote the distance between the sets $S_1$ and $S_2$: 
\begin{align*}
d(S_1,S_2) =\inf_{s_j\in S_j, j=1,2} |s_1-s_2|.
\end{align*}
Then, it holds that 
\begin{align*}
\|\eta(t) S(t)f \cdot \eta(t) S(t)g\|_{L^{2}_{t,x}} \les d(S_1,S_2)^{-\frac 12} \|f\|_{L^2}\|g\|_{L^2}.
\end{align*}
\end{remark}
Finally, we have the following useful $L^p$-boundedness for the modulation operators $\Q_{\ll K}$ in \eqref{Qpro}. A similar version of this result was a key ingredient in \cite{GLM}. See \cite[Lemma 4.6]{GLM}. For the reader's convenience, we detail a proof.

\begin{lemma}
Let $1<p<\infty$, and  $N,K\in 2^{\N}$ satisfying $K\gg N^2$. Then, there exists $C_p>0$, depending only on $p$,  such that
\begin{align}
\| \P_{N} \Q_{\ll K}f\|_{L^{p}_{t,x}(\R^2)} \leq C_{p} \| \P_{N}f\|_{L^p_{t,x}(\R^2)}. \label{QKLp}
\end{align}
\end{lemma}
 \begin{proof}
 Let $m(\tau, \xi) : = \wt{\psi}_{N}(\xi) \eta_{10^{-10}K}(\tau+\xi^2)$ be the Fourier multiplier associated to the Fourier multiplier operator $\wt{\P}_{N}\Q_{\ll K}$, where $\wt\P_N$ denotes the wider projector and $\Q_{\ll K}$ is as in \eqref{Qpro}.  The $L^p\to L^p$ boundedness then follows from the Marcinkiewicz multiplier theorem \cite[Corollary 6.2.5]{Grafakos} once we show that 
 \begin{align}
|\partial_{\tau}^{\al_1} \partial_{\xi}^{\al_2} m(\tau,\xi)| \les |\tau|^{-|\al_1|}|\xi|^{-|\al_2|} \label{Marc}
\end{align}
for any $\tau, \xi\neq 0$ and multiindex $(\al_1,\al_2)$ with $|\al_1|+|\al_2|\leq 2$. 

Note that since $|\xi|\sim N$ and $K\gg N^2$, it holds that $|\tau|\les K$. Then, we compute:
\begin{align*}
|\partial_{\xi} m(\tau,\xi)| &\sim N^{-1} |\wt \psi'(\tfrac{\xi}{N}) \eta_{10^{-10}K}(\tau+\xi^2)| +\frac{|\xi|}{K}  |\wt{\psi}_{N}(\xi) \eta'(\tfrac{\tau+\xi^2}{10^{-10}K})| \les |\xi|^{-1}, \\
|\partial_{\tau} m(\tau,\xi)| &\sim K^{-1} |\wt{\psi}_{N}(\xi) \eta'(\tfrac{\tau+\xi^2}{10^{-10} K} ) |  \les |\tau|^{-1} ,
\end{align*}
where $\wt\psi := \wt\psi_1$. 
Proceeding in the same way, we obtain the estimates for the second derivatives and conclude that \eqref{Marc} holds for $m$.
 \end{proof}

\subsection{Gauge transformation} \label{SEC:Gauge}

Following \cite{PMP}, 
we define the gauged variables 
\begin{align}
v := \P_{+,\text{hi}}[ e^{i\be F[u]} u] \qquad \text{and} \qquad w:= \P_{-,\text{hi}}u, 
\label{gauge}
\end{align}
where $F[u]$ denotes the primitive of $|u|^2$ as in \eqref{F}.
For CCM \eqref{CCM} in the Hardy space, we simply have $w\equiv 0$.
We establish the equations for the gauged variables $v$ and $w$ in \eqref{gauge}.

\begin{lemma}\label{LEM:gaugeeqns}
Let $T>0$, $0 < h <\infty$, and $u \in C([0,T];H^{\infty}(\R))$ be a smooth solution of~\eqref{INLS2}.
Then, the variables $v$ and $w$ defined in \eqref{gauge} satisfy the following equations:
\begin{align}
\dt v + i\dx^2 v & =\mathcal{N}_{v}(u)
 =-2\be\P_{+, \hi}[ v \P_{-}\dx(|u|^2)]
+ \P_{+, \hi}[ e^{i\be F} u\QQ_{h}(|u|^2)] ,\label{veq} \\
\dt w+i\dx^2 w&  = \mathcal{N}_{w}(u)  =  2\be \P_{-,\textup{hi}}[ w \P_{+}\dx(|u|^2)]+  \P_{-, \hi}[u \QQ_{h}(|u|^2) ],     \label{weq}
\end{align}
where $\QQ_h$ is as in \eqref{Qh}. 
\end{lemma}

\begin{proof}
We follow the proof in \cite[Lemma 3.4]{PMP}. Using \eqref{gauge} and writing $F=F[u]$ for simplicity, we compute 
\begin{align}
\begin{split}
\dt v +i\dx^2 v &  = \dt \Pbhip[e^{i\be F}u] + i \dx^2 \Pbhip[e^{i\be F}u] \\
& = \Pbhip[ e^{i\be F} i\be (\dt F) u] +\Pbhip[ e^{i\be F}( \dt u + i\dx^2 u)] 
-\Pbhip[ e^{i\be F} \be (\dx^2 F)u]\\
& \hphantom{X} - i \Pbhip[ e^{i\be F} \be^2 (\dx F)^2  u] + 2i\Pbhip[ \dx(e^{i\be F})\dx u] \\
&  =\Pbhip[ e^{i\be F} i \be  (\dt F +i\dx^2 F - \be (\dx F)^2)u] + \Pbhip[ e^{i \be F}(\dt u +i\dx^2 u)] \\
& \hphantom{X} -2\be \Pbhip[ e^{i\be F} (\dx F)\dx u ].
\end{split} \label{veq2}
\end{align}
Meanwhile, using \eqref{F}, \eqref{INLS2}, the fact that $\be, \g\in\R$, and that the operator $\GG_h$ in \eqref{Lh} preserves real-valuedness, we have
\begin{align*}
\dt F & =  i (u\dx \cj{u}-\cj{u}\dx u )+\be |u|^4
\end{align*}
and hence
\begin{align}
 \dt F+ i\dx^2 F - \be (\dx F)^2 &  =2i u \dx\cj{u} .  
\notag
\end{align}
Inserting the above and \eqref{INLS2} into \eqref{veq2}, we find
\begin{align*}
\dt v +i\dx^2 v &  = 2\be \Pbhip[e^{i\be F} \{u \P_{+}\dx(|u|^2)-|u|^2 \dx u-u^2 \dx \cj{u}\}] \\
&  \hphantom{X}
+i\g \Pbhip[ e^{i\be F} |u|^2 u] -i\be \Pbhip[ e^{i\be F} u \L_{h}(|u|^2)].
\end{align*}
Noting that $u\dx(|u|^2) =u^2 \dx\cj{u}+|u|^2 \dx u,$
and using \eqref{pm1},
we further obtain
\begin{align*}
2\be \Pbhip[e^{i\be F} \{u \P_{+}\dx(|u|^2)-|u|^2 \dx u-u^2 \dx \cj{u}\}]
&
=
-2\be \Pbhip[e^{i\be F} u \, \P_{-}\dx(|u|^2)]
.
\end{align*}
Moreover, using that $\P_{+}\P_{-}=0$ and $\P_{+}[ \P_{-}f \cdot \P_{-}g]=0$, we have	
\begin{align*}
- 2 \be \Pbhip[e^{i\be F} u\, \P_{-}\dx(|u|^2)]  
&= -2\be \Pbhip[ v\, \P_{-}\dx(|u|^2)]
- 2 \be \Pbhip[ \Pblo(e^{i\be F}u) \P_{-}\dx(|u|^2)] \\
&= 
-2 \be \Pbhip[ v\, \P_{-}\dx(|u|^2)].
\end{align*}	
Here, the second term vanishes since $\Pbhip[ \Pblo f \cdot \P_{-}g] = 0.$
Then, we rewrite as		
\begin{align*}
\dt v +i\dx^2 v &  = -2\be\P_{+, \hi}[ v \P_{-}\dx(|u|^2)]
+i\g \P_{+, \hi}[ e^{i\be F} |u|^2 u] -i\be \P_{+, \hi}[ e^{i\be F} u \L_{h}(|u|^2)].
\end{align*}									
This establishes \eqref{veq}.

To obtain \eqref{weq}, we simply apply $\P_{-, \text{hi}}$ to both sides of \eqref{INLS2} and noting that 
\begin{align*}
\P_{-,\text{hi}}[ u \P_+ \dx(|u|^2)] &= \P_{-,\text{hi}}[ w \P_{+}\dx(|u|^2)]+ \P_{-,\text{hi}}[ \P_{-,\text{lo}}u \cdot \P_{+}\dx(|u|^2)] \\
& =  \P_{-,\text{hi}}[ w \P_{+}\dx(|u|^2)],
\end{align*}
since $\P_{-,\text{hi}}[ \P_{-,\text{lo}}f \cdot \P_{+}g ]  =0
,$
completing the proof.
\end{proof}

Lastly, from the definition of the gauged variables in \eqref{gauge}, we have the following recovery formula for the solution $u$ to \eqref{INLS2}:
\begin{align}
u&= \Pbhi u + \Pblo u \notag   \\
& = \Pbhip u + w+\Pblo u  \notag \\
&  =  \Pbhip [ e^{i\be F} u e^{-i\be F}] +w+\Pblo u 
\notag  \\
&  =\Pbhip [ e^{-i\be F} v] +  \Pbhip [ e^{-i\be F} \P_{\text{lo}}( e^{i\be F} u)]  + \Pbhip[e^{-i\be F} \P_{-,\text{hi}}(e^{i\be F}u)] + w+\Pblo u. 
\notag
\end{align}
We now apply $\P_{\text{HI}}$ to both sides, recalling that $\P_{\HI} \P_{\hi} = \P_{\HI}$ and $\P_{\HI} \P_{\lo} = 0$,  to obtain:
\begin{align}
\begin{split}
\PbHI u & = \PbHIp[ e^{-i\be F} v] +\PbHIp[ \Pbhi(e^{-i\be F}) \P_{\text{lo}}( e^{i\be F}u)]\\
& \hphantom{X} + \PbHIp[ e^{-i\be F} \P_{-,\text{hi}}(e^{i\be F}u)] + \PbHI w. 
\end{split} \label{PbHIu}
\end{align}
Note that in the second term, the projectors $\PbHI$ and $\Pblo$ allow us to place for free an extra projector $\Pbhi$ onto the first factor $e^{-i\be F}$.\footnote{This can be justified pointwise since the function $u$ and hence $F$ are smooth, from which we have
\begin{align*}
\int \PbHIp [ e^{-i\be F} \P_{\text{lo}}( e^{i\be F} u)] \phi dx 
&= \int e^{-i\be F}\dx \dx^{-1}[ \P_{\text{lo}}( e^{i\be F} u) \P_{-, \HI} \phi]dx 
\\
&=- \int \dx e^{-i\be F} \cdot \dx^{-1}[ \P_{\text{lo}}( e^{i\be F} u) \P_{-, \HI} \phi]dx \\
& =  -\int \dx \Pbhi e^{-i\be F} \cdot  \dx^{-1}[ \P_{\text{lo}}( e^{i\be F} u) \P_{-, \HI} \phi]dx \\
& = \int \PbHIp [\Pbhi( e^{-i\be F} )\P_{\text{lo}}( e^{i\be F} u)] \phi dx 
\end{align*}
for any $\phi \in C^{\infty}_{c}(\R)$.
}

\subsection{CCM regularity properties: using the Hardy space assumption}

First, we establish the $X^{s,b}$-regularity of solutions $u$ to CCM \eqref{CCM}. The relatively simple form of the estimate is entirely thanks to the Hardy space assumption. 

\begin{lemma}\label{LEM:CCMXsb}
Let  $s\geq s_0>\frac 14$, $0<T\leq 1$, and $u$ be a $H^{\infty}_{+}(\R)$-solution to \textup{CCM} \eqref{CCM} on $[0,T]$. Then, 
\begin{align}\label{CCMXsb00}
\begin{split}
\sup_{0\leq \ta \leq 1} \|u\|_{X^{s-\ta,\ta}_{T}} 
&
\les 
\| u\|_{L^\infty_TH^s_x}+ \| J^{s_0}\PbHI u\|_{\wt{L^4_{T,x}}}  \| u\|_{L^\infty_T H^{s_0}_x}\| J^{s}\PbHI u\|_{\wt{L^4_{T,x}}}
\\
&
\qquad 
+ T^\frac12 \| u \|_{L^\infty_T H^s_x} \|u\|^2_{L^\infty_T L^2_x}
.
\end{split}
\end{align}
\end{lemma}

\begin{proof}
We argue as in \cite[Proposition 3.2]{MP}. 
The main point is that for a suitable extension $\wt{u}$ on $[0,T]$ of $u$, it holds that 
\begin{align*}
\sup_{0\le \ta \le 1} 
\| \wt{u}\|_{X^{ s-\ta, \ta}} \les \| \dt u +i\dx^2 u\|_{L^{2}_{T}H^{ s-1}_x} + \|u\|_{L^{\infty}_{T}H^{ s}_x}.
\end{align*}
Thus, from \eqref{CCM}, we need to estimate 
\begin{align}
\|  u\P_{+}\dx(|u|^2)\|_{L^{2}_{T}H^{s-1}_x}.
\notag
\end{align}
By dyadic decomposition, we focus on controling
\begin{align}
N^{s-1} \big\| \P_{N}[ \P_{N_1}u \cdot \P_{+}\dx\P_{N_{23}}[ \cj{\P_{N_2}u} \P_{N_3}u]] \big\|_{L^{2}_{T,x}}
.
 \label{CCMXsb}
\end{align}
Since $\P_{+}u=u$, the frequencies, which satisfy $\xi=\xi_1-\xi_2+\xi_3$, additionally imply
\begin{align}
|\xi| = |\xi_1|+|\xi_3-\xi_2|. \label{CCMfreq}
\end{align}
Therefore, $N\ges N_1 \vee N_{23}$ and it is clear that $N_{23} \les N_2 \vee N_3$. 

\medskip
\noi
$\bullet$ \underline{\textbf{Case 1:} $N\sim N_{1}$.}

\smallskip
\noi
By H\"{o}lder and Bernstein inequalities, we have 
\begin{align}
\eqref{CCMXsb} &\les N^{s-1}N_{23} N_{1}^{-s} \|J^{s}\P_{N_1}u\|_{L^{4}_{T,x}} \|\P_{N_2}u \cdot \P_{N_3}u\|_{L^{4}_{T,x}} \notag \\
&\les N^{-1}N_{23} \|J^{s}\P_{N_1}u\|_{L^{4}_{T,x}} \|\P_{N_2 \vee N_3}u\|_{L^{4}_{T,x}} \|\P_{N_2\wedge N_3}u\|_{L^{\infty}_{T,x}} \notag \\
&\les \frac{N_{23}  (N_2 \wedge N_3)^{\frac 12-s_0}}{N (N_2\vee N_3)^{s_0}}   \|J^{s}\P_{N_1}u\|_{L^{4}_{T,x}}  \|J^{s_0}\P_{N_2\vee N_3}u\|_{L^{4}_{T,x}}  \|J^{s_0}\P_{N_2 \wedge N_3}u\|_{L^{\infty}_{T}L^2_x}.
\label{CCMXsb2}
\end{align}
Now we consider the dyadic prefactor in \eqref{CCMXsb2}. As $s_0>\frac 14$, we may write 
\begin{align}
\frac{N_{23}  (N_2 \wedge N_3)^{\frac 12-s_0}}{N (N_2\vee N_3)^{s_0}}  \sim \frac{N_{23} }{N (N_2\vee N_3)^{2s_0-\frac 12}}. \label{CCMXsbd1}
\end{align}
If $N_2 \vee N_3\ges N$, then we use that $N\ges N_{23}$ to control and $N_{23}$ and we have a negative power of the largest dyadic which allows us to sum over all of the dyadics. If instead $N_2 \vee N_3 \ll N$, then we further bound by
\begin{align*}
 \eqref{CCMXsbd1} \les N^{-1} (N_2 \vee N_3)^{\max(\frac 54-2s_0, 0)} \les N^{-1}\ind_{\{s_0\geq \frac 58\}} + N^{-(2s_0-\frac 14)} \ind_{\{ \frac 14<s_0<\frac 58\}}, 
\end{align*}
where we have a negative power of the largest dyadic frequency, in either case, allowing us to sum in all dyadics.

\medskip
\noi
$\bullet$ \underline{\textbf{Case 2:} $N\gg N_{1}$.}

\smallskip
\noi
By \eqref{CCMfreq}, we must have $N\sim N_{23}\gg N_1$. In particular, $N_{2}\vee N_3 \ges N$.
We then have two further cases depending on the size of $N_{2}\wedge N_3$. 

\medskip
\noi
$\bullet$ \underline{\textbf{Case 2.1:} $N_2 \wedge N_3 \ges N$.}

\smallskip
\noi
In this case, we then have $N_2 \sim N_3 \ges N$. We follow the argument in Case 1 but placing $\P_{N_1}u$ into the space $L^{\infty}_{T,x}$ while $(N_2,N_3)$ both go into $L^4_{T,x}$, with $\P_{N_2 \vee N_3}u$ taking the higher $s$-derivatives. This leads to the dyadic factor
\begin{align*}
\frac{N^{s-1}N_{23} N_{1}^{\frac 12-s_0}}{ N_2^s N_3^{s_0}}  \les 
\frac{N^{s} N_{1}^{\frac 12-s_0}}{N_2^s N_3^{s_0}}  \les N_2^{\max(\frac12-s_0,0)-s_0}  
\sim N_{\max}^{-s_0}
\ind_{\{s_0 \ge \frac12\}}
+
N_{\max}^{-(2s_0 - \frac12)}
\ind_{\{\frac14<s<\frac12\}}, 
\end{align*}
which allows us to sum in the dyadics. 

\medskip
\noi
$\bullet$ \underline{\textbf{Case 2.2:} $N_2 \wedge N_3 \ll N$.}

\smallskip
\noi
This case is finer since we only have one large input frequency. By symmetry, we will assume that $N_3 \sim N\sim N_{23} \gg N_1\vee N_2$. We place $\P_{N_3}u$ into $L^{4}_{T,x}$ with $s$-derivatives and we want to place $\P_{N_1 \wedge N_2}u$ into $L^{\infty}_{T,x}$.
Then by H\"{o}lder and Bernstein as before, we have 
\begin{align*}
\eqref{CCMXsb} &\les \Big(\frac{N}{N_3}\Big)^{s} \|J^s \P_{N_3}u\|_{L^{4}_{T,x}} \cdot \frac{N_{23}}{N} \|\P_{N_1 \vee N_2}u\|_{L^{4}_{T,x}} \|\P_{N_1 \wedge N_2}u\|_{L^{\infty}_{T,x}} \\
& \les  \Big(\frac{N}{N_3}\Big)^{s} \|J^s \P_{N_3}u\|_{L^{4}_{T,x}} \cdot \frac{N_{23}}{N} \cdot \frac{(N_1\wedge N_2)^{\max(\frac 12-s_0,0)}}{(N_1 \vee N_2)^{s_0}} \|J^{s_0}\P_{N_1 \vee N_2}u\|_{L^{4}_{T,x}} \|\P_{N_1 \wedge N_2}u\|_{L^{\infty}_{T}H^{s_0}_x}.
\end{align*}
Now since $s_0>\frac 14$, we have
\begin{align*}
\frac{(N_1\wedge N_2)^{\max(\frac 12-s_0,0)}}{(N_1 \vee N_2)^{s_0}}  \les (N_1 \vee N_2)^{-(2s_0-\frac 12)} \ind_{\{s_0<\frac 12\}} +(N_2\vee N_3)^{-s_0}\ind_{\{s_0\geq \frac 12\}},
\end{align*}
which is a negative power and allows us to perform the dyadic sums over $(N_1,N_2)$.
 For the sum over $N_{23}$, we have
\begin{align} 
\sum_{N_{23} \les N} \frac{N_{23}}{N} \les \frac{N}{N} \sim 1.
\notag
\end{align}
It only remains to sum in $N\sim N_3$, for which by Cauchy-Schwarz, we have
\begin{align}
\begin{split}
\sum_{N\gg 1} \bigg( \sum_{N_3 \sim N}  \bigg(\frac{N}{N_3}\bigg)^{s} \| J^s \P_{N_3}u\|_{L^{4}_{T,x}} \bigg)^2 & \sim \sum_{N\gg 1} \bigg( \sum_{|j| \leq 2}  2^{-sj } \| J^s \P_{2^{j}N}u\|_{L^{4}_{T,x}} \bigg)^2 \\
& \les \sum_{|j|\leq 2} \sum_{N\gg 1}  \| J^s \P_{2^{j}N}u\|_{L^{4}_{T,x}}^{2} \les \| J^s \Pbhi u\|^2_{ \wt{L^4_{T,x}}}, 
\end{split} \label{sum}
\end{align}
which completes the estimate when $N\sim N_3$.

Note that in the above arguments, we do not necessarily always close the estimates with $J^{s_0}\PbHI u$ in $\wt{L^4_{T,x}}$. Up to adding more factors of $L^{\infty}_{T}H^{s}_{x}$, this causes no issue since by writing $u=\PbHI u + \PbLO u$ and using Bernstein's inequality, we have
\begin{align}
\| J^{s_0}u\|_{\wt{L^4_{T,x}}} \leq \| J^{s_0}\PbHI u\|_{\wt{L^4_{T,x}}}  + CT^{\frac 14}\| u\|_{L^{\infty}_{T}L^2_{x}}. 
\notag
\end{align}
This accounts for the presence of terms such as the third one on the right-hand side of \eqref{CCMXsb00}.  We apply this comment throughout the rest of the article without further explicit mention.
\end{proof}

We make a few remarks about Lemma~\ref{LEM:CCMXsb}  and the proof above. First, we only used the Hardy space assumption for the first factor in $u\P_{+}\dx(|u|^2)$ not for the factors $\cj{u} \cdot u$. This ensured that \eqref{CCMfreq} held true so that the output frequency could always be used to control the derivative. This will no longer be the case for INLS \eqref{INLS2} as we discuss in the next section.
Second, the regularity restriction $s>\frac 14$ in Lemma~\ref{LEM:CCMXsb} can be improved down to, at least, $s>0$ by additionally using $L^6_{T,x}$. However, including this extra space becomes a nuisance later as we can no longer gain any factor of $T$ in the equivalent version of \eqref{L4}. As the latter arguments rely on $s>\frac 14$, notably Proposition~\ref{PROP:tri}, we opted for a unified presentation.

\begin{lemma}[Estimates for CCM]
\label{LEM:uinfo}
Let $\frac 14 <s< \frac{3}{4}$, $0\leq T\leq 1$, and $u$ be a $H^{\infty}_{+}(\R)$-solution to \eqref{CCM} on $[0,T]$, and $v$ the gauge variable in \eqref{gauge}. Then, it holds that:
\begin{align}
\begin{split}
\| J^{s}\PbHI u\|_{\wt{L^{4}_{T,x}}} 
\les T^{\frac 14-}(1+\|u\|_{L^{\infty}_{T}H^{\frac 14+}_x}^{2})  \|v\|_{X^{s,\frac 12+\dl}_{T}}   +T^{\frac 14}
 \|u\|_{L^{\infty}_{T}H^{\frac 14+}_x}^3 (1+\|u\|_{L^{\infty}_{T}H^{\frac 14+}_x}^2)
. 
\end{split}
\label{Jsu}
\end{align}
Moreover, for any $M\in \N$ sufficiently large, there exists $\ta>0$ such that
\begin{align}
\begin{split}
\| u\|_{L^{\infty}_{T}H^{s}_{x}} 
\les 
\| u_0 \|_{H^{s}_x} 
+\|v\|_{X^{s,\frac 12+\dl} _{T}}   
&+\|u\|_{L^{\infty}_{T}H^{\frac 14+}_x}^2 \big\{ TM^{3} \|u\|_{L^{\infty}_{T}H^{\frac 14+}_x}  \\
&+ M^{-\ta}(1+\|u\|_{L^{\infty}_{T}H^{\frac 14+}_x}^3+\|v\|_{X^{s,\frac 12+\dl} _{T}})\big\}
.
 \end{split}
\label{uHsbd}
\end{align}
\end{lemma}

\begin{proof}
We first prove \eqref{Jsu}. 
By the recovery formula in \eqref{PbHIu}, recalling that $w\equiv 0$ for \eqref{CCM}, and  triangle inequality, we have
\begin{align}
\begin{split}
\| J^{s}\PbHI u\|_{\wt{L^{4}_{T,x} }}
& 
\les 
\| J^{s}\PbHIp[ e^{-i\be F} v]\|_{\wt{L^{4}_{T,x} }}
+\|J^s \PbHIp[ \Pbhi(e^{-i\be F}) \P_{+,\text{lo}}( e^{i\be F} u)]\|_{\wt{L^{4}_{T,x} }}
 \\
&\hphantom{XX} 
+\|J^{s}\PbHIp [ e^{-i\be F}\P_{-,\text{hi}}(e^{i\be F}u)] \|_{\wt{L^{4}_{T,x} }} 
\\
&
=: \I+\II+\III.
\end{split} \label{Jsu1}
\end{align}

We begin by estimating $\I$, which we write as
\begin{align*}
\| J^{s}\PbHIp[ e^{-i \be F} v]\|_{\wt{L^{4}_{T,x} }} 
 \sim \bigg( \sum_{N\gg 1} N^{2s} \| \P_{N}\P_+[ e^{-i\be F}v]\|_{L^4_{T,x}}^2\bigg)^{\frac{1}{2}}.
\end{align*} 
We split $e^{-i\be F}= \Pblo e^{-i \be F} + \Pbhi e^{-i \be F}$ and consider each contribution, beginning with the contribution from $\Pblo e^{-i \be F}$.  In the following, after we make use of the outer projection factors such as $\PbHIp$ to enforce a large output frequency and some frequency sign behaviour we then remove them using their $L^p\to L^p$ boundedness for any $1<p<\infty$. We will do this without further explicit mention.

By duality,
\begin{align*}
\|\P_{N}[ \Pblo (e^{-i \be F}) v]\|_{L^{4}_{T,x}} = \sup_{\| g\|_{L^{4/3}_{T,x}}=1}  \bigg| \int_{0}^{T}\int_{\R} \cj{\P_{N}g}\cdot  \Pblo (e^{-i \be F}) v dx dt\bigg|.
\end{align*}
Then, for fixed $g\in L^{4/3}_{T,x}$ of unit norm, we have
\begin{align*}
 & \bigg| \int_{0}^{T}\int_{\R} \cj{\P_{N}g}\cdot  \Pblo (e^{i\be F}) v dx dt\bigg| 
  \leq |\jb{\P_{N}g, \Pblo(e^{-i \be F}) (\P_{\ll N}v)}_{L^2_{T,x}}|
  \\
  &
  \hspace{3.5cm}
  +|\jb{\P_{N}g, \Pblo(e^{-i \be F}) (\P_{\gg N}v)}_{L^2_{T,x}}| 
  + |\jb{\P_{N}g, \Pblo(e^{-i \be F}) (\wt{\P}_{N}v)}_{L^2_{T,x}}|
\end{align*}
Then, by an integration by parts and \eqref{PloeiF}, we have
\begin{align*}
\jb{ \Pblo (e^{-i \be F}) , (\P_{N}g) \cj{(\P_{\ll N}v)}}_{L^2_{T,x}} 
&= \jb{ \Pblo (e^{-i \be F}) , \wt{\P}_{N}[ {(\P_{N}g)} \cj{(\P_{\ll N}v)}]}_{L^2_{T,x}} \\
&= -\jb{\dx \Pblo (e^{-i \be F}) ,\dx^{-1} \wt{\P}_{N}[{(\P_{N}g)} \cj{(\P_{\ll N}v)}]}_{L^2_{T,x}} \\
& =  \jb{ \wt{\P}_{N}\Pblo \dx (e^{-i \be F}) , \wt{\P}_{N}[{(\P_{N}g)} \cj{(\P_{\ll N}v)}]}_{L^2_{T,x}}\\
&= 0,
\end{align*}
since $\wt\P_N\Pblo =0$ for $N\gg1$. For similar reasons, we see that 
\begin{align*}
\jb{\P_{N}g, \Pblo(e^{-i \be F}) (\P_{\gg N}v)}_{L^2_{T,x}} = 0.
\end{align*}
Therefore, from H\"older's inequality and \eqref{DeLinfty}, we have
\begin{align}
 \|\P_{N}[ \Pblo (e^{-i \be F}) v]\|_{L^{4}_{T,x}} \leq \|\P_{N}[ \Pblo (e^{- i \be F}) (\wt{\P}_{N}v)]\|_{L^{4}_{T,x}} \les \|\wt{\P}_{N}v\|_{L^{4}_{T,x}},
 \notag
\end{align}
and hence by \eqref{L4}
\begin{align*}
\bigg(\sum_{N\gg 1} N^{2s} \| \P_{N}\P_+[ \Pblo (e^{-i \be F})v]\|_{L^4_{T,x}}^2  \bigg)^{\frac 12}
&\les  \| J^{s}v\|_{\wt{L^{4}_{T,x}}}\les T^{\frac 14-} \|v\|_{X^{s,\frac 12}_{T}}.
\end{align*}

We consider the contribution from $\Pbhi(e^{-i \be F})$. We decompose
\begin{align*}
 \| \P_{N}\PbHIp[ \Pbhi(e^{-i \be F})v]\|_{L^4_{T,x}} \leq \sum_{N_1, N_2 \in 2^{\N}}  \| \P_{N}\PbHIp[ \P_{N_1}(e^{-i \be F})\P_{N_2}v]\|_{L^4_{T,x}}.
\end{align*}
When $N_1 \ges  N_2$, by Bernstein's inequality, Sobolev inequality, and \eqref{YsCTHs}, we have 
\begin{align*}
 \| \P_{N}\PbHIp[ \P_{N_1}(e^{-i \be F})\P_{N_2}v]\|_{L^4_{T,x}} & \les  \|  \P_{N_1}e^{-i \be F}\|_{L^4_{T,x}}  \|\P_{N_2}v\|_{L^{\infty}_{T,x}} \\
 & \les N_1^{-1} \|  \P_{N_1}(|u|^2 e^{-i \be F})\|_{L^4_{T,x}}  N_2^{\frac 12-s} \|v\|_{L^{\infty}_{T}H^{s}_x} \\
 & \les N_{1}^{-\frac 34}  \|  \P_{N_1}(|u|^2 e^{-i \be F})\|_{L^4_{T}L^{2}_x} N_2^{\frac 12-s} \|v\|_{L^{\infty}_{T}H^{s}_x} \\
 &\les N_{1}^{-\frac{3}{4}} N_2^{\frac 12-s} T^{\frac{1}{4}}\| u\|^2_{L^{\infty}_{T}H^{\frac 14+}_x}\|v\|_{X^{s,\frac 12+\dl}_{T}}.
\end{align*}
Then, since $s<\frac 34$, we have
\begin{align*}
N^s N_{1}^{-\frac34} N_2^{\frac 12-s} 
\les N_1^{s- \frac34 + (\frac12-s) \lor 0}
\les N_{\max}^{0-}
,
\end{align*}
which we can use to perform the dyadic summations.
When $N_2 \gg N_1$, then $N_2\sim N$ and we take $\P_{N_1}(e^{i \be F})$ into $L^{\infty}_{T,x}$ and perform the summation over $(N,N_2)$ as in \eqref{sum}, together with \eqref{L4}.
Then,  we have 
\begin{align*}
\| \P_{N_1}(e^{-i \be F})\|_{L^{\infty}_{T,x}} \les N_{1}^{-1+\frac 12} \| \P_{N_1}(|u|^2 e^{-i \be F})\|_{L^{\infty}_{T}L^2_x} \les N_{1}^{-\frac 12}\|u\|_{L^{\infty}_{T}H^{\frac 14+}_x}^2, 
\end{align*}
which allows us to sum in $N_1$. Combining all the estimates above, we have
\begin{align}
\label{CCMI}
\I \les T^{\frac14-} \big(1 + \|u\|^2_{L^\infty_T H^{\frac14+}_x} \big) \| v\|_{X^{s, \frac12 + \dl}_T} .
\end{align}

We move onto $\II$.  
By frequency considerations, we have
\begin{align}
\| \P_{N}\PbHIp[ \Pbhi(e^{-i\be F}) \P_{+,\text{lo}}( e^{i\be F} u)]\|_{L^{4}_{T,x} }  &= \| \P_{N}\PbHIp[ \wt{\P}_{N}\Pbhi(e^{-i\be F}) \P_{+,\text{lo}}( e^{i\be F} u)]\|_{L^{4}_{T,x} } \notag \\
& \les  \| \wt{\P}_{N}\Pbhi(e^{-i\be F}) \P_{+,\text{lo}}( e^{i\be F} u)\|_{L^{4}_{T,x} } \notag \\
& \les \|\P_{+,\text{lo}}( e^{i\be F} u)\|_{L^{\infty}_{T,x} }  \| \wt{\P}_{N}\Pbhi(e^{-i\be F}) \|_{L^{4}_{T,x} }. \label{LpII1}
\end{align} 
We note that by Bernstein's inequality, we have
\begin{align}
\begin{split}
\|\P_{+,\text{lo}}( e^{i\be F} u)\|_{L^{\infty}_{T,x} }  
 \les \|u\|_{L^{\infty}_{T}L^2_x}.
\end{split} \label{Linftyu}
\end{align}
Therefore, with \eqref{Linftyu} used in \eqref{LpII1}, we have 
\begin{align}
(\II)^2 & \les \|u\|_{L^{\infty}_{T}L^2_x}^2 \sum_{N\ges 1} N^{2s}\| \wt{\P}_{N}\Pbhi(e^{-i\be F}) \|_{L^{4}_{T,x} }^2  \notag \\
& \les  \|u\|_{L^{\infty}_{T}L^2_x}^2 \sum_{N\ges 1} N^{-2\eps} N^{2(s-1+\eps)}\|\wt{\P}_{N}\Pbhi( |u|^2 e^{-i\be F}) \|_{L^{4}_{T,x} }^2 \notag  \\
& \les   \|u\|_{L^{\infty}_{T}L^2_x}^2  \sup_{N\ges 1} N^{2(s-1+\eps)}\|\wt{\P}_{N}\Pbhi( |u|^2 e^{-i\be F}) \|_{L^{4}_{T,x} }^2, \label{LpII2}
\end{align}
for $0<\eps \ll 1-s$. It remains to bound the second factor in \eqref{LpII2}, for which we write $|u|^2 = |\P_\lo u|^2 + |\P_\hi u|^2 + 2 \Re(\P_\lo u \cj{\P_\hi u})$. 
By the boundedness of $\Pbhi$ and $\wt{\P}_{N}$ on $L^4_x$, and that $s<1$, we have
\begin{align*}
 \sup_{N\ges 1} N^{2(s-1+\eps)}\|\wt{\P}_{N}\Pbhi( |\Pblo u|^2 e^{-i\be F}) \|_{L^{4}_{T,x} }^2 \les T^{\frac{1}{2}} \|u\|_{L^{\infty}_{T}L^2_x}^{4}.
\end{align*}
By Bernstein's inequality with $1<q<4$ given by $\frac 1q= 1+\frac 14-s-\eps$ and Sobolev embedding, we have 
\begin{align*}
N^{s-1+\eps}\|\wt{\P}_{N}\Pbhi( |\Pbhi u|^2 e^{-i\be F}) \|_{L^{4}_{T,x} } 
&
\les \|\wt{\P}_{N}\Pbhi( |\Pbhi u|^2 e^{-i \be F})\|_{L^4_{T}L^{q}_{x}} \les \| \Pbhi u\|_{L^{8}_{T}L^{2q}_x}^2 
\\
&
\les T^{\frac{1}{4}}\|u\|_{L^{\infty}_{T}H^{\frac 14+}_x}^{2}.
\end{align*}
Finally, by Sobolev embedding, we have
\begin{align*}
N^{s-1+\eps}\|\wt{\P}_{N}\Pbhi( (\Pblo u)(\Pbhi u) e^{-i\be F}) \|_{L^{4}_{T,x} }& \les \|\wt{\P}_{N}\Pbhi( (\Pblo u)(\Pbhi u) e^{-i\be F}) \|_{L^{4}_{T}L^{2}_x }  \\
&\les T^{\frac 14} \|\Pblo u\|_{L^{\infty}_{T,x}} \|\Pbhi u\|_{L^{\infty}_{T}L^{2}_x}  \\
&\les T^{\frac 14} \|u\|_{L^{\infty}_{T}H^{\frac 14+}_x}^2.
\end{align*}
Combining these estimates, we see that 
\begin{align}\label{CCMII}
\II \les T^{\frac 14}\|u\|^3_{L^{\infty}_{T}H^{\frac 14+}_x}.
\end{align}

We move onto the term $\III$. By a similar argument as used for $\I$, we have
\begin{align}
\| \P_{N}\PbHIp[ e^{-i \be F} \P_{-,\text{hi}}(e^{i \be F}u)]\|_{L^{4}_{x}}
& = \| \P_{N}\PbHIp[ \wt{\P}_{N}(e^{-i \be F}) \P_{-,\text{hi}}(e^{i \be F}u)]\|_{L^{4}_{x}} \notag  \\
& \les \sum_{N_2 \les N} \| \P_{N}\PbHIp[ \wt{\P}_{N}(e^{-i \be F}) \P_{-,\text{hi}}\P_{N_2}(e^{i \be F}u)]\|_{L^{4}_{x}} \notag \\
& \les \sum_{N_2 \les N} \| \wt{\P}_{N}(e^{-i \be F})\|_{L^4_x} \| \P_{N_2}\P_{-,\text{hi}}(e^{i \be F}u)]\|_{L^{\infty}_{x}}. \label{III1}
\end{align}
We first control each factor in \eqref{III1}. Fix $\eps>0$ sufficiently small. By Bernstein's inequality, we have
\begin{align}
 \| \P_{N_2}\P_{-,\text{hi}}(e^{i \be F}u)]\|_{L^{\infty}_{x}} 
 &\les   \| \P_{N_2}\P_{-,\text{hi}}(e^{i \be F}\PbLO u)]\|_{L^{\infty}_{x}}  +N_{2}^{\eps}  \| \P_{N_2}\P_{-,\text{hi}}(e^{i \be F}\PbHI u)]\|_{L^{\frac{1}{\eps}}_{x}}
\notag
 \\
 &\les  \|u\|_{L^2_x} +N_{2}^{\eps}  \| \P_{N_2}\P_{-,\text{hi}}(e^{i \be F}\PbHI u)]\|_{L^{\frac{1}{\eps}}_{x}}.
\label{III1a}
\end{align}
Note that in the second inequality we used that the Fourier multiplier associated to $\P_{N_2}\Pbhip$ is supported away from the zero frequency so that $\P_{N_2}\P_{-,\text{hi}}$ is bounded from $L^{\infty}_x$ to $L^{\infty}_x$ uniformly in $N_2$. 
Let $r=\frac{1}{1-\eps}$. By duality, using that $\P_{+}u=u$, and integration by parts, we have
\begin{align*}
 \| \P_{N_2}\P_{-,\text{hi}}(e^{i \be F}\PbHI u)]\|_{L^{\frac{1}{\eps}}_{x}}
& = \sup_{ \|g\|_{L^{r}_x =1}} \bigg|  \int_{\R} e^{i \be F} \cj{\P_{-,\text{hi}}\P_{N_2} g}\PbHI u dx    \bigg| \\
 & = \sup_{ \|g\|_{L^{r}_x =1}} \bigg|  \int_{\R} (\dx e^{i \be F}) \dx^{-1}[\cj{\P_{-,\text{hi}}\P_{N_2} g }\PbHI u] dx    \bigg|.
\end{align*}
Fix $h\in L^r_x$ of unit norm. Then, by H\"{o}lder's inequality and Bernstein,
\begin{align*}
\int_{\R} (\dx e^{i \be F}) \dx^{-1}[\cj{\P_{-,\text{hi}}\P_{N_2}g}\PbHI u] dx & = \sum_{M\ges N_2 \vee K} \int_{\R} \P_{M}(\dx e^{i \be F})\cdot  \dx^{-1}[\cj{\P_{-,\text{hi}}\P_{N_2}g} \cdot \P_{K} u] dx \\
& \les \sum_{M\ges N_2 \vee K} M^{-1} \| \P_{M}(\dx e^{i \be F})\|_{L^{\infty}_{x}} \|\P_{N_2}g\|_{L^r_x} \|\P_{K}u\|_{L^{\frac{1}{\eps}}_x} \\
& \les \sum_{M\ges N_2 } M^{-\frac 34} \| \P_{M}( |u|^2 e^{i \be F})\|_{L^{\infty}_{x}} \|u\|_{H^{ \frac{1}{4}}_x} \\
&\les \sum_{M\ges N_2 } M^{-\frac 34} M^{\frac 12} \| |u|^2 \|_{L^{2}_x}\|u\|_{H^{ \frac{1}{4}}_x}  \\
& \les N_{2}^{-\frac 14}\| u\|_{H^{\frac 14+}_x}^{3}.
\end{align*}
Therefore, 
\begin{align}
\| \P_{N_2}\P_{-,\text{hi}}(e^{i \be F}u)\|_{L^{\infty}_{x}} \les  \|u\|_{L^2_x}  + N_{2}^{-\frac 14+\eps}\| u\|_{H^{\frac 14+}_x}^{3}. \label{III2}
\end{align}
We reiterate that \eqref{III2} was derived under the Hardy space assumption $\P_{+}u=u$.
Now, we consider the first factor in \eqref{III1}. By Bernstein and Sobolev inequality, 
\begin{align}
\| \wt{\P}_{N}(e^{-i \be F})\|_{L^{4}_x} & \les N^{-1} \| \wt{\P}_{N}( |u|^2 e^{-i \be F})\|_{L^{4}_x} \les N^{-\frac 34} \| |u|^2\|_{L^2_x} \les  N^{-\frac 34}  \| u\|_{H^{\frac 14+}_x}^2. \label{III3}
\end{align}
Inserting \eqref{III2} and \eqref{III3} into \eqref{III1}, we find
\begin{align*}
N^{s} \| \P_{N}\PbHIp[ e^{-i \be F} \P_{-,\text{hi}}(e^{i \be F}u)]\|_{L^{4}_{x}} &\les N^{s-\frac 34 + } \|u\|_{H^{\frac 14+}_x}^3 (1+\|u\|_{H^{\frac 14+}_x}^2),
\end{align*}
where we have a negative power of $N$ provided that $s<\frac 34$. We finally obtain
\begin{align}
\III \les T^{\frac 14}  \|u\|_{H^{\frac 14+}_x}^3 (1+\|u\|_{H^{\frac 14+}_x}^2),  \label{III4}
\end{align}
completing the estimate in \eqref{Jsu}.

We now move onto the estimate \eqref{uHsbd}. Fix $M\in \N$ sufficiently large and decompose $u=\P_{\leq M}+ \P_{>M}u$. 
To control $\P_{\leq M} u$, we use the Duhamel formulation of \eqref{CCM} and obtain
\begin{align}
\|\P_{\leq M} u \|_{L^{\infty}_{T}H^{s}_x} 
&
\les \| u_0\|_{H^{s}_{x}} +  T \| \P_{\leq M}( u \P_{+}\dx(|u|^2))\|_{L^{\infty}_{T}H^{s}_{x}}  
.
\label{JsuPLO}
\end{align}

\noi
As $\P_{+}u=u$, in the second term $u \P_{+}\dx(|u|^2)$, all input functions are supported on frequencies $|\xi|\les M$  and thus by H\"{o}lder's inequality and Sobolev embedding, we have
\begin{align}
\| \P_{\leq M}( u \P_{+}\dx(|u|^2))\|_{L^{\infty}_{T}H^{s}_{x}}  
\les  M^{1+s}\|\P_{\les M}u\|_{L^{\infty}_{T}L^{4}_x}  \| \P_{\les M} u\|_{L^{\infty}_{T}L^{8}_x}^2
 \les M^{s+\frac 32} \|u\|_{L^{\infty}_{T}H^{\frac 14+}_x}^{3}.
\label{PLOG1}
\end{align}
Combining \eqref{JsuPLO} and \eqref{PLOG1}, we obtain
\begin{align*}
\|\P_{\leq M} u \|_{L^{\infty}_{T}H^{s}_x}  \les  \| u_0\|_{H^{s}_{x}} + TM^{s+\frac 32} \|u\|_{L^{\infty}_{T}H^{\frac 14+}_x}^{3}.
\end{align*}

We move onto the high-frequency portion $\P_{>M}u$ and use \eqref{PbHIu} with $w\equiv0$, to estimate each part in $L^{\infty}_{T}H^{s}_x$.
We begin with the corresponding term $\I$ in \eqref{Jsu1}. First, by the fractional Leibniz rule, \eqref{YsCTHs} and \eqref{DeLinfty}, we have
\begin{align}
\|\P_{>M}J^s\PbHIp[& \Pblo (e^{-i \be F})v]\|_{L^\infty_T L^{2}_x}   \sim \|\P_{>M}D^s\PbHIp[ \Pblo (e^{-i \be F})v]\|_{L^\infty_T L^{2}_x} \notag \\ 
& \les \| D^s \Pblo e^{-i \be F}\|_{L^{\infty}_{T,x}} \|v\|_{L^\infty_T L^2_x}  + \|e^{-i \be F}\|_{L^{\infty}_{T,x}} \|D^s v\|_{L^\infty_T L^2_x} \les \|v\|_{X^{s,\frac 12+}_{T}}.  \label{I1}
\end{align}
Now, by the triangle inequality, we have
\begin{align*}
\|\P_{>M}&J^s\PbHIp[ \Pbhi (e^{-i \be F})v]\|_{L^{2}_x}  \\
&\leq \|\P_{>M}J^s\PbHIp[ \Pbhi \P_{\ll M} (e^{-i \be F}) \P_{\ges M}v]\|_{L^{2}_x}  +\|\P_{>M}J^s\PbHIp[ \P_{\ges M} (e^{-i \be F})v]\|_{L^{2}_x}
\end{align*}
and we estimate each of these terms. 
By the fractional Leibniz rule, Sobolev embedding, and \eqref{DeLinfty}, we have
\begin{align}
\|&\P_{>M}J^s\PbHIp[ \Pbhi \P_{\ll M} (e^{-i \be F})\P_{\ges M}v]\|_{L^\infty_T L^{2}_x}  \notag  \\
& \les \| J^{s}\Pbhi \P_{\ll M} (e^{-i \be F})\|_{L^{\infty}_{T}L^{4}_{x}}\| \P_{\ges M}v\|_{L^{\infty}_{T}L^{4}_{x}} 
+ \| \Pbhi \P_{\ll M} (e^{-i \be F})\|_{L^{\infty}_{T,x}} \|J^{s} \P_{\ges M}v\|_{L^{\infty}_{T}L^{2}_x}\notag  \\
& \les M^{\frac 14 -s}  \| J^{s-1}\Pbhi \P_{\ll M} (|u|^2 e^{-i \be F})\|_{L^{\infty}_{T}L^{4}_{x}}  \|v\|_{X^{s,\frac 12+}_{T}} +  \|v\|_{X^{s,\frac 12+}_{T}} \notag \\
& \les M^{\frac 14 -s}  \| J^{s-\frac 34}\Pbhi \P_{\ll M} (|u|^2 e^{-i \be F})\|_{L^{\infty}_{T}L^{2}_{x}}  \|v\|_{X^{s,\frac 12+}_{T}} +  \|v\|_{X^{s,\frac 12+}_{T}}\notag  \\
& \les M^{\frac 14 -s}  \|u\|_{L^{\infty}_{T}H^{\frac 14}_x}^{2}  \|v\|_{X^{s,\frac 12+}_{T}}+ \|v\|_{X^{s,\frac 12+}_{T}}. \label{I2}
\end{align}
For the second term, we argue similarly: 
\begin{align}
\|\P_{>M}J^s\PbHIp[ \P_{\ges M} (e^{-i \be F})v]\|_{L^\infty_T L^{2}_x}& \les  \| J^{s}\P_{\ges M} (e^{-i \be F})\|_{L^{\infty}_{T}L^{4}_{x}} \|v\|_{X^{\frac 14,\frac 12+}_{T}}+ \|v\|_{X^{s,\frac 12+}_{T}} \notag \\
& \les M^{s-\frac 34}  \|u\|_{L^{\infty}_{T}H^{\frac 14}_x}^{2}  \|v\|_{X^{s,\frac 12+}_{T}}+ \|v\|_{X^{s,\frac 12+}_{T}}. \label{I3}
\end{align}
This completes the estimate for $\I$.

By the signs and the $\Pblo$ on the second term of $\II$ (analogous to $\II$ in \eqref{Jsu1}),  we have
\begin{align*}
\II: = \|J^s \P_{>M}\PbHIp[ \Pbhi(e^{-i \be F}) \P_{+,\text{lo}}( e^{i \be F} u)]\|_{L^{\infty}_{T}L^2_x }  \les  \|D^s [ \P_{\ges M}(e^{-i \be F}) \P_{+,\text{lo}}( e^{i \be F} u)]\|_{L^{\infty}_{T}L^2_x }
.
\end{align*}
Thus, by the fractional Leibniz rule (Lemma~\ref{LEM:leib}), Bernstein's inequality, and H\"older's inequality, we get
\begin{align}
 & \|D^s [ \P_{\ges M}(e^{-i \be F}) \P_{+,\text{lo}}( e^{i \be F} u)]\|_{L^{\infty}_{T}L^2_x }
\notag
\\
 & \les \|D^{s} \P_{\ges M} (e^{-i \be F})\|_{L^{\infty}_{T}L^2_x} \|\P_{+,\text{lo}}(e^{i \be F}u)\|_{L^{\infty}_{T,x}}
+ \|\P_{\ges M}e^{-i \be F}\|_{L^{\infty}_{T}L^2_x}  \|D^{s}\P_{+,\text{lo}}(e^{i \be F}u)\|_{L^{\infty}_{T,x}}  
\notag
\\
 & \les \|D^{s-1}\P_{\ges M} (|u|^2 e^{-i \be F})\|_{L^{\infty}_{T}L^2_x} \|u\|_{L^{\infty}_{T}L^2_x} 
\notag
\\
& \les M^{-\frac 14} \|u\|_{L^{\infty}_{T}H^{\frac 14+}_x}^2 \|u\|_{L^{\infty}_{T}L^2_x}.
 \label{PU234}
\end{align}

Now we estimate the analog of $\III$ in \eqref{Jsu1}, namely
\begin{align}
\III := \|J^s \P_{>M}\PbHIp [ \Pbhi (e^{-i \be F})\P_{-,\text{hi}}(e^{i \be F}u)] \|_{L^2_x}.   \label{PU5671}
\end{align}
By duality, we reduce to controlling
\begin{align*}
\sum_{N\ges M} \sum_{N_1 \ges N\vee N_2} \int_{\R} \cj{J^s \P_{N}g} \cdot  \P_{N_1}(e^{-i \be F}) \cdot \P_{N_2}\P_{-,\text{hi}}(e^{i \be F}u) dx,
\end{align*}
where $g\in L^2_x$ of unit norm.
We split the summation into two cases: $N_1 \gg N_2$ or $N_1 \sim N_2$.
In the former case, by Cauchy-Schwarz and \eqref{III2}, we bound this contribution by
\begin{align*}
& \|u\|_{H^{\frac 14+}}(1+ \|u\|_{H^{\frac 14+}}^2)
\sum_{N\sim N_1 \ges M}  N^{s} \|\P_{N}h\|_{L^2_x} \|\P_{N_1}(e^{-i \be F})\|_{L^2_x} \\
& \les  \|u\|_{H^{\frac 14+}}(1+ \|u\|_{H^{\frac 14+}}^2) \sum_{N\sim N_1 \ges M}  N^{s-1} \|\P_{N}h\|_{L^2_x} \|\P_{N_1}( |u|^2 e^{-i \be F})\|_{L^2_x}  \\ 
& \les M^{-\frac 14}  \|u\|_{H^{\frac 14+}}(1+ \|u\|_{H^{\frac 14+}}^2) \sum_{N\sim N_1}\|\P_{N}h\|_{L^2_x} \|\P_{N_1}( |u|^2 e^{-i \be F})\|_{L^2_x}  \\
&\les   M^{-\frac 14} \|u\|_{H^{\frac 14+}}(1+ \|u\|_{H^{\frac 14+}}^2) \|u\|_{L^4_x}^{2} \|h\|_{L^2_x} \\
&\les  M^{-\frac 14}\|u\|_{H^{\frac 14+}}^3 (1+ \|u\|_{H^{\frac 14+}}^2),
\end{align*}
where in the second inequality we used that $s<\frac 34$.
If instead we have $N_1\sim N_2$, the $N_1 \ges N$ and so by a similar argument, we have
\begin{align*}
& \|u\|_{H^{\frac 14+}}(1+ \|u\|_{H^{\frac 14+}}^2)
\sum_{N}  N^{s} \|\P_{N}h\|_{L^2_x} \sum_{N_1\sim N_2 \ges M}  \|\P_{N_1}(e^{-i \be F})\|_{L^2_x} \\
& \les  \|u\|_{H^{\frac 14+}}(1+ \|u\|_{H^{\frac 14+}}^2) \sum_{N} N^{-\eps} \|\P_{N}h\|_{L^2_x} \sum_{N_1 \ges M} N_1^{s-1+\eps}  \|\P_{N_1}( |u|^2 e^{-i \be F})\|_{L^2_x}  \\ 
& \les   \|u\|_{H^{\frac 14+}}(1+ \|u\|_{H^{\frac 14+}}^2) \|h\|_{L^2_x}\bigg(  \sum_{N_1 \ges M}N_1^{2(s-1+\eps)}\bigg)^{\frac 12} \|u\|_{L^4_x}^2  \\
&\les M^{-\frac 14} \|u\|_{H^{\frac 14+}}^3 (1+ \|u\|_{H^{\frac 14+}}^2).
\end{align*}
Therefore,
\begin{align}
\| J^{s}\P_{>M}\P_{+} [ \Pbhi (e^{-i\be F})\P_{-,\text{hi}}(e^{i\be F}u)] \|_{L^{\infty}_{T}L^2_x }  & \les M^{-\frac 14} \|u\|_{L^{\infty}_{T}H^{\frac 14+}}^3 (1+ \|u\|_{L^{\infty}_{T}H^{\frac 14+}}^2). \label{III+}
\end{align}

Thus, we have shown that for $\ta= \min(s-\frac 14, \frac 34-s, \frac 14)>0$, it holds that
\begin{align}
\|\Pbhi u\|_{L^{\infty}_{T}H^{\l, s}_x} \les \|v\|_{X^{s,\frac 12+\dl}} +M^{-\ta} \|u\|_{L^{\infty}_{T}H^{\frac 14+}}^2 (1+ \|u\|_{L^{\infty}_{T}H^{\frac 14+}}^3 + \|v\|_{X^{s,\frac 12+\dl}}). 
\notag
\end{align}
This completes the proof of \eqref{uHsbd}. 
\end{proof}

\subsection{INLS regularity properties: a decomposition}\label{SEC:INLSdecomp}

We now consider $u$ to be a solution of INLS \eqref{INLS2}, for which we want to show that $u \in X^{s-\s, \frac12}_T$ for some $\s\in (0,1)$ and $s\leq \frac 12$, akin to Lemma~\ref{LEM:CCMXsb} for solutions to CCM \eqref{CCM} in the Hardy space. In Lemma~\ref{LEM:CCMXsb}, the restriction to functions in the Hardy space avoids certain bad frequency interactions in the nonlinearity. However, for a solution $u$ to INLS, the equation does not preserve the Hardy space, and we must exploit alternative structure to show an analogue of Lemma~\ref{LEM:CCMXsb}. 
In particular, we start by rewriting the nonlinearity in \eqref{INLS2}:
\begin{align}
u\P_{+}\dx( |u|^2) & = (\P_+ u) \P_{+}\dx( |u|^2) +( \P_{-,\text{lo}}u)\P_{+}\dx( |u|^2)
 + w\P_{+}\dx(|u|^2),
\label{INLSn1}
\end{align}
where $w$ is the gauged variable as in \eqref{gauge}.
The first term is like in CCM case where the output frequency satisfies $|\xi|= |\xi_1|+|\xi_3-\xi_2| \geq |\xi_3-\xi_2|$. The second term essentially has $|\xi|\sim |\xi_3-\xi_2|$, like in the first term. The main issue for INLS \eqref{INLS2}, as compared to CCM \eqref{CCM}, is the third term in \eqref{INLSn1} which needs further work.

We write the third term in \eqref{INLSn1} as follows
\begin{align}
\begin{split}
w\P_{+}\dx(|u|^2) = & w\mathbb{\P}_{+, \text{LO}} \dx(|u|^2)
+\PbHIp[w\PbHIp\dx(|u|^2)] \\
&+\PbLO[w\PbHIp\dx(|u|^2)] + \P_{-,\text{HI}}[w\PbHIp \dx(|u|^2)].
\end{split}
 \notag
\end{align}
The first term is harmless as it essentially contains no derivative. The third term can be controlled adequately at least for $s>\frac 14$. 
However, to observe a stronger smoothing effect, the second and fourth terms need further decomposition due to dangerous Low-Low-High to High interactions.
We then write
\begin{align}
\P_{\pm, \text{HI}}[w\PbHIp\dx(|u|^2)] = \mathcal{W}^{\pm}_{1}(u,w) + \mathcal{W}^{\pm}_{2}(u,w) 
\notag
\end{align}
where
\begin{align}
\mathcal{W}^{\pm}_{1}(u_1,u_2,w) &: = \sum_{N\ges N_1} \P_{N}\P_{\pm, \text{HI}}[ (\P_{N_1}w)\PbHIp\dx(\cj{u_1}u_2)], 
\notag
\\
\mathcal{W}^{\pm}_{2}(u_1,u_2,w) &: = \sum_{N\ll N_1} \P_{N}\P_{\pm, \text{HI}}[ (\P_{N_1}w)\PbHIp\dx(\cj{u_1}u_2)],
\label{W2}
\end{align}
and where, with a slight abuse of notation, we defined $\mathcal{W}^{\pm}_{2}(u,w):= \mathcal{W}^{\pm}_{2}(u,u,w)$.
In the support of $\mathcal{W}^{\pm}_1$, by sign consideration,  we at least have $|\xi|\ges |\xi_3-\xi_2|$ which provides control on the derivative.
However, in the support of $\mathcal{W}^{\pm}_2$, we have $|\xi|\ll |\xi_3-\xi_2|$ meaning that no amount of smoothing from the norm can help to control the derivative below $H^{\frac 12}(\R)$.

To summarise, we have the following decomposition of the main part of the nonlinearity:
\begin{align}
u\P_{+}\dx(|u|^2)  & = (\P_{+}u)\P_{+}\dx(|u|^2) + (\P_{-,\lo}u)\P_{+}\dx(|u|^2) + w\P_{+,\text{LO}}\dx(|u|^2) 
+ w\PbHIp \dx(|u|^2) \notag \\
& =: G_1(u,u, u) + G_2(u,u, w) +  B(u,u ,w),  \label{INLSn4}
\end{align}
where
\begin{align}
\begin{split}
G_1(u_1,u_2,u_3) &: =(\P_{+}u_1)\P_{+}\dx(\cj{u_2}u_3) + (\P_{-,\lo}u_1)\P_{+}\dx(\cj{u_2}u_3), 
\\
G_{2}(u_1,u_2,w)&:= w\P_{+,\text{LO}}\dx(\cj{u_1}u_2) +\mathcal{W}^{+}_{1}(u_1,u_2,w)+\mathcal{W}^{-}_{1}(u_1,u_2,w), \\
B(u_1,u_2,w)  &:= \mathcal{W}^{+}_{2}(u_1,u_2,w) +\mathcal{W}^{-}_{2}(u_1,u_2,w) + \PbLO[ w\PbHIp \dx(\cj u_1 \cdot u_2)].
\end{split} \label{goodbad}
\end{align}
We point out that the operators $G_1, G_2,$ and $B$ are privileged to the location of $w$, which appears linearly. Consequently, we only need to vary the 
functions $(u_1,u_2,u_3)$ when obtaining difference estimates.

We start by estimating the ``bad part" $B(u,u,w)$ of the nonlinearity in \eqref{INLSn4}.

\begin{lemma}\label{LEM:BXsb}
Fix $\dl>0$ small, $s\geq s_0>0$ such that $s+s_0>\frac 12+10\dl$,  and 
\begin{align}
\s>\max(\tfrac 12-s_0+10\dl, s-s_0). \label{sigma}
\end{align}
Then, for $B$ defined in \eqref{goodbad}, we have
\begin{align*}
\|B(u_1,u_2,w)\|_{X^{ s-\s, -\frac 12+2\dl}_{T}}
&\les \|w\|_{X^{s,\frac 12+\dl}_{T}}\Big\{   \|J^{s_0} u_1\|_{L^{4}_{T,x}}  \|J^{s_0} u_2\|_{L^{4}_{T,x}}+ T^{\frac12} \|u_1\|_{L^{\infty}_{T}H^{s_0}_x}\|u_2\|_{L^{\infty}_{T}H^{s_0}_x} \Big\} .
\end{align*}
\end{lemma}
\begin{proof}
We only prove the estimate when $u_1=u_2=u$ as the general case clearly follows.
By \eqref{goodbad} and the triangle inequality, we need to show
\begin{align*}
\max_{\mathfrak{s} \in \{\pm\}}\|\mathcal{W}_2^{\mathfrak{s}} (u,w)\|_{X^{s-\s, -\frac 12+2\dl}_{T}}
+&\| \PbLO[ w\PbHIp \dx(|u|^2)]\|_{X^{s-\s, -\frac 12+2\dl}_{T}} \\
&\les \|w\|_{X^{s,\frac 12+\dl}_{T}}  \Big[ \|J^{s_0} u\|_{L^{4}_{T,x}}^2 + T^{\frac12} \|u\|_{L^{\infty}_{T}H^{s_0}_x}^2 \Big] , 
\end{align*}
where $\W^\pm_2$ are as in \eqref{W2}. 
By dyadic decomposition and duality, it suffices to estimate
\begin{align}
\sum_{\substack{N,N_1,N_2,N_3,N_{23} \\ N\ll N_1 \sim N_{23}}}[ \ind_{\{N\gg 1\}} N^{s-\s} +\ind_{\{N\les 1\}}]\int_{0}^{T} \int_{\R} \cj{\P_{N}g} \cdot \P_{N_1}w \,  \PbHIp \dx \P_{N_{23}}( \cj{\P_{N_2}u} \P_{N_3}u) dxdt   \label{dyadic1}
\end{align}
where $g\in X^{0,\frac 12-2\dl}_T$ and $\|g\|_{X^{0,\frac 12 -2\dl}_T}\leq 1$ 
and we have noted that in the support of both $\mathcal{W}^{\pm}_{2}(u,w)$ and  $\PbLO[ w\PbHIp \dx(|u|^2)]$, we have $N_{1}\sim N_{23}$.

If $N_{23}\les 1$, it is enough to place all functions into $L^4_{T} W^{s_0,4}_x$ and use \eqref{L4}, which requires $\s>s-s_0$. 
Thus, we assume that $N_{23}\gg 1$. We first consider the case when $N\ges 1$.

\smallskip
\noi
$\bul$ \underline{\textbf{Case 1: $N_{2}\wedge N_3 \gg1$}.}\\
By H\"{o}lder's inequality and \eqref{bilin1},  we bound \eqref{dyadic1} above by
\begin{align*}
&\sum_{\substack{N,N_1,N_2,N_3,N_{23} \\ N\ll N_1\sim N_{23}}} N^{s-\s} N_{23}^{\frac 12+10\dl} \|\P_{N} g\|_{X^{0,\frac 12-2\dl}_T} \|\P_{N_1}w\|_{X^{0,\frac 12}_T} \|\P_{N_2}u\|_{L^4_{T,x}} \|  \P_{N_3}u\|_{L^{4}_{T,x}} \\
&\les \sum_{\substack{N,N_1,N_2,N_3,N_{23} \\ N\ll N_1 \sim N_{23}}} \frac{N^{s-\s} N_{23}^{\frac 12+10\dl  }}{ N_1^s (N_2 N_3)^{s_0}} \|\P_{N} g \|_{X^{0,\frac 12-2\dl}_T} \|\P_{N_1}w\|_{X^{s,\frac 12}_T} \|J^{s_0}\P_{N_2}u\|_{L^4_{T,x}} \| J^{s_0} \P_{N_3}u\|_{L^{4}_{t,x}}.
\end{align*}

\noi
Now we consider the dyadic factor.  
Since $N_1 \sim N_{23} \les N_2 \lor N_3$, we have
\begin{align*}
\frac{N^{s-\s} N_{23}^{\frac 12+10\dl}}{ N_1^s (N_2 N_3)^{s_0}} 
\les 
\frac{N^{s-\s} }{ N_1^{s + s_0 - \frac12 - 10\dl -} (N_2 \lor N_3)^{0+}}
\les 
N_1^{-s_0 + \frac12 - \s + 10 \dl +} (N_2 \lor N_3)^{0-}
\end{align*}
where we have used that $s+s_0>\frac 12+10\dl$. Note that we have a negative powers above provided that $s_0+\s>\frac 12 +10\dl$, which forces the first condition in \eqref{sigma}.

\smallskip
\noi
$\bul$ \underline{\textbf{Case 2: $N_2 \wedge N_3 \les 1$}.}\\
In this case, we handle $g$ and $w$ as in Case 1, and we simply place the term with $\P_{N_2 \land N_3}$ into $L^{\infty}_{T,x}$ and use Bernstein's inequality, while the term with $\P_{N_2 \vee N_3}$ is placed in $L^{2}_{T,x}$. This yields the same numerology on $(s,s_0,\s)$ as above, since we did not use the extra weight $(N_2 \land N_3)^{-s_0}$ before.

\smallskip

Lastly, if  $N \les 1$, which is the case for $\PbLO[ w\PbHIp \dx(|u|^2)]$, arguing as above in Cases 1-2, we impose no condition on $\s$ and only need $s+s_0>\frac 12+10\dl$.
This completes the proof.
\end{proof}

Let $\mathcal{I}$ denote the Duhamel integral operator  associated to \eqref{INLS2}:
\begin{align*}
\mathcal{I}[f](t) = \int_{0}^{t} S(t-t') f(t')dt'
\end{align*}
which is the solution to the inhomogenous equation:
\begin{align}
\begin{cases}
\dt\mathcal{I}[f]+i\dx^2 \mathcal{I}[f] = f,\\
\mathcal{I}[f]\vert_{t=0}=0.
\end{cases} \label{duh}
\end{align}
Next, we obtain the $X^{s, b}$-regularity of $u$ after removing the contribution from the bad part~$B$.

\begin{lemma}\label{LEM:GXsb}
Let $0< T \le 1$, $\dl>0$ sufficiently small,  and $\frac14<s_0\leq s<\frac 34$ such that $s+s_0>\frac 12+10\dl$. Let $u$ be an $H^{\infty}(\R)$-solution to \eqref{INLS2} on $[0,T]$ and $w=\P_{-,\textup{hi}}u$.
Then, for $B$ as in \eqref{goodbad}, there exists $\ta>0$ such that 
\begin{align}
\begin{split}
&\sup_{0\leq \eta\leq 1}\| u - \mathcal{I}[B(u,u,w)]\|_{X^{s-\eta,\eta}_{T}} \\
&\les \|u\|_{L^{\infty}_{T}H^{s}_{x}} +\|w\|_{X^{s,\frac 12+\dl}_{T}} \Big[ \|J^{s_0} u\|_{L^{4}_{T,x}}^2 + T^{\ta} \|u\|_{L^{\infty}_{T}H^{s_0}_x}^2 \Big] +  T^{\frac 12}\|\QQ_{h}\|_{\textup{op}} \|u\|_{L^{\infty}_{T}H^{\frac 14+}_x}^{3} \\
&\hphantom{X}  +\| J^{s}\PbHI u\|_{\wt{L^4_{T,x}}} \big( \| J^{s_0}\PbHI u\|_{\wt{L^4_{T,x}}} +  \|u\|_{L^{\infty}_{T}H^{s_0}_x}\big) \|u\|_{L^{\infty}_{T}H^{s_0}_x}
+ T^{\frac12-} \|u\|^3_{L^\infty_T H^{s_0}_x}
\end{split} \label{uIB}
\end{align}
where $\QQ_h$ is as in \eqref{Qh}.
Moreover, 
let $u_j$ be $H^{\infty}(\R)$-solutions to \eqref{INLS2} on $[0,T]$ with initial data $u_j(0)=u_{0,j}$, $j=1,2$, and let $U:=u_1-u_2$ and $W:=w_1-w_2=\P_{-,\textup{hi}}U$. 
Then, we have the difference estimate:
\begin{align}
&\sup_{0\leq \eta\leq 1}\| (u_1 - \mathcal{I}[B(u_1,u_1, w_1)])-(u_2 - \mathcal{I}[B(u_2, u_2, w_2)])\|_{X^{s-\eta,\eta}_{T}} \notag \\
&\les \|U\|_{L^{\infty}_{T}H^{s_0}_{x}} +\|W\|_{X^{s,\frac 12+\dl}_{T}} \max_{j=1,2}\Big[ \|J^{s_0} u_j\|_{L^{4}_{T,x}}^2+ T^{\ta} \|u_j\|_{L^{\infty}_{T}H^{s_0}_x}^2 \Big]  \label{uIBdiff} \\
& \hphantom{X} +\max_{j=1,2}(\|u_j\|_{L^{\infty}_{T}H^{s_0}_x}^2 + \|
J^{s_0}\PbHI u_j\|_{\wt{L^4_{T,x}}}^2)\big[  \|
J^{s}\PbHI U\|_{\wt{L^4_{T,x}}} +(1+\|\QQ_{h}\|_{\textup{op}})\|U\|_{L^{\infty}_{T}H^{s}_x} \big]  \notag \\
& \hphantom{X} +\max_{j=1,2}(\|u_j\|_{L^{\infty}_{T}H^{s_0}_x}\|u_j\|_{L^{\infty}_{T}H^{s}_x}  + \|
J^{s_0}\PbHI u_j\|_{\wt{L^4_{T,x}}}\|
J^{s}\PbHI u_j\|_{\wt{L^4_{T,x}}})\big[  \|
J^{s_0}\PbHI U\|_{\wt{L^4_{T,x}}} +\|U\|_{L^{\infty}_{T}H^{s_0}_x} \big]. \notag
\end{align}
\end{lemma}

\begin{proof}
We show only \eqref{uIB}, as the difference estimate \eqref{uIBdiff} follows by a similar argument. 
We first argue as in \cite[Proposition 3.2]{MP}, where it is shown that for a suitable extension $\wt{z}$ of a space-time function $z$ on $[0,T]$, it holds that 
\begin{align*}
\sup_{0\leq \eta\leq 1} \|\wt{z}\|_{X^{s-\eta,\eta}_T} \les \| \dt z+ i\dx^2 z\|_{L^{2}_{T}H^{s-1}_x} + \|z\|_{L^{\infty}_{T}H^{s}_x}.
\end{align*}
Then, taking $z=u-\mathcal{I}[B(u, u,w)]$ and using \eqref{INLS2}, \eqref{duh}, and \eqref{INLSn4}, we have
\begin{align}
& \sup_{0\leq \eta \leq 1}\| u - \mathcal{I}[B(u, u,w)]\|_{X^{s-\eta,\eta}_{T}} 
\notag
\\
&\les \|(\dt +i\dx^2)(u - \mathcal{I}[B(u, u,w)])\|_{L^{2}_{T}H^{s-1}_x} + \| u - \mathcal{I}[B(u,u, w)]\|_{L^{\infty}_{T}H^{s-1}_x} \notag \\
& \les  \| G(u,u ,w)\|_{L^{2}_{T}H^{s-1}_{x}}+ \| u\QQ_{h}(|u|^2)\|_{L^{2}_{T}H^{s-1}_{x}} + \|u\|_{L^{\infty}_{T}H^{s}_x} + \|B(u,u, w)\|_{X^{s-1,-\frac 12+2\dl}_{T}} \label{uIB2}
,
\end{align}
where $\QQ_h$ is as in \eqref{Qh}, and we used \eqref{YsCTHs} and \eqref{lin2} for the last contribution. 

For the fourth term in \eqref{uIB2}, we apply Lemma~\ref{LEM:BXsb} with $\s=1$. This imposes the second condition in $s+s_0>\frac 12+10\dl$ and $s\leq 1$.
Considering the second term in \eqref{uIB2}, using Sobolev's inequality and H\"older's inequality, for $s<\frac34$, we have
\begin{align}
\begin{split}
\|  u\QQ_{h}(|u|^2)]\|_{L^{2}_{T}H^{s-1}_x} \les
T^\frac12
\| \jb{\dx}^{-\frac 14}\big[ u\QQ_{h}(|u|^2)]\|_{L^{\infty}_{T} L^2_x}
& \les T^{\frac 12}  \|  u\QQ_{h}(|u|^2)]\|_{L^{\infty}_{T}L^{\frac 43}_x} \\
& \les T^{\frac 12} \|\QQ_h\|_{\text{op}}  \|u\|_{L^{\infty}_{T}L^{4}_x}^{3} \\
& \les T^{\frac 12}\|\QQ_{h}\|_{\text{op}}  \|u\|_{L^{\infty}_{T}H^{\frac 14+}_x}^{3}.
\end{split} \label{Lhterm}
\end{align} 

It remains to control the first term in \eqref{uIB2}, for which we apply the triangle inequality and estimate the terms coming from \eqref{goodbad}.
First, we have 
\begin{align*}
\|  (\P_{+}u)\P_{+}\dx(|u|^2)\|_{L^{2}_{T}H^{s-1}_x}   \les   \| J^{s_0}\PbHI u\|_{\wt{L^4_{T,x}}}  \| u\|_{L^\infty_T H^{s_0}_x}\| J^{s}\PbHI u\|_{\wt{L^4_{T,x}}}
\end{align*}
as the contribution is restricted to the same frequency regions as in \eqref{CCMfreq}, so we can proceed as in the proof of Lemma~\ref{LEM:CCMXsb}. 
 Note that this term imposes the spatial regularity $s-1$. 

Now we consider the second term in $G_1$ in \eqref{goodbad} which we write as
\begin{align*}
(\P_{-,\lo}u)\P_{+}\dx(|u|^2) = \dx[(\P_{-,\lo}u) \P_{+}(|u|^2)] -(\dx \P_{-,\text{lo}}u)\P_{+}(|u|^2).
\end{align*}
Then, by Bernstein's inequality, $\P_{-}$ on $L^p$, $1<p<\infty$,  and the fractional Leibniz rule,
\begin{align*}
&\| \dx[(\P_{-,\lo}u) \P_{+}(|u|^2)]\|_{L^{2}_{T}H^{s-1}_x}  \\
&\les \| \dx[(\P_{-,\lo}u) \P_{+,\LO}(|u|^2)]\|_{L^{2}_{T}H^{s-1}_x} + \| \dx[(\P_{-,\lo}u) \PbHIp(|u|^2)]\|_{L^{2}_{T}H^{s-1}_x} \\
& \les T^{\frac 12}\|u\|_{L^{\infty}_{T}H^{s_0}_x}^3 +  \| D^{s}\P_{-,\text{lo}}u\|_{L^{\infty}_{T,x}}\|\Pbhi u\|_{L^{4}_{T,x}}^2  + \|\P_{-,\text{lo}}u\|_{L^{\infty}_{T,x}} \| D^{s}\PbHI (|u|^2)\|_{L^{2}_{T,x}} \\
& \les  T^{\frac 12} \|u\|_{L^{\infty}_{T}H^{s_0}_{x}}^{3}+ \|u\|_{L^{\infty}_{T}H^{s_0}_x}\|J^{s_0}\Pbhi u\|_{L^{4}_{T,x}}\|J^{s}\Pbhi u\|_{L^{4}_{T,x}}.
\end{align*}
Similar arguments can be used to control the  other term $-(\dx \P_{-,\text{lo}}u)\P_{+}(|u|^2)$.

Now we consider the first term in $G_2$ in \eqref{goodbad}. Since $s\leq 1$, we have 
\begin{align*}
\| w\P_{+,\text{LO}} \dx(|u|^2)\|_{L^{2}_{T}H^{s-1}_x} 
&\les \| w\P_{+,\text{LO}} \dx(|u|^2)\|_{L^{2}_{T,x}}  \\
&\les \|w\|_{L^{2}_{T,x}} \|\PbLO \P_{+}\dx( |u|^2)\|_{L^{\infty}_{T, x}}\\ 
&\les \|w\|_{L^{2}_{T,x}} \|  |u|^2 \|_{L^{\infty}_{T} L^{1+}_x}\\ 
& \les \|w\|_{L^{2}_{T,x}} \|u\|_{L^{\infty}_{T}H^{s_0}_{x}}^2, 
\end{align*}
where we used Bernstein's and Sobolev inequality, given that 
 $s_0>0$. 

Finally we consider the terms $\mathcal{W}^{\pm}_{1}(w,u)$. We apply duality and further dyadic decompositions as in \eqref{dyadic1} and reduce to controlling
\begin{align}
\sum_{N\ges 1}\sum_{\substack{N_1,N_2,N_3,N_{23} \\ N\ges N_1}} N^{s-1} \int_{0}^{T}\int_{\R} \cj{\P_{N} g} \cdot \P_{\pm , \HI} \big[ \P_{N_1}w \,  \PbHIp \dx \P_{N_{23}}( \cj{\P_{N_2}u} \P_{N_3}u) \big] dxdt,   \label{dyadicW1}
\end{align}
where $g\in L^2_{T,x}$ with $\|g\|_{L^2_{T,x}}\leq 1$.
For the contribution with $\P_{+, \HI} $ in \eqref{dyadicW1} coming from $\mathcal{W}^{+}_{1}$, we have two cases: (i) $N\sim N_1 \sim N_{23}$ and (ii) $ N \sim N_{23}\gg N_1$. 
Notice that for the contribution with $\P_{-, \HI} $, we can only have $N \sim N_1 \ges N_{23}$. 

\smallskip
\noi $\bul$ \underline{\textbf{Case (i): $\W_1^{\pm}$ with $N\sim N_1 \ges N_{23}$.}}

\noi
This case follows immediately by the same argument as in Case 1 of the proof of Lemma~\ref{LEM:CCMXsb} giving rise to the bound 
\begin{align}
 \| J^{s_0}\PbHI u\|_{\wt{L^4_{T,x}}}  \| u\|_{L^\infty_T H^{s_0}_x}\| J^{s}\PbHI u\|_{\wt{L^4_{T,x}}}.
 \label{W1bd1}
\end{align}
Note that here we harmlessly replaced $w$ by $\P_{-,\text{hi}}u$.

\smallskip
\noi $\bul$ \underline{\textbf{Case (ii): $\W_1^+$ with $N\sim N_{23} \ges N_1$.}}

\noi
This case follows immediately by the same argument as in Case 2 of the proof of Lemma~\ref{LEM:CCMXsb} giving rise to the same bound as in \eqref{W1bd1}.

Combining the different cases, we have shown
\begin{align}
\begin{split}
\|G(u,w)\|_{L^2_{T}H^{s-1}_x} \les& \| J^{s_0}\PbHI u\|_{\wt{L^4_{T,x}}}  \| u\|_{L^\infty_T H^{s_0}_x}\| J^{s}\PbHI u\|_{\wt{L^4_{T,x}}} + T^{\frac 12}  \|u\|_{L^{\infty}_{T}H^{s_0}_x}^{3}.
\end{split} \label{Gbd}
\end{align}
which completes the proof. \qedhere
\end{proof}

\subsection{INLS regularity properties: more nonlinear estimates}

The main goal of this subsection is to establish the analogue of Lemma~\ref{LEM:uinfo} for INLS \eqref{INLS2}, as well as on the difference of solutions.

\begin{lemma}[Estimates for INLS]
\label{LEM:uinfoINLS}
Let $0<T \le 1$, $\frac 14 <s< \frac{3}{4}$, $0\leq T\leq 1$, and $u$ be a $H^{\infty}(\R)$-solution to \eqref{INLS2} on $[0,T]$. Then it holds that:
\begin{align}
\begin{split}
\| J^{s}\PbHI u\|_{\wt{L^{4}_{T,x}}} 
\les & T^{\frac 14-}(1+\|u\|_{L^{\infty}_{T}H^{\frac 14+}_x}^{2})   \|v\|_{X^{s,\frac 12+\dl}_{T}}  \\&+T^{\frac 14-}(1+\|u\|_{L^{\infty}_{T}H^{\frac 14+}_{x}}^2)\Big[\|w\|_{X^{s,\frac 12+\dl}_{T}} +
 \|u\|_{L^{\infty}_{T}H^{\frac 14+}_x}^3\Big]
,
\end{split}
\label{Jsu2}
\end{align}
where $(v,w)$ are the gauged variables in \eqref{gauge}. 
Moreover, there exist $\ta_1,\ta_2,\ta_3>0$ such that for any $M\in \N$ sufficiently large, it holds that
\begin{align}
\begin{split}
\| u\|_{L^{\infty}_{T}H^{s}_{x}} 
& \les 
\| u_0 \|_{H^{s}_x} 
+\|v\|_{X^{s,\frac 12+\dl} _{T}}+\|w\|_{X^{s,\frac 12+\dl} _{T}} +TM\|\QQ_h\|_{\textup{op}}\|u\|_{L^{\infty}_{T}H^{\frac 14+}_x}^3   \\
&+T^{\ta_1}M^{\ta_2}\Big\{  \|u\|_{L^{\infty}_{T}H^{\frac 14+}_x}^3
+ \|J^{\frac 14+}\PbHI u\|_{\wt{L^4_{T,x}}}^{3}  + \|w\|_{X^{\frac 14+,\frac 12+}_{T}}^{3} \Big\}
 \\
&+ M^{-\ta_3}\Big\{(1+\|u\|_{L^{\infty}_{T}H^{\frac 14+}_x}^2) (\|u\|_{L^{\infty}_{T}H^{\frac 14+}}+ \|w\|_{X^{s,\frac 12+\dl}})+\|v\|_{X^{s,\frac 12+\dl} _{T}})\Big\}, 
 \end{split}
\label{uHsbd2}
\end{align}
where $\QQ_h$ is as in \eqref{Qh}. 
\end{lemma}

\begin{proof}
The arguments we use here are the same as those we used for obtaining Lemma~\ref{LEM:uinfo} with one key difference: we no longer have the Hardy space assumption so we need new arguments to deal with terms that relied on this assumption. Thus, in the following, we detail the necessary changes to the proof of Lemma~\ref{LEM:uinfo}.

We first consider \eqref{Jsu2}. 
By the recovery formula \eqref{PbHIu}, we obtain \eqref{Jsu1} with the additional term $\| J^s \P_\HI w\|_{\wt{L^4_{T,x}}}$, which is easily estimated by \eqref{L4} to give
\begin{align}
\| J^{s}\PbHI u\|_{\wt{L^{4}_{T,x}}} 
&
\les 
\I+ \II+\III + T^{\frac14-} \|w \|_{X^{s, \frac12+\dl}_T}, 
\notag
\end{align}
where $\I, \II$, and $\III$ are as in \eqref{Jsu1}. The estimates for $\I$ and $\II$ in \eqref{CCMI} and \eqref{CCMII}, respectively, also apply in this setting, thus it remains to estimate $\III$, since we used the Hardy space assumption to establish \eqref{III4}. 
By writing
\begin{align}
\begin{split}
\III \leq& \|J^{s}\PbHIp [ e^{-i\be F}\P_{-,\text{hi}}(e^{i\be F} \PbHIp u)] \|_{\wt{L^{4}_{T,x} }}    \\
&   + \|J^{s}\PbHIp [ e^{-i\be F}\P_{-,\text{hi}}(e^{i\be F} \Pblo u)] \|_{\wt{L^{4}_{T,x} }}  +\|J^{s}\PbHIp [ e^{-i\be F}\P_{-,\text{hi}}(e^{i\be F}w)] \|_{\wt{L^{4}_{T,x} }},
\end{split} \label{Lpsigns}
\end{align}
we can proceed in exactly the same way as in Lemma~\ref{LEM:uinfo} to handle the contribution from the first two terms, since they have $u=\P_+u$.
Arguing as in \eqref{III1} with \eqref{III3}, we have
\begin{align}
\III^2
&
\les 
\| u\|_{L^\infty_T H^{\frac14+}_x}^4
\sum_N \bigg\{ \sum_{N_2 \les N } N^{s-\frac34}\| \P_{N_2 } \P_{-, \hi} (e^{i \be F} w ) \|_{L^4_T L^\infty_{x}}  \bigg\}^{2}
,
\notag
\end{align}
and
we now need a replacement for \eqref{III2}, which due to \eqref{III1a} amounts to adequately controlling  
\begin{align}
\|\P_{N_2}\P_{-,\text{hi}}( e^{i\be F} w)\|_{L^4_T L^{\frac{1}{\eps}}_{x}}. \label{Jsu3}
\end{align}
However, simply by Sobolev embedding and \eqref{L4} we have 
\begin{align*}
\eqref{Jsu3} \les \| w\|_{L^4_T L^{\frac{1}{\eps}}_{x}} \les \| J^{\frac 14+}w\|_{L^{4}_{T,x}}
\les T^{\frac14- } \|w\|_{X^{\frac14+, \frac12}_T}. 
\end{align*}
Then, combining the estimates above, for $s<\frac34$, we have 
\begin{align}
\III
\les T^{\frac14-} \| u \|^2_{L^\infty_T H^{\frac14+}_x} \big( \|u\|_{L^\infty_T L^2_x} +  \|w\|_{X^{\frac14+, \frac12}_T} \big), 
\notag
\end{align}
completing the proof of \eqref{Jsu2}.

We now consider \eqref{uHsbd2}. Fixing $M\in \N$ sufficiently large, we split $u=\P_{\leq M}u + \P_{>M}u$.
For the high frequency term $\P_{>M}u$, compared to CCM, we again have the extra term $\PbHI w$ which is controlled using \eqref{YsCTHs}, the contributions corresponding to $\I, \II$ in \eqref{Jsu1} are estimated analogously, and it remains to consider the term $\III$ which we decompose as in \eqref{Lpsigns}. For the contribution with $\P_{+,\text{HI}}u$ , we bound it exactly as we did in obtaining \eqref{III+}. For the contribution due to $w$, we have
\begin{align}
\|D^s\P_{>M}\P_{+} [ \Pbhi (e^{-i \be F})\P_{-,\text{hi}}(e^{i \be F} w)] \|_{L^{\infty}_{T}L^2_x} \les \|D^s [ \P_{\ges M} (e^{-i \be F})\P_{-,\text{hi}}(e^{i \be F} w)] \|_{L^{\infty}_{T}L^2_x} ,\label{IIIHs1}
\end{align}
because of the signs on the frequencies.
By the fractional Leibniz rule, \eqref{DeLinfty}, Bernstein and Sobolev embedding, we have 
\begin{align*}
\eqref{IIIHs1} & \les \|D^s\P_{\ges M}  e^{-i \be F}\|_{L^{\infty}_{T}L^{4}_x} \|\P_{-,\text{hi}}(e^{i\be F}w)\|_{L^{\infty}_{T}L^4_x} +\|\P_{\ges M}e^{-i \be F}\|_{L^{\infty}_{T,x}} \| D^{s}\P_{-,\text{hi}}(e^{i \be F}w)\|_{L^{\infty}_{T}L^2_x} \\
& \les M^{s-\frac 34} \|u\|_{L^{\infty}_{T}H^{\frac 14+}_x}^{2}  \|w\|_{L^{\infty}_{T}L^{4}_{x}}  +M^{-\frac 12}\|u\|_{L^{\infty}_{T}H^{\frac 14+}_x}^2 \|D^{s}w\|_{L^{\infty}_{T}L^2_x}  \\
& \les M^{s-\frac 34} \|u\|_{L^{\infty}_{T}H^{\frac 14+}}^2 \|w\|_{X^{s,\frac 12+\dl}_{T}}.
\end{align*}
Finally, when we have $\Pblo u $ in \eqref{Lpsigns}, we still have 
\begin{align}
\|D^s\P_{>M}\P_{+} [ \Pbhi (e^{-i \be F})\P_{-,\text{hi}}(e^{i \be F} \Pblo u)] \|_{L^{\infty}_{T}L^2_x} \les \|D^s [ \P_{\ges M} (e^{-i \be F})\P_{-,\text{hi}}(e^{i \be F} \Pblo u)] \|_{L^{\infty}_{T}L^2_x} ,\label{IIIHs12}
\end{align}
Then, we estimate exactly as we did for \eqref{IIIHs1} using Bernstein's inequality for the term $\Pblo u$ and we obtain
\begin{align*}
\eqref{IIIHs12} \les M^{s-\frac 34} \|u\|_{L^{\infty}_{T}H^{\frac 14+}_{x}}^{3}.
\end{align*}
This completes the estimate for the high frequency part.

For the low frequencies, we need to refine the argument for \eqref{JsuPLO}, and we use the decomposition of the nonlinearity described in Section~\ref{SEC:INLSdecomp}.  Namely, applying $\P_{\leq M}$ to both sides of \eqref{INLSn4} and recalling the definitions in \eqref{goodbad}, we have 
\begin{align*}
\P_{\leq M}[ u\dx\P_{+}(|u|^2)] &= \P_{\leq M} G(u,w) + \P_{\leq M} B(u,u, w).
\end{align*}
Then, by the Duhamel formula for \eqref{INLS2}, Lemma~\ref{LEM:linXsb},  \eqref{YsCTHs}, and  \eqref{Lhterm}, we have
\begin{align}
&\|\P_{\leq M} u\|_{L^{\infty}_{T}H^{s}_x} \notag\\
& \le \|\P_{\leq M}u_0\|_{H^{s}_x} 
+ 2 |\be| T^{\frac 12}\|\P_{\leq M} G(u,w)\|_{L^{2}_{T}H^{s}_x} 
  + 2 |\be|  \|\mathcal{I}[ \P_{\leq M}B(u,u,w)]\|_{L^{\infty}_{T}H^{s}_x} 
\notag
\\
&
\quad 
+T\|\P_{\leq M}(u\QQ_{h}(|u|^2))\|_{L^{\infty}_{T}H^{s}_x} \notag
\\
& \leq \|u_0\|_{H^{s}_x}  
+ C T^{\frac 12}M^{1+s-s_0}\|\P_{\leq M} G(u,w) \|_{L^{2}_{T}H^{s_0-1}_x} 
+ CT^{\dl}M^{\s+s-s_0} \|\P_{\leq M} B(u,u,w)\|_{X^{s_0-\s,-\frac 12+2\dl}_{T}} 
\notag
\\
&
\quad 
+ CTM \|\QQ_{h}\|_{\text{op}} \|u\|_{L^{\infty}_{T}H^{\frac 14+}_x}^3, 
\label{PbLOu}
\end{align}
for $\QQ_h$ as in \eqref{Qh}.
Now we notice that we have already controlled $G(u,w)\in L^2_{T}H^{s-1}_{x}$ in the course of proving Lemma~\ref{LEM:GXsb} and we simply apply \eqref{Gbd} with $s=s_0=\frac 14+$. For the term with $B(u,u, w)$, we apply Lemma~\ref{LEM:BXsb} with $s=s_0=\frac 14+$ and we take $\s=\frac 14+>0$. This completes the proof. 
\end{proof}

\begin{lemma}[Difference estimates] \label{LEM:diffests}
Let $\frac 14<s< \frac 34$, $0\leq T\leq 1$, and $u_j$ be $H^{\infty}(\R)$-solutions to \eqref{INLS2} on $[0,T]$ with initial data $u_j(0)=u_{0,j}$, $j=1,2$.  Let $F_j=F_j [u_j]$ be as defined in \eqref{F} and set $v_j= v_j [u_j]$ and $w_j= w_{j}[u_j]$ as in \eqref{gauge}. Define $U:=u_1-u_2$, $V:=v_1-v_2$, $W:=w_1-w_2$, and $R=R(u_1,u_2): = \|u_1\|_{L^{\infty}_{T}H^{\frac 14+}_{x}}+\|u_2\|_{L^{\infty}_{T}H^{\frac 14+}_{x}}$.
Then,
\begin{align}
\| J^{s}\PbHI U\|_{\wt{L^{4}_{T,x}}}  &\les   T^{\frac 14-}\big[1+Q(R)\big] \|V\|_{X^{s,\frac 12+}_{T}} +T^{\frac 14-}\| W\|_{X^{s,\frac 12+}_{T}}   
\notag
\\
& + T^{\frac 14-}Q(R)(1+ \|w_1\|_{X^{\frac 14+,\frac 12+}_{T}}+\|v_2\|_{X^{s,\frac 12+}_{T}})[ \|F_1-F_2\|_{L^{\infty}_{T,x}} + \|U\|_{L^{\infty}_{T} H^{\frac 14+}_{x}}], \notag
\end{align}
where $Q\geq 0$ is some polynomial of at least first order.
Moreover, there exist $\ta_0,\ta_1,\ta_2>0$ such that for any $M\in \N$ sufficiently large, we have the following
\begin{align}
\|& U\|_{L^{\infty}_{T}H^{s}} \notag\\
& \les  \|U(0)\|_{H^{s}_x} +\|\P_{\ges M}V\|_{X^{s,\frac 12+\dl}_{T}} \|F_{1}-F_{2}\|_{L^{\infty}_{T,x}} \notag\\
&
\quad 
+M^{-\ta_0}(1+R)^{3} \Big[ \|V\|_{X^{s,\frac 12+}_{T}} + \|W\|_{X^{s,\frac 12+}_{T}}\notag \\
&  \hphantom{XXXXXXXXX} + (Q(R)+\|v_2\|_{X^{s,\frac 12+}_{T}} +\|w_2\|_{X^{s,\frac 12+}_{T}} ) \big\{ \|U\|_{L^{\infty}_{T}H^{\frac 14+}_{x}} + \|F_1-F_2\|_{L^{\infty}_{T,x}}\big\}\Big] \notag\\
& 
\quad 
+T^{\dl}M^{\ta_1}\Big[\|W\|_{X^{s,\frac 12+\dl}_{T}}\big\{ \| J^{\frac 14+}u_1\|_{L^{4}_{T,x}} + T^{\frac 12}R^2\big\} \notag\\
&  \hphantom{XXXXXXX} + \|w_2\|_{X^{s,\frac 12+\dl}_{T}} \big\{ \|J^{\frac 14+}(u_1+u_2)\|_{L^{4}_{T,x}}\|J^{\frac 14+}U\|_{L^{4}_{T,x}} + T^{\frac 12}R \|U\|_{L^{\infty}_{T}H^{\frac 14+}_x}\big\} \Big]\notag\\
& 
\quad 
+ T^{\frac 12}M^{\ta_2} \Big\{ \|U\|_{L^{\infty}_{T}H^{\frac14+}_{x}} +\|W\|_{X^{s,\frac 12+\dl}_{T}} \Big[ \max_{j=1,2}\big( \|J^{\frac 14 +} u_j\|_{L^{4}_{T,x}}^2)+ R^2 \Big]\notag  \\
&  \hphantom{XXXXXXX} +(\max_{j=1,2} \|
J^{\frac 14 +}\PbHI u_j\|_{\wt{L^4_{T,x}}}^2 + R^2)\big[  \|
J^{s}\PbHI U\|_{\wt{L^4_{T,x}}} + \|\QQ_h\|_{\textup{op}} \|U\|_{L^{\infty}_{T}H^{s}_x} \big]\notag  \\
&  \hphantom{X} +\max_{j=1,2}(R\|u_j\|_{L^{\infty}_{T}H^{s}_x}  + \|
J^{\frac 14 +}\PbHI u_j\|_{\wt{L^4_{T,x}}}\|
J^{s}\PbHI u_j\|_{\wt{L^4_{T,x}}})\big[  \|
J^{\frac 14 +}\PbHI U\|_{\wt{L^4_{T,x}}} +\|U\|_{L^{\infty}_{T}H^{\frac 14 +}_x} \big]  \Big\}  \notag \\
&
\quad 
 +TM R^2\|\QQ_h\|_{\textup{op}} \|U\|_{L^{\infty}_{T}H^{\frac 14+}_{x}}, 
\label{uHsbddiff}
\end{align}
for $\QQ_h$ as in \eqref{Qh}.
\end{lemma}

\begin{proof}
Using \eqref{PbHIu}, we have 
\begin{align}
\PbHI U  =& \PbHIp[ (e^{i \be F_1}-e^{i \be F_2}) v_1] +  \PbHIp[ e^{i \be F_2}V] \label{PU1} \\
& +\PbHIp[ \Pbhi( e^{- i \be F_1}-e^{- i \be F_2})\Pblo(e^{i \be F_1}u_1)] \label{PU2} \\
& + \PbHIp[ \Pbhi( e^{- i \be F_2})\Pblo((e^{i \be F_1}-e^{i \be F_2})u_1)]  \label{PU3}  \\
& + \PbHIp[ \Pbhi( e^{- i \be F_2})\Pblo(e^{i \be F_2}U)]  \label{PU4}  \\
&  +\PbHIp[ (e^{- i \be F_1}-e^{- i \be F_2})\P_{-,\text{hi}}(e^{i \be F_1}u_1)]  \label{PU5} \\
&  +\PbHIp[ e^{- i \be F_2}\P_{-,\text{hi}}((e^{i \be F_1}-e^{i \be F_2})u_1)]  \label{PU6} \\
&  +\PbHIp[ e^{- i \be F_2}\P_{-,\text{hi}}(e^{i \be F_2} U)] \label{PU7}    \\
& + \PbHI W. \label{PU8}
\end{align}

The second part of \eqref{PU1} is exactly $\I$ in the proof of \eqref{Jsu}. and thus \eqref{CCMI} applies.
As for the first term in \eqref{PU1}, we follow the same argument except that we use 
\begin{align}
\dx( e^{i \be F_1}-e^{i \be F_2}) 
= (|u_1|^2-|u_2|^2)e^{i \be F_1}+|u_2|^2(e^{i \be F_1}-e^{i \be F_2}). 
 \label{expdiffs}
\end{align}
We always place the difference $e^{i \be F_1}-e^{i \be F_2}$ on the right-hand side into $L^{\infty}_{T,x}$. Then, from the mean value theorem, it is controlled by $\|F_1-F_2\|_{L^{\infty}_{T,x}}$. Thus, by slightly modifying the steps in \eqref{I1}, \eqref{I2}, and \eqref{I3}, taking into account \eqref{expdiffs}, and using \eqref{DeLinfty2}, we have
\begin{align}
\begin{split}
\|J^{s}\eqref{PU1}\|_{\wt{L^4_{t,x}}}  \les&T^{\frac 14-} (1+R^2) \|V\|_{X^{s,\frac 12+}_{T}}+T^{\frac 14-} \|u_2\|_{L^{\infty}_{T}H^{\frac 14+}_x}\|V\|_{X^{s,\frac 12+}_{T}}^2 \\
&+T^{\frac 14-}  \big[  R \|U\|_{L^{\infty}_{T}H^{\frac 14+}_x}  + \|u_2\|_{L^{\infty}_{T}H^{\frac 14+}_x}^2 \|F_1-F_2\|_{L^{\infty}_{T,x}} \big] \|v_2\|_{X^{s,\frac 12+}_{T}}.
\end{split} 
\notag
\end{align}

For \eqref{PU2}, we use \eqref{expdiffs} and get
\begin{align*}
\|J^{s}\eqref{PU2}\|_{\wt{L^4_{t,x}}} & \les T^{\frac 14} \big[  R \|U\|_{L^{\infty}_{T}H^{\frac 14+}_x}  + \|u_2\|_{L^{\infty}_{T}H^{\frac 14+}_x}^2 \|F_1-F_2\|_{L^{\infty}_{T,x}} \big] \|u_1\|_{L^{\infty}_{T}L^2_x}.
\end{align*}
Similarly for \eqref{PU3} and \eqref{PU4}, we replace \eqref{Linftyu} by
\begin{align*}
\| \P_{+,\text{lo}}( e^{i \be F_1}u_1-e^{i \be F_2}u_2)\|_{L^{\infty}_{T,x}} \les \|e^{i \be F_1}-e^{i \be F_2}\|_{L^{\infty}_{T,x}} \|u_1\|_{L^{\infty}_{T}L^2_x} + \|U\|_{L^{\infty}_{T}L^2_x},
\end{align*}
and so
\begin{align*}
\|J^{s}\eqref{PU3}\|_{\wt{L^4_{t,x}}} +\|J^{s}\eqref{PU4}\|_{\wt{L^4_{t,x}}}  \les T^{\frac 14}\|u_1\|_{L^{\infty}_{T}H^{\frac 14+}_x}^2 \big[ \|F_1-F_2\|_{L^{\infty}_{T,x}} \|u_1\|_{L^{\infty}_{T}L^2_x} + \|U\|_{L^{\infty}_{T}L^2_x}\big].
\end{align*}

For \eqref{PU5}, \eqref{PU6}, and \eqref{PU7}, we need to split these three contributions according to \eqref{Lpsigns}. For the contributions coming from the good $+$ sign or the low frequency $\Pblo$, we just use the following replacement estimates for \eqref{III2} and \eqref{III3}:
\begin{align*}
\|\P_{N_2}\P_{-,\text{hi}}&( (e^{i \be F_1}-e^{i \be F_2})u\|_{L^{4}_{x}}\les N_{2}^{\eps}\|F_1-F_2\|_{L^{\infty}_x}\|u\|_{L^2_x}  \\
&+ N_{2}^{-\frac 14+\eps} \|u\|_{H^{\frac 14+}}\big[ R \|U\|_{H^{\frac 14+}_x}  + \|u_2\|_{L^{\infty}_{T}H^{\frac 14+}_x}^2 \|F_1-F_2\|_{L^{\infty}_{T,x}} \big], \\
\|\wt{\P}_{N}(e^{- i \be F_1}-e^{- i \be F_2})\|_{L^{4}_x} &\les N^{-\frac 34} \big[ R \|U\|_{H^{\frac 14+}_x}  + \|u_2\|_{H^{\frac 14+}_x}^2 \|F_1-F_2\|_{L^{\infty}_{x}} \big].
\end{align*}
Thus, 
\begin{align*}
&\|J^s \PbHIp[ (e^{- i \be F_1}-e^{- i \be F_2})\P_{-,\text{hi}}(e^{i \be F_1} (\Pbhip +\Pblo)u_1)] \|_{\wt{L^4_{T,x}}} \\ 
&+\|J^s\PbHIp[  e^{- i \be F_2}\P_{-,\text{hi}}((e^{- i \be F_1}-e^{- i \be F_2})(\Pbhip +\Pblo) u_1)] \|_{\wt{L^4_{T,x}}} \\
&\hphantom{XXX}\les T^{\frac 14}\|u_1\|_{L^{\infty}_{T}H^{\frac 14+}_x}(1+\|u_1\|_{L^{\infty}_{T}H^{\frac 14+}_x}^3)
 \\
&\hphantom{XXX}\times  \big[  R \|U\|_{L^{\infty}_{T}H^{\frac 14+}_x}  + \|u_2\|_{L^{\infty}_{T}H^{\frac 14+}_x}^2 \|F_1-F_2\|_{L^{\infty}_{T,x}} \big], 
\end{align*}
Next we follow the argument we used to deal with the contribution from $\P_{-,\text{hi}}u=w$ in \eqref{Lpsigns}. We have 
\begin{align*}
\|J^{s}\PbHIp &[ (e^{-i \be F_1}-e^{- i \be F_2})\P_{-,\text{hi}}(e^{i F_1}w_1)] \|_{\wt{L^{4}_{T,x} }}  \\
&\les T^{\frac 14-}   \|w_1\|_{X^{\frac 14+,\frac 12+}_{T}} \big[  R \|U\|_{L^{\infty}_{T}H^{\frac 14+}_x}  + \|u_2\|_{L^{\infty}_{T}H^{\frac 14+}_x}^2 \|F_1-F_2\|_{L^{\infty}_{T,x}} \big]
\end{align*}
and
\begin{align*}
\|&J^{s}\PbHIp [ e^{- i \be F_2}\P_{-,\text{hi}}((e^{i F_1}-e^{i \be F_2})w_1)] \|_{\wt{L^{4}_{T,x} }} \\
& \les T^{\frac 14}   \|w_1\|_{X^{\frac 14+,\frac 12+}_{T}} \big[  R \|U\|_{L^{\infty}_{T}H^{\frac 14+}_x}  +(1+ \|u_2\|_{L^{\infty}_{T}H^{\frac 14+}_x}^2)\|u_2\|_{L^{\infty}_{T}H^{\frac 14+}_x} \|F_1-F_2\|_{L^{\infty}_{T,x}} \big].
\end{align*}
This completes the estimates for \eqref{PU5} and \eqref{PU6}. For \eqref{PU7}, we just keep track of the dependence on $U$.  Finally, the bound for \eqref{PU8} follows from \eqref{L4}. 

We move onto verifying \eqref{uHsbddiff}, where we once again split $U=\P_{\leq M}U+\P_{>M}U$ for any  fixed $M\in \N$ sufficiently large.
We begin with the high frequency portion $\P_{>M} U$.  For the second term in \eqref{PU1}, we write it as 
\begin{align*}
\P_{>M}[ & e^{i \be F_2}V] = \P_{>M}[ \Pblo (e^{i \be F_2})V]  + \P_{>M}[ \Pbhi(e^{i \be F_2})V]  \\
& = \P_{>M}[ \Pblo (e^{i \be F_2}) \P_{\ges M}V] +\P_{>M}[ \Pbhi \P_{\ll M} (e^{i \be F_2}) \P_{\ges M}V] +\P_{>M}[ \Pbhi \P_{\ges M} (e^{i \be F_2})V]
\end{align*}
where for the first term in the second equality we applied an argument similar to term $\I$ in \eqref{Jsu1} to place $\P_{\ges M}$ onto $V$\footnote{This step is not necessary for the second term in \eqref{PU1} but is needed for the first term in \eqref{PU1}.}. Then, by the fractional Leibniz rule and \eqref{DeLinfty}, we have
\begin{align*}
\|\P_{>M}[ (\Pblo + \Pbhi \P_{\ll M}) (e^{i \be F_2}) \P_{\ges M}V] \|_{L^{\infty}_{T}H^s_x} \les  \|\P_{\ges M} V\|_{L^{\infty}_{T}H^s_x} \les \|\P_{\ges M}V\|_{X^{s,\frac 12+\dl}_{T}}.
\end{align*}
Next, 
\begin{align}
\|&\P_{>M}[ \P_{\ges M} (e^{i \be F_2}) V] \|_{L^{\infty}_{T}H^{s}_x}  \les \|D^{s}\P_{\ges M}e^{i \be F_2}\|_{L^{\infty}_{T}L^4_x} \|V\|_{L^{\infty}_{T}L^4_x} + \|\P_{\ges M}e^{i \be F_2}\|_{L^{\infty}_{T,x}} \|V\|_{L^{\infty}_{T}H^{s}_x}  \notag\\
& \les \|D^{s-1}\P_{\ges M}(|u_2|^2e^{i \be F_2})\|_{L^{\infty}_{T}L^4_x} \|V\|_{X^{\frac 14+,\frac 12+}_{T}} +M^{-\frac 12+}\| |u_2|^2 e^{i \be F_2}\|_{L^{\infty}_{T}L^2_x} \|V\|_{L^{\infty}_{T}H^s_x}  \notag\\
& \les M^{-\frac 14} \||u_2|^2 \|_{L^{\infty}_{T}L^2_x}\|V\|_{X^{\frac 14+,\frac 12+}_{T}}+M^{-\frac 12+}\| u_2\|_{L^{\infty}_{T}L^4_x}^2 \|V\|_{L^{\infty}_{T}H^s_x} \notag \\
& \les M^{-\frac 14}\|u_2\|_{L^{\infty}H^{\frac 14+}_x}^2 \|V\|_{X^{s,\frac 12+}_{T}}.\label{Mneg1}
\end{align}
Therefore, 
\begin{align*}
\|J^s \P_{>M}[ e^{i \be F_2}V]\|_{L^{\infty}_{T}L^2_x} & \les \|\P_{\ges M}V\|_{X^{s,\frac 12+\dl}_{T}} +  M^{-\frac 14}\|u_2\|_{L^{\infty}H^{\frac 14+}_x}^2 \|V\|_{X^{s,\frac 12+}_{T}}.
\end{align*}
Following a similar strategy additionally using \eqref{DeLinfty2} and \eqref{expdiffs}, we also get
\begin{align*}
\|&\P_{>M}[ (e^{i \be F_1}-e^{i \be F_2}) v_1]\|_{L^{\infty}_{T}L^2_x}  \\
& \les 
\big\{\|\P_{\ges M}v_1\|_{X^{s,\frac 12+\dl}_{T}} +  M^{-\frac 14}\|u_2\|_{L^{\infty}H^{\frac 14+}_x}^2 \|V\|_{X^{s,\frac 12+}_{T}} \big\}\big\{ \|F_1-F_2\|_{L^{\infty}_{T,x}} +\|U\|_{L^{\infty}_{T}H^{\frac 14+}_{x}}\big\}.
\end{align*}
 For each of \eqref{PU2}-\eqref{PU4}, we can place a projection $\P_{\ges M}$ onto the first factor and gain negative powers of $M$ as we did in \eqref{Mneg1}. 
Following \eqref{PU234} and using \eqref{expdiffs} we have 
\begin{align*}
\| J^{s}[ \eqref{PU2}+\eqref{PU3}+\eqref{PU4}]\|_{L^{\infty}_{T}L^2_x} & \les M^{-\frac 14+}R^2\big\{ \|U\|_{L^{\infty}_{T}H^{\frac 14+}_x}+R\|F_1-F_2\|_{L^{\infty}_{T,x}}\big\}.
\end{align*}
For \eqref{PU5}, \eqref{PU6}, and \eqref{PU7}, we need to split each of these terms into the three parts as in \eqref{Lpsigns}. Note though that due to the signs, we can always place a projection $\P_{\ges M}$ onto the first factor which is the smoother allowing us to gain negative powers of $M$ in the estimates.
For the first two such parts, we argue as we did in \eqref{PU5671} with \eqref{expdiffs}. Here, we obtain the bound by 
\begin{align*}
M^{-\ta}R^2(1+R^2) \{ \|U\|_{L^{\infty}_{T}H^{\frac 14+}_x}+ R\|F_1-F_2\|_{L^{\infty}_{T,x}}\big\},
\end{align*}
for some $\ta>0$.
We then need to control the corresponding third part in \eqref{Lpsigns}, which we do by modifying the argument we used for \eqref{IIIHs1}. This gives the bound
\begin{align*}
M^{-\ta}(1+R^2)\big[ \big\{ \|U\|_{L^{\infty}_{T}H^{\frac 14+}_x}+\|F_{1}-F_2\|_{L^{\infty}_{T,x}}\} \|w_2\|_{X^{s,\frac 12+}_{T}} + \|W\|_{X^{s,\frac 12+}_{T}}\big].
\end{align*}
Therefore, for $\PbHI U$, we have obtained
\begin{align*}
\|\P_{>M} U \|_{L^{\infty}_{T}H^{\frac 14+}_{x}} 
& \les \|\P_{\ges M}V\|_{X^{s,\frac 12+\dl}_{T}} \|F_{1}-F_{2}\|_{L^{\infty}_{T,x}}+
(1+M^{-\ta}R)^{3} \Big[ \|V\|_{X^{s,\frac 12+}_{T}} + \|W\|_{X^{s,\frac 12+}_{T}} 
\\
& \hphantom{XX}+M^{-\ta} (R+\|v_2\|_{X^{s,\frac 12+}_{T}} +\|w_2\|_{X^{s,\frac 12+}_{T}} ) \big\{ \|U\|_{L^{\infty}_{T}H^{\frac 14+}_{x}} + \|F_1-F_2\|_{L^{\infty}_{T,x}}\big\}\Big]
\end{align*}
Now we consider the low frequencies $\P_{\leq M} U$ for which we take the difference of the Duhamel formulas and have the following analogue of \eqref{PbLOu}:
\begin{align*}
\|\P_{\leq M} U\|_{L^{\infty}_{T}H^{s}_x}  \leq & \|U(0)\|_{H^{s}_x} + T^{\frac12}M^{1+s-s_0}\|G(u_1,w_1)-G(u_2,w_2)\|_{L^2_{T}H^{s_0-1}_x}      \\
& + CT^{\dl}M^{\s+s-s_0}\|B(u_1,w_1)-B(u_2,w_2)\|_{X^{s_0-\s,-\frac 12+2\dl}_{T}}\\
&  + CTM (1+\|\GG_h\|_{\text{op}}) R^{2} \|U\|_{L^{\infty}_{T}H^{\frac 14+}_{x}}.
\end{align*}
By Lemma~\ref{LEM:BXsb}, we have
\begin{align*}
\| & B(u_1,w_1)-B(u_2,w_2)\|_{X^{s-\s, -\frac 12+2\dl}_{T}}  \\
& \les \|W\|_{X^{s,\frac 12+\dl}_{T}}\big\{ \| J^{\frac 14+}u_1\|_{L^{4}_{T,x}} + T^{\frac 12}R^2\big\} \\
&  \hphantom{X} + \|w_2\|_{X^{s,\frac 12+\dl}_{T}} \big\{ \|J^{\frac 14+}(u_1+u_2)\|_{L^{4}_{T,x}}\|J^{\frac 14+}U\|_{L^{4}_{T,x}} + T^{\frac 12}R \|U\|_{L^{\infty}_{T}H^{\frac 14+}_x}\big\}.
\end{align*}
Lastly, from \eqref{uIBdiff}, we get
\begin{align*}
\|& G(u_1,w_1)-G(u_2,w_2)\|_{L^{2}_{T}H^{s_0-1}_x} \\
&\les \|U\|_{L^{\infty}_{T}H^{s_0}_{x}} +\|W\|_{X^{s,\frac 12+\dl}_{T}} \Big[ \max_{j=1,2}\big( \|J^{s_0} u_j\|_{L^{4}_{T,x}}^2)+ R^2 \Big]  \\
& \hphantom{X} +(\max_{j=1,2} \|
J^{s_0}\PbHI u_j\|_{\wt{L^4_{T,x}}}^2 + R^2)\big[  \|
J^{s}\PbHI U\|_{\wt{L^4_{T,x}}} +(1+\|\GG_h\|_{\textup{op}})\|U\|_{L^{\infty}_{T}H^{s}_x} \big]  \\
& \hphantom{X} +\max_{j=1,2}(R\|u_j\|_{L^{\infty}_{T}H^{s}_x}  + \|
J^{s_0}\PbHI u_j\|_{\wt{L^4_{T,x}}}\|
J^{s}\PbHI u_j\|_{\wt{L^4_{T,x}}})\big[  \|
J^{s_0}\PbHI U\|_{\wt{L^4_{T,x}}} +\|U\|_{L^{\infty}_{T}H^{s_0}_x} \big].
\end{align*}
Putting all of these terms together finishes the proof of \eqref{uHsbddiff}.
\end{proof}

\section{Trilinear estimates}\label{SEC:tri}

We now state the crucial trilinear estimates.

\begin{proposition}[Trilinear estimates for CCM \eqref{CCM}]\label{PROP:tri}
Let $s\geq s_0 >\frac 14$ and $0<T\leq 1$. Then, for any $u_j\in L^{\infty}_{T}H^{ s_0}_{+} \cap L^{4}_{T}W^{s_0,4}_x \cap  X_{T}^{ s_0-\frac 12,\frac 12}\cap X_{T}^{ s_0-1,1} $, $j=2,3$, with
\begin{align}
Y(u_2, u_3) : =   &\prod_{j=2}^3 \big(  \|J^{s_0}\PbHI u_j\|_{L^{4}_{T,x}} + \| u_j\|_{L^\infty_T H^{s_0}_x} \big)
\notag
\\
& + \prod_{\substack{j_1,j_2 \in \{2,3\} \\ j_1 \ne j_2}} \|u_{j_1} \|_{X^{s_0-\frac12, \frac12}_T} \big( \| u_{j_2} \|_{X^{s_0-1,1}_T} + \|u_{j_2} \|_{L^\infty_T H^{s_0}_x} \big)
\label{Bu}
\end{align}
and for any $v\in X_{T}^{s,\frac 12+\dl}$, it holds that:
\begin{align}
\| \ind_{[0,T]} \Pbhip[ v \, \P_{-}\dx( \cj{u_2} u_3)] \|_{X^{s,-\frac 12+\dl}} & \les T^{\dl} \|v\|_{X_{T}^{s,\frac12+\dl}} Y(u_2, u_3) .\label{vX0}
\end{align}
\end{proposition}

\begin{proof}
By passing from $-\frac 12+\dl$ to $-\frac 12+2\dl$ in the modulation variable, we may gain the factor of $T^{\dl}$ appearing on the right-hand side of \eqref{vX0}. Moreover, we consider arbitrary extensions of the functions $v$ and $u_2,u_3$ on $[0,T]$ and take an infimum over all extensions at the end of the estimates to recover the time localised norms appearing in \eqref{Bu} and \eqref{vX0}. Moreover, we either associate the sharp cutoff function $\ind_{[0,T]}$ with the dual function (when we use duality) or with $v$ (just for the bound for \eqref{X0}). We omit these details to not overburden the notation.

We perform a dyadic decomposition in space:
\begin{align}
\| &\Pbhip[ v \, \P_{-}\dx(\cj{u_2} u_3)] \|_{X^{s,-\frac 12+2\dl}}^{2}  \notag \\
 & \les \sum_{N>1}  N^{2s} \bigg( \sum_{N_1,N_2,N_3, N_{23}}\| \P_{N}\Pbhip[ \P_{N_1}v\cdot  \P_{-}\P_{N_{23}}\dx( \cj{\P_{N_2}u_2} \P_{N_3}u_3)] \|_{X^{0,-\frac 12+2\dl}}\bigg)^{2}
.
 \label{Xdyad}
\end{align}
The projections $\P_{+}$ and $\P_{-}$ here imply that:
\begin{align}
N_{1} \ges N_{23} \vee N. \label{N1cond}
\end{align}
In particular, $N_1 \ges 2$.
Thus, we control \eqref{Xdyad} under the additional restriction \eqref{N1cond}. We let 
\begin{align*}
N_{\max}: = \max(N_1,N_2,N_3,N_{23},N)
\end{align*}
and note that 
\begin{align}
N_{23}\les N_2 \vee N_3. \label{N23}
\end{align}

\smallskip
\noi
\underline{\textbf{Case 0:} $N_{23}\les 1$.}

\smallskip
\noi
By duality, we have
\begin{align}
\begin{split}
\|& \P_{N}\Pbhip[ \P_{N_1}v \cdot \P_{-}\P_{N_{23}}\dx( \cj{\P_{N_2}u_2} \P_{N_3}u_3)] \|_{X^{0,-\frac 12+2\dl}}  \\
&  = \sup_{ \|g\|_{X^{0,\frac 12-2\dl}} \leq 1} \int_{\R} \int_{\R} \cj{\P_{N}g} \cdot \Pbhip[ \P_{N_1}v \cdot \P_{-}\P_{N_{23}}\dx( \cj{\P_{N_2}u_2} \P_{N_3}u_3)]dx dt.
\end{split} \label{X1dual}
\end{align}
For fixed $g\in X^{0,\frac 12-2\dl}$ with $\|g\|_{X^{0,\frac 12-2\dl}}\leq 1$, we use H\"{o}lder's inequality and \eqref{L4} to bound the integrals on the right-hand side of \eqref{X1dual} by
\begin{align*}
&\| \P_{N}g\|_{L^{4}_{t,x}} \|\P_{N_1}v\|_{L^{4}_{t,x}} \|\P_{N_2}u_2\|_{L^{4}_{t,x}} \|\P_{N_3}u_3\|_{L^{4}_{t,x}} \\
&\les N_1^{-s} (N_2 N_3)^{-s_0}   \| \P_{N}g\|_{X^{0,\frac 12-2\dl}} \|\P_{N_1}v\|_{X^{s,\frac 12}} \|J^{s_0}\P_{N_2}u_2\|_{L^{4}_{t,x}} \|J^{s_0}\P_{N_3}u_3\|_{L^4_{t,x}}.
\end{align*}
If $N_2 \vee N_3\ges N_1$, then we use $N_{1}^{-s}$ to control the norm derivative $N^s$ and we have a negative power of the maximum frequency to perform all of the dyadic summations. If instead $N_{2}\vee N_3 \ll N_1$, then $N\sim N_1$. Since $s_0>0$, we can sum over $(N_2,N_3)$. As for the sum over $(N,N_1)$, we have by Cauchy-Schwarz,
\begin{align*}
& 
\sum_{N} N^{2s}  \bigg( \sum_{N_1\sim N}N_1^{-s} \|\P_{N}g\|_{X^{0,\frac 12-2\dl}} \|\P_{N_1}v\|_{X^{s,\frac 12}}\bigg)^2
\\
& \les \sum_{N} N^{2s} \|\P_{N}g\|_{X^{0,\frac 12-2\dl}}^2  \bigg(\sum_{N_1\sim N}N_1^{-2s}\bigg)  \|v\|_{X^{s,\frac 12}}^2 \\
& \les \sum_{N}  \|\P_{N}g\|_{X^{0,\frac 12-2\dl}}^2 \|v\|_{X^{s,\frac 12}}^2 \\
& \les \|g\|_{X^{0,\frac 12-2\dl}}^2 \|v\|_{X^{s,\frac 12}}^2.
\end{align*}

For the rest of the proof we assume that $N_{23}\gg 1$.

\smallskip
\noi
\underline{\textbf{Case 1:} $N_2 \vee N_3 \sim N_2 \wedge N_3$.}

\smallskip
\noi
We use \eqref{X1dual}.
By H\"{o}lder's inequality and \eqref{bilin1}, we have 
\begin{align}
\begin{split}
 &\bigg| 
\int_{\R} \int_{\R} \cj{\P_{N}g} \cdot \Pbhip[ \P_{N_1}v \cdot \P_{-}\P_{N_{23}}\dx( \cj{\P_{N_2}u_2} \P_{N_3}u_3)]dx dt
\bigg| \\
 & \les N_{23} \big\| \P_{N_{23}}[\cj{\P_{N}  g} \cdot \P_{N_1} v]\|_{L^{2}_{t,x}} \| \P_{N_2}u_2\|_{L^{4}_{t,x}} \|\P_{N_3}u_3\|_{L^{4}_{t,x}}  \\
 & \les \frac{N_{23}^{\frac 12+10\dl}}{N_{1}^{s} N_{2}^{s_0}N_{3}^{s_0}}  \|\P_{N} g\|_{X^{0,\frac 12 -2\dl}}\| \P_{N_1}v\|_{X^{s,\frac 12 -2\dl}} \|J^{s_0} \P_{N_2}u_2\|_{L^{4}_{t,x}} \| J^{s_0} \P_{N_3}u_3\|_{L^{4}_{t,x}}.
 \end{split} \label{L4arg}
\end{align}
Now, considering the dyadic factors and using \eqref{N1cond} and \eqref{N23}, we have
\begin{align}
\frac{N^s N_{23}^{\frac 12+10\dl}}{N_1^s (N_2 N_3)^{s_0}}. \label{L4mult}
\end{align}
In this case, $N_2\sim N_3$ and with \eqref{N1cond}, we must have either (i) $N_1\sim N$ or (ii) $N\ll N_1$. In case (i), by \eqref{N23} we have
\begin{align*}
\eqref{L4mult} \les N_{2}^{-2s_0+\frac 12 +10\dl}
\end{align*}
which is a negative power provided that $s_0>\frac 14+5 \dl$,
and we can sum over $(N,N_1)$ using a Cauchy-Schwarz argument as in Case 0.

In case (ii), since $N \ll N_1 \sim N_{23} \les N_2 \sim N_3$, we have
\begin{align*}
\eqref{L4mult}\les \frac{N^{s}}{N_{1}^{s} N_{2}^{2s_0-\frac 12-10\dl}} \les  N_{\max}^{-2s_0+\frac 12 +10\dl}
\end{align*}
which allows us to perform all of the dyadic sums since $s_0>\frac 14+5\dl$.

\smallskip
\noi
\underline{\textbf{Case 2:} $N_2\vee N_3 \gg N_2\wedge N_3$ and $N_2 \wedge N_3 \ges N$.}

\smallskip
\noi
In this case, we can also proceed by using duality, \eqref{bilin1} and \eqref{L4arg}, and we are reduced to adequately controlling the multiplicative factor \eqref{L4mult}. Note that we must have $N_1\sim N_2 \vee N_3\sim N_{23} \sim N_{\max}$.
 By \eqref{N1cond}, we have
\begin{align*}
\eqref{L4mult} \les N^{s-s_0}N_{23}^{\frac 12+10\dl}N_{1}^{-s-s_0} \les N^{s-s_0} N_{1}^{\frac 12+10\dl-s-s_0} \les N_{\max}^{-2s_0+\frac 12+10\dl},
\end{align*}
where in the second (and third) inequality we used that $s\geq s_0$ and $s_0>\frac 14+5\dl$.
We can then perform all of the dyadic summations since $\dl>0$.

\smallskip
\noi
\underline{\textbf{Case 3:} $N_2\vee N_3 \gg N_2\wedge N_3$ and $N_2 \wedge N_3 \ll N$.}

\smallskip
\noi
In the remaining cases, we need to make use of the phase function.
Given $\tau, \tau_j, \xi,\xi_j \in \R$, with $\s := \tau - \xi^2$ and $\s_j := \tau_j - \xi_j^2$, $j=1,2,3$, satisfying $\tau = \tau_1 - \tau_2 + \tau_3$ and $\xi = \xi_1 - \xi_2+\xi_3$, we have the following resonance identity
\begin{align}
\s_1-\s_2 + \s_3 -\s= \xi_{1}^{2}-\xi_{2}^{2}+\xi_{3}^{2} -\xi^{2} = -2(\xi-\xi_1)(\xi-\xi_3) =: \Phi(\cj{\xi}). \label{mods}
\end{align}

\noi 
Note that under $\xi_1-\xi_2+\xi_3=\xi$, we have $\xi-\xi_1=\xi_3-\xi_2$.
Thus, for $|\xi_3-\xi_2| \sim N_{23} \gg 1$, we have
\begin{align}
\s_{\max}:=\max( |\s_1|,|\s_2|,|\s_3|,|\s|) \ges N_{23}|\xi-\xi_3|.  \label{nonres}
\end{align}
Moreover, from the frequency assumptions, it holds that
\begin{align}
N_{23} \sim N_2\vee N_3 \quad \text{and} \quad N_{1} \sim N_{\max}. \label{N1max}
\end{align}

\smallskip
\noi
For notational purposes, we define 
\begin{align}
K : = N_{23} (N\vee N_3). \label{K}
\end{align}
We also write $N_{(2)}= N_2 \vee N_3$ and $N_{(3)}=N_2 \wedge N_3$, and we extend the $(j)$ notation to denote the corresponding index on the functions $\{u_2,u_3\}$ and set $\iota_{(j)}$ to be define the conjugation operation if $(j)=2$, or to be the identity operation if $(j)=3$.

For fixed $(N,N_1,N_2,N_3, N_{23})$, the triangle inequality implies
\begin{align}
&\| \P_{N}\Pbhip[ \P_{N_1}v \, \P_{-}\P_{N_{23}}\dx( \cj{\P_{N_2}u_2}  \cdot \P_{N_3}u_3)] \|_{X^{0,-\frac 12+2\dl}} \notag \\
& \les \| \Q_{\ges K} \P_{N}\Pbhip[ \P_{N_1}v \cdot \P_{-}\P_{N_{23}}\dx( \cj{\P_{N_2}u_2} \,\P_{N_3}u_3)] \|_{X^{0,-\frac 12+2\dl}} \label{X0}  \\
&  + \| \Q_{\ll K} \P_{N}\Pbhip[ (\P_{N_1}\Q_{\ges K}v)  \P_{-}\P_{N_{23}}\dx( \cj{\P_{N_2}u_2}\, \P_{N_3}u_3)] \|_{X^{0,-\frac 12+2\dl}} \label{X1} \\ 
&  + \| \Q_{\ll K} \P_{N}\Pbhip[ (\P_{N_1}\Q_{\ll K}v)  \P_{-}\P_{N_{23}}\dx(  [\Q_{\ges K}  \P_{N_{(3)}}u_{(3)}]^{\iota_{(3)}} \,[\P_{N_{(2)}}u_{(2)}]^{\iota_{(2)}})] \|_{X^{0,-\frac 12+2\dl}} \label{X2} \\
&+ \| \Q_{\ll K} \P_{N}\Pbhip[ (\P_{N_1}\Q_{\ll K}v)  \P_{-}\P_{N_{23}}\dx([\Q_{\ll K}  \P_{N_{(3)}}u_{(3)}]^{\iota_{(3)}}\, [\Q_{\ges K}  \P_{N_{(2)}}u_{(2)}]^{\iota_{(2)}}] \|_{X^{0,-\frac{1}{2}+2\dl}} \label{X3}\\
& + \| \Q_{\ll K} \P_{N}\Pbhip[ (\P_{N_1}\Q_{\ll K}v)  \P_{-}\P_{N_{23}}\dx( \cj{ \Q_{\ll K}  \P_{N_2}u_2}\, \Q_{\ll K}\P_{N_3}u_3)] \|_{X^{0,-\frac 12+2\dl}}, \label{Xres}
\end{align}
where $\Q_{\ges K}$ and $\Q_{\ll K}$ are the projections in \eqref{Qpro}. 

The terms \eqref{X0} through \eqref{X3} are the non-resonant contributions while \eqref{Xres} is the nearly-resonant contribution. As $\P_{+}u=u$ for \eqref{CCM}, we must have that $N_2 \les N_3$. However, we choose to proceed without this extra information since it will not be the case for \eqref{INLS2}.

We first control the non-resonant contributions.

\smallskip
\noi
\underline{Bound for \eqref{X0}:}

\smallskip
\noi
 By \eqref{nonres}, \eqref{L4}, and Bernstein's inequality,
\begin{align*}
\eqref{X0} & \les K^{-\frac 12+2\dl} \| \P_{N_1}v \P_{-}\P_{N_{23}}\dx[ \cj{\P_{N_2}u_2} \P_{N_3}u_3]\|_{L^{2}_{t,x}}  \\
& \les K^{-\frac 12+2\dl} \|\P_{N_1}v\|_{L^{4}_{t,x}} \| \P_{-}\P_{N_{23}}\dx[ \cj{\P_{N_2}u_2} \P_{N_3}u_3]\|_{L^{4}_{t,x}} \\
& \les K^{-\frac 12+2\dl} N_{23}N_{1}^{-s} \|\P_{N_1}v\|_{X^{s,\frac 12-\dl}} \|\P_{N_2}u_2 \cdot \P_{N_3}u_3\|_{L^{4}_{t,x}} \\
& \les \frac{N_{23} (N_2 \wedge N_3)^{\frac 12-s_0}}{K^{\frac 12-2\dl}N_{1}^s (N_2 \vee N_3)^{s_0} }
  \|\P_{N_1}v\|_{X^{s,\frac 12}}\| J^{s_0}\P_{N_{(2)}}u_{(2)}\|_{L^{4}_{t,x}} \| \P_{N_{(3)}} u_{(3)}\|_{L^{\infty}_{t}H^{s_0}_{x}}.
\end{align*}

\noi 
We reduce to bounding the ensuing multiplier:
\begin{align}
\frac{N^s N_{23} (N_2 \wedge N_3)^{\frac 12-s_0}}{K^{\frac 12-2\dl}N_{1}^s (N_2 \vee N_3)^{s_0} } 
 & \sim \frac{N^s  N_{23}^{\frac 12+2\dl}(N_2 \wedge N_3)^{\frac 12-s_0}}{  (N\vee N_3)^{\frac{1}{2} - 2\dl} N_1^s (N_2 \vee N_3)^{s_0}} \label{X0mult}.
\end{align}

\noi 
If $N\sim N_1$, then
\begin{align*}
\eqref{X0mult} \les N_{\max}^{-\frac 12 + 2\dl} (N_2 \vee N_3)^{\frac12 + 2\dl - s_0 + (\frac12-s_0)\lor 0} \les N_{\max}^{-s_0 + 4\dl + (\frac12-s_0)\lor 0}
\end{align*}

\noi 
which is a negative power provided $s_0>\frac 14+ 2\dl$.
Then, if $N\ll N_1$, by \eqref{N1cond}, we must have $N_1\sim (N_2\vee N_3)$.  If  $N_3 \gg N_2$, then,
\begin{align*}
\eqref{X0mult} \les  \frac{N^s N_2^{\frac 12-s_0}}{N_{\max}^{s+s_0-4\dl}} \les N_{\max}^{-s_0+ 4\dl + (\frac 12-s_0)\lor 0}, 
\end{align*}
while in the case $N_2 \gg N_3$, we have
\begin{align*}
\eqref{X0mult} \les \frac{N^s N_3^{\frac 12-s_0}}{(N\vee N_3)^{\frac 12 - 2\dl } N_{\max}^{s+s_0-\frac 12 -2\dl}} 
\les (N\vee N_3)^{s-s_0 + 2\dl}N_{\max}^{-s-s_0+\frac 12 +2\dl} \les N_{\max}^{-2s_0+\frac 12+4\dl}.
\end{align*}
 In both cases, we have a negative power of the largest dyadic, thus completing the estimate for \eqref{X0}.

\smallskip
\noi
\underline{Bound for \eqref{X1}:}

\smallskip
\noi
We move onto the estimate for \eqref{X1}. By removing the outer modulation restriction and then using duality, it suffices to control 
\begin{align*}
N^{s} \int_{\R\times\R}  \cj{\P_{N}\bigg( \frac{\ft g}{\jb{\s}^{\frac 12-2\dl}}\bigg)^{\vee} }\P_{N_1}\Q_{\ges K}v \cdot \P_{-}\P_{N_{23}}\dx[  \cj{\P_{N_2}u_2}\P_{N_3} u_3    ]dt dx
\end{align*}
for $g\in L^{2}_{t,x}$ of unit norm. Using Cauchy-Schwarz, \eqref{L4}, and Bernstein's inequality, we have
\begin{align*}
&N^s N_{23} \| \P_{N_1}\Q_{\ges K}v\|_{L^{2}_{t,x}}  \bigg\|  \P_{N}\bigg( \frac{\ft g}{\jb{\s}^{\frac 12-2\dl}}\bigg)^{\vee}\bigg\|_{L^{4}_{t,x}}  \|\P_{N_2}u_2 \cdot \P_{N_3}u_3\|_{L^{4}_{t,x}} \\
& \les  \frac{N^s N_{23} (N_2 \wedge N_3)^{\frac 12-s_0}}{K^{\frac 12} N_1^s  (N_2\vee N_3)^{s_0}}
\|\P_{N_1}v\|_{X^{s, \frac 12}} \|g\|_{L^{2}_{t,x}} \| J^{s_0}\P_{N_2\vee N_3}u_{j_1}\|_{L^{4}_{t,x}} \| \P_{N_2\wedge N_3} u_{j_2}\|_{L^{\infty}_{t}H^{s_0}_{x}}. 
\end{align*}
These dyadic factors are controlled by \eqref{X0mult}, thus the proof follows from the previous case.

\smallskip
\noi
\underline{Bound for \eqref{X2}:}

\smallskip
\noi
By duality, H\"{o}lder's inequality, \eqref{bilin1}, and Bernstein's inequality, we have
\begin{align*}
\eqref{X2} & \les N_{23}\| \P_{N_{23}} [\cj{\Q_{\ll K}\P_{N}\Pbhip g} \cdot \Q_{\ll K}\P_{N_1}v]\|_{L^{2}_{t,x}}\|\Q_{\ges K}\P_{N_{(3)}}u_{(3)}\|_{L^{2}_{t,x}} \| \P_{N_{(2)}}u_{(2)}\|_{L^{\infty}_{t,x}} \\
& \les \frac{N_{23}^{\frac 12+10\dl} (N_2 N_3)^{\frac 12-s_0} }{N_1^s K^{\frac 12}} \| \P_{N}g\|_{X^{0,\frac 12-2\dl}}\|\P_{N_1}v\|_{X^{s,\frac 12}} \|\P_{N_{(3)}}u_{(3)}\|_{X^{s_0-\frac 12, \frac 12}}\|\P_{N_{(2)}}u_{(2)}\|_{L^{\infty}_{t}H^{s_0}_x}
\end{align*}

\noi
We consider the ensuing multiplier which is
\begin{align}
\frac{N^s N_{23}^{\frac 12+10\dl} (N_2 N_3)^{\frac 12-s_0} }{N_1^s K^{\frac 12}}
\les \frac{(N\vee N_3)^{s-s_0} N_2^{\frac12-s_0}}{N_{\max}^{s-10\dl}} \les N_{\max}^{-s_0 + 10\dl + (\frac12-s_0)\lor 0}
\notag
\end{align}
where we have used that $s+s_0>\frac 12+10\dl$ and $s_0>\frac 14+5\dl$. Under these conditions, we can then sum over all of the dyadics.

\smallskip
\noi
\underline{Bound for \eqref{X3}:}

\smallskip
\noi
Recall that $N_1 \sim N_{\max}$.
Fix $\eps>0$ sufficiently small, to be chosen later.
 By duality, H\"{o}lder's inequality, \eqref{bilin1}, and Bernstein's inequality, we control \eqref{X3} by
\begin{align*}
&N^{s} N_{23}^{\frac 12+10\dl}N_1^{-s} \| \P_{N} g \|_{X^{0,\frac 12 -2\dl}}\|\P_{N_1}v\|_{X^{s,\frac 12}}   \| \P_{N_{(3)}} \Q_{\ll K} u_{(3)}  \|_{L_{t,x}^{\frac{2(1+\eps)}{\eps}}} \| \P_{N_{(2)}} \Q_{\ges K} u_{(2)}\|_{L^{2(1+\eps)}_{t,x}} \\
& \les 
\frac{ N^{s} N_{23}^{\frac 12+10\dl}N_{(3)}^{1-s_0- \frac{\eps}{2(1+\eps)}}  N_{(2)}^{1-s_0+\frac{\eps}{2(1+\eps)}}}{N_1^s K^{1-\frac{\eps}{2(1+\eps)}}}
\|\P_{N} g\|_{X^{0,\frac 12}} \|\P_{N_1}v\|_{X^{s,\frac 12}} \\
&\hphantom{XXXXxXXXxXXXXXXXXXXX} \times    \| \P_{N_{(3)}} \Q_{\ll K} u_{(3)}  \|_{X^{s_0-\frac 12,\frac 12}} \| \P_{N_{(2)}}  u_{(2)}\|_{X^{s_0-1,1}}. 
\end{align*}

\noi 
We now consider the dyadic multiplier. We have
\begin{align*}
\frac{ N^{s} N_{23}^{\frac 12+10\dl}N_{(3)}^{1- \frac{\eps}{2(1+\eps)}-s_0}  N_{(2)}^{1-s_0+\frac{\eps}{2(1+\eps)}}}{N_1^s K^{1-\frac{\eps}{2(1+\eps)}}} 
&\les
\frac{N^s N_{(2)}^{\frac12 - s_0 + 10\dl + \frac{2\eps}{1+\eps} }  N_{(3)}^{1-s_0 + \eps}}{ N_1^s (N \lor N_3)^{1-\eps}}
 \les 
N_{\max}^{\al}, 
\end{align*}
where for $N_1 \sim N$ we have $\al = -1 + \frac{\eps}{2(1+\eps)} + (\frac12 - s_0 + 10\dl + \frac{2\eps}{2(1+\eps)}) \lor 0 + (1-s_0 +\frac{\eps}{2(1+\eps)})\lor 0$, and for $N_1 \sim N_{(2)} \gg N$, we have $\al = \frac12 - s_0 -s + 10\dl + 2\eps+ (-1 + s + \eps)\lor 0 +(1-s_0 +\eps)\lor 0$, which gives $\al<0$ in both cases, given $\eps>0$ sufficiently small so that $2\eps<\dl$ and $s_0>\frac{1}{4}+6\dl$.

\smallskip
\noi
\underline{Bound for \eqref{Xres}:}

\smallskip
\noi
We argue that due to impossible frequency interactions, we must actually have that 
\begin{align}
\eqref{Xres} = 0. \label{Xres0}
\end{align}
If $N_{2}\wedge N_3 =N_3$, then from the Case 3 condition, we have $N\gg N_3$ and hence, from \eqref{mods}, we have $|\Phi(\cj{\xi})|\sim N_{23} N$. If instead $N_{2}\wedge N_3 =N_2$, then from the hyperplane condition $\xi=\xi_1-\xi_2+\xi_3$ and \eqref{N1cond} which implies $N_1 \ges N \gg N_2$, we have $|\xi-\xi_3|=|\xi_1-\xi_2| \sim N_1\ges (N\vee N_3)$. Therefore, in either case, from \eqref{mods}, we have 
\begin{align}
K \les |\Phi(\cj{\xi})| \le |\s| + |\s_1| + |\s_2| + |\s_3|, 
\notag
\end{align}
where $K$ was defined in \eqref{K}. However, due to the projectors $\Q_{\ll K}$ in \eqref{Xres}, we have
\begin{align*}
|\s| + |\s_1| + |\s_2| + |\s_3| \ll K, 
\end{align*}
which is incompatible with the lower bound  and hence \eqref{Xres0} holds true.

This completes the proof of \eqref{vX0}.
\end{proof}

We now give the main trilinear estimates for the case of \eqref{INLS2}; namely, without the Hardy space assumption. The key new ingredients relative to Proposition~\ref{PROP:tri} are the decomposition in Section~\ref{SEC:INLSdecomp} and the $L^p$ boundedness property in \eqref{QKLp}.
We specify a particular case for the parameters in Lemma~\ref{LEM:BXsb} and Lemma~\ref{LEM:GXsb} that we will use. We choose $s=s_0>\frac 14$, let $\dl>0$ be sufficiently small so that 
\begin{align}
\dl\ll \tfrac{1}{100}(s_0-\tfrac 14) \label{delta}
\end{align}
 and take $0<\s<1$ such that
\begin{align}
\s=\s(\dl) =\tfrac 14+11\dl. \label{sigmadef}
\end{align}
Of course, from \eqref{sigma}, we could take any $\s>\frac 12-s_0+10\dl$ but taking the worst possible regularity in \eqref{sigmadef} is sufficient.

\begin{proposition}[Trilinear estimates for INLS \eqref{INLS2}]\label{PROP:tri2}
Let $s\geq s_0>\frac 14$ and $0<T\leq 1$. Let $\dl>0$ and $\s=\s(\dl)$ be as in \eqref{sigmadef}.
Then, for any $u_j\in L^{\infty}_{T}H^{s_0} \cap L^{4}_{T}W^{s_0,4}_x $ and
\begin{align*}
u_j\in  X_{T}^{s_0-\frac 12,\frac12}, \quad u_j-\mathcal{I}[B(u_j,w_j)]\in X^{s_0-1,1}_{T}, \quad \text{and} \quad \mathcal{I}[B(u_j,w_j)]\in X^{s_0-\s,\frac 12}_{T},
\end{align*}
for $j=2,3$,
with
\begin{align}
Y(u_2, u_3) : =   &\prod_{j=2}^3 \big( \|J^{s_0}\PbHI u_j\|_{L^{4}_{T,x}} + \| u_j\|_{L^\infty_T H^{s_0}_x}+ \|u_j\|_{X_{T}^{ s_0-\frac 12,\frac 12}}
\big)
\notag \\
& + \sum_{\substack{j_1,j_2\in\{2,3\}\\ j_1 \ne j_2}}
\Big(
 \| u_{j_1}\|_{X_T^{s_0-\frac 12,\frac 12}}  \| u_{j_2}-\mathcal{I}[B(u_{j_2},w_{j_2})]\|_{X_T^{s_0-1,1}} 
\notag
\\
&
\hspace{2.5cm}
+
 \| u_{j_1} \|_{L^\infty_T H^{s_0}_x} \|\mathcal{I}[B(u_{j_2},w_{j_2})]\|_{X_{T}^{s_0-\s, \frac 12}}
\Big)
\notag
\end{align}
\noi
and for any $v,w\in X_{T}^{s,\frac 12+\dl}$ satisfying $v=\Pbhip v$ and $w=\P_{-, \textup{hi}}w$, it holds that:
\begin{align}
\| \ind_{[0,T]} \Pbhip[ v \, \P_{-}\dx( \cj{u_2} u_3)] \|_{X^{s,-\frac 12+2\dl}} & \les T^{\dl} \|v\|_{X_{T}^{s,\frac12+\dl}} Y(u_2, u_3),
 \label{vX1}\\
\| \ind_{[0,T]} \P_{-,\textup{hi}}[ w \, \P_{+}\dx( \cj{u_2} u_3)] \|_{X^{s,-\frac 12+2\dl}} & \les T^{\dl} \|w\|_{X_{T}^{s,\frac12+\dl}} Y(u_2, u_3) 
.
\label{wX1}
\end{align}
\end{proposition}

\begin{proof}

Fix parameters $(s,s_0, \dl, \s)$ as in the statement.
We will only show \eqref{vX1}, as the estimate for \eqref{wX1} follows by the same argument. In particular, note that for both quantities, the projections guarantee the following relation between the frequencies $|\xi_1| \ges |\xi| \lor |\xi_2 - \xi_3|$, as used in \eqref{N1cond} in the proof of Proposition~\ref{PROP:tri}. 
By further  examining this proof, 
we see that we only need to modify the argument when one of the functions $u_2,u_3$ is placed in $X^{s_0-1,1}_T$. 
In particular, 
it remains to estimate the term \eqref{X3}, since in the proof of Proposition~\ref{PROP:tri}, we placed the term $\P_{N_{(2)}}\Q_{\ges K}u_{(2)}$ into the space $L^{2(1+\eps)}_{t,x}$ which was controlled in $X^{s_0-1,1}$. We need to refine this step as we no longer can assume that $u_{(2)}$ lives  in $X^{s_0-1,1}$.

Going one step back in the computation for \eqref{X3}, by duality, Cauchy-Schwarz and \eqref{bilin1}, we have to control
\begin{align}
N^{s}N^{\frac 12+10\dl}_{23}N_1^{-s}\|\P_{N}g\|_{X^{0,\frac 12-2\dl}}\|\P_{N_1}v\|_{X^{s,\frac 12}} \| \P_{N_{(3)}}\Q_{\ll K} u_{(3)} \cdot \P_{N_{(2)}}\Q_{\ges K} u_{(2)}\|_{L^{2}_{t,x}}. \label{N30}
\end{align}
We focus on the last factor,  using the triangle inequality to bound it as follows
\begin{align}
\| &\P_{N_{(3)}}\Q_{\ll K} u_{(3)} \cdot \P_{N_{(2)}}\Q_{\ges K} u_{(2)}\|_{L^{2}_{t,x}} \notag \\
& \leq \| \P_{N_{(3)}}\Q_{\ll K} u_{(3)} \cdot \P_{N_{(2)}}\Q_{\ges K} \{u_{(2)}-\mathcal{I}[B(u_{(2)},w_{(2)})]\} \|_{L^{2}_{t,x}}  \label{N31}  \\
&\hphantom{XX} + \| \P_{N_{(3)}}\Q_{\ll K} u_{(3)} \cdot \P_{N_{(2)}}\Q_{\ges K} \mathcal{I}[B(u_{(2)},w_{(2)})] \|_{L^{2}_{t,x}}. \label{N32}
\end{align}
For the contribution to \eqref{N30} coming from \eqref{N31}, we can argue exactly as we did for \eqref{X3} and we have that this contribution is bounded by 
\begin{align*}
N_{\max}^{0-} \| \P_{N_1}v\|_{X^{s,\frac 12}} \|\P_{N_{(3)}} u_{(3)}\|_{X^{s_0-\frac 12,\frac 12}} \| \P_{N_{(2)}} \{u_{(2)}-\mathcal{I}[B(u_{(2)},w_{(2)})]\}\|_{X^{s_0-1,1}}, 
\end{align*} 
thus allowing us to sum in all the dyadics.

We now consider the contribution to \eqref{N30} coming from \eqref{N32}.
Let $\eps>0$ sufficiently small, to be chosen later.
In view of the definition of $K$ in \eqref{K}, \eqref{N1max}, and the Case 3 assumption ($N_{(3)} \ll N\wedge N_{(2)}$) we have $K\gg N_{(3)}^{2}$. 
Thus, by \eqref{QKLp}, interpolation, and Sobolev embedding, we have
\begin{align}
 \| \P_{N_{(3)}}\Q_{\ll K} u_{(3)}\|_{L^{\frac{2(1+\eps)}{\eps}}_{t,x}} & \les  \| \P_{N_{(3)}} u_{(3)}\|_{L^{\frac{2(1+\eps)}{\eps}}_{t,x}} \notag \\
 & \les \|\P_{N_{(3)}} u_{(3)}\|_{L^{\infty}_{t,x}}^{1-\frac{\eps}{1+\eps}} \|\P_{N_{(3)}} u_{(3)}\|_{L^2_{t,x}}^{\frac{\eps}{1+\eps}} \notag \\
 &\les N_{(3)}^{\frac 12-s_0-\frac{\eps (1/2-s_0)}{(1+\eps)}}\|J^{s_0}\P_{N_{(3)}} u_{(3)}\|_{L^{\infty}_{t}L^2_x}^{1-\frac{\eps}{1+\eps}}\|\P_{N_{(3)}} u_{(3)}\|_{X^{0,0}}^{\frac{\eps}{1+\eps}} . \label{u3bd}
\end{align}
Thus, by H\"{o}lder's inequality and \eqref{u3bd}, we have 
\begin{align*}
\eqref{N32} 
&
\leq \| \P_{N_{(3)}}\Q_{\ll K} u_{(3)}\|_{L^{\frac{2(1+\eps)}{\eps}}_{t,x}}  \| \P_{N_{(2)}}\Q_{\ges  K}\mathcal{I}[B(u_{(2)},w_{(2)})] \|_{L^{2(1+\eps)}_{t,x}} \\
& \les \frac{N_{(3)}^{\frac 12-s_0 - \frac{\eps (1/2-s_0)}{(1+\eps)}}}{  K^{\frac 12 -\frac{\eps}{2(1+\eps)}}  N_{(2)}^{s_0-\s(s_0,\dl) - \frac{\eps }{2(1+\eps)}}}\|J^{s_0}\P_{N_{(3)}} u_{(3)}\|_{L^{\infty}_{t}L^2_x}^{1-\frac{\eps}{1+\eps}}\|\P_{N_{(3)}} u_{(3)}\|_{X^{0,0}}^{\frac{\eps }{1+\eps}} \\
& \hphantom{XXXXXXxXXXXXXXXXXX} \times   \|\P_{N_{(2)}} \mathcal{I}[B(u_{(2)},w_{(2)})]\|_{X^{s_0-\s,\frac 12}} .
\end{align*}
Inserting this bound into \eqref{N30}, we obtain the dyadic factor
\begin{align*}
\frac{N^s N_{23}^{\frac 12+10\dl} (N_2\wedge N_3)^{\frac 12-s_0 - \frac{\eps (1/2-s_0)}{(1+\eps)}} }{N_1^s K^{\frac 12 -\frac{\eps}{2(1+\eps)}}  (N_2 \vee N_3)^{s_0-\s(s_0,\dl) - \frac{\eps}{2(1+\eps)}} } 
&\les \frac{ (N\vee N_{(3)})^{s-2s_0+\s+\frac{\eps (s_0+1/2)}{1+\eps}}   }{ N_1^{ s-10\dl-\frac{\eps}{1+\eps}}}  \\
& \les N_{\max}^{ -2s_0+\s+10\dl +\frac{\eps(s_0+3/2)}{1+\eps}}.
\end{align*}
\noi 
This is a negative power of $N_{\max}$ provided that 
\begin{align}
s_0+(s_0-\s)> 10\dl +\tfrac{\eps(s_0+3/2)}{1+\eps}. 
\notag
\end{align}
This is clearly satisfied for since $s_0>\frac 14$ and by taking $\dl,\eps>0$ sufficiently small.
This completes the proof of \eqref{vX1}. \qedhere
\end{proof}

We now control the remaining contributions from the right-hand sides of \eqref{veq} and \eqref{weq}

\begin{lemma}\label{LEM:easyterms}
Let $0< T \le 1$, $\frac 14<s<\frac 23$, $0<\dl_0<\frac 18$, and $u\in L^{\infty}_{T}H^{s}_x$. Then, it holds that 
\begin{align}
\|  \ind_{[0,T]} \P_{\pm,\textup{hi}}[ & e^{i \be F}u\L_{h}(|u|^2)]\|_{ X^{s,-\frac 12+\dl_0}}  \notag\\
 & \les T^{\dl_0} \max_{p\in \{2,3,4\}} \big(\|\QQ_{h}\|_{L^ p\to L^p}\big)   \big( 1+ \|u\|_{L^{\infty}_{T}H^{\frac 14+}_{x}}^{3} \big) 
\| J^{s} u\|_{L^{4}_{T,x}}\| J^{\frac 14+}  u\|_{L^{4}_{T,x}}^2, 
 \label{vbd2}
\end{align}
for $\QQ_h$ as in \eqref{Qh}.
\end{lemma}

\begin{proof}
From Lemma~\ref{LEM:linXsb}, Cauchy-Schwarz,  Sobolev embedding, and the assumption $\dl_0<\frac 18$, it holds that
\begin{align}
\|\ind_{[0,T]} f\|_{X^{s, -\frac{1}{2}+\dl_0}} & \les T^{\frac 14-\dl_0} \|\ind_{[0,T]} f\|_{X^{s, -\frac{1}{2}+2\dl_0}} \les T^{\dl_0} \|f\|_{L^{\frac{1}{1-2\dl_0}}_{T}H^{s}_x}  \les T^{\frac 14-\dl_0} \|f\|_{L^{\frac{4}{3}}_{T}H^{s}_{x}}  
\notag
\end{align}
We then use \eqref{L2F1}, Bernstein's inequality, and the fractional Leibniz rule. We give the bounds for \eqref{vbd2}:
\begin{align*}
\| \P_{\ll 1}( u\mathcal{Q}_{h}(|u|^2))\|_{L^{\infty}_x} 
& \les \|  u\mathcal{Q}_{h}(|u|^2)\|_{L^{\frac 43}_x} 
 \les \|\mathcal{Q}_{h}\|_{\text{op}}\|u\|_{L^4_x}^3 , \\
 \|  u\mathcal{Q}_{h}(|u|^2)\|_{L^{3}_x} & \les  \|u\|_{L^{\infty}_x} \|\mathcal{Q}_{h}\|_{L^3 \to L^3} \|u\|_{L^6_{x}}^{2} \les \|\mathcal{Q}_{h}\|_{L^3 \to L^3} \| J^{\frac 14+}u\|_{L^4_x}^3, \\
 \|D^{s}\P_{\ges 1}( u\mathcal{Q}_{h}(|u|^2))\|_{L^2_x}& \les\|\mathcal{Q}_{h}\|_{L^4 \to L^4} \|D^s u\|_{L_x^4} \| J^{\frac 14+}u\|_{L^4_x}^2. 
\end{align*} 
This completes the proof. \qedhere
\end{proof}
It is clear that by using \eqref{L2F2} we also have corresponding difference estimates for \eqref{vbd2}.

\section{Well-posedness and the infinite depth limit}
\label{SEC:LWP}

In this section, we prove Theorem~\ref{THM:main}. As the general argument here is quite standard we will be brief. 
For further details, we refer to \cite[Sections 5-6]{CFL1} where similar arguments were made for \eqref{INLS2} with periodic boundary conditions; the argument itself is based on that in \cite{MP, MP2, GLM}. We will focus on the case of \eqref{INLS2} with the same result for \eqref{CCM} with the Hardy space assumption following in a similar but simpler fashion.

By the result of \cite{PMP}, we have local-in-time well-posedness for \eqref{INLS2} in $H^{s}(\R)$ for any $s>\frac 12$. We fix $s_1>\frac 12$ but close to $\frac 12$ and fix $\frac 14<s_0\leq s\leq \frac 12 < s_1$. As we do not know if the solutions in $H^{s_1}(\R)$ are global-in-time, we first obtain apriori estimates on these solutions. These will ensure that their maximal time of existence  is lower bounded by a time $T_{\ast}>0$ only depending on $\|u_0\|_{H^{s_0}}$ and not on $\|u_0\|_{H^{s_1}}$.

Given $u_0\in H^{\infty}(\R)$, let $u\in C([0,T_{\max}];H^{\infty}(\R))$ be the solution to \eqref{INLS2}. We then define
\begin{align}
\begin{split}
N_{T}^{s}(u) = \max\bigg(  \|u\|_{L^{\infty}_{T}H^{s}_x}, &\|J^{s}u\|_{\wt{L^4_{T,x}}},  \|v(0)\|_{H^{s}}, \|w(0)\|_{H^{s}},  \\
&\|\ind_{[0,T]}\NN_{v}(u)\|_{X^{s, -\frac 12+2\dl_0}}, \| \ind_{[0,T]}\NN_{w}(u)\|_{X^{s, -\frac 12+2\dl_0}}  \bigg),
\end{split} \label{NTu}
\end{align}
where $v,w$ are the gauged variables \eqref{gauge}, 
$\NN_{{v}}$ and $\NN_{w}$ are the nonlinear terms defined in \eqref{veq} and \eqref{weq}, and $0<\dl_0\ll 1$ is fixed as in \eqref{delta}. Given $0<h\leq \infty$, we also set 
\begin{align}
 L_{h} : = \max_{p\in \{2, 4\}} \|\QQ_{h} \|_{L^{p}\to L^p}
 \notag
\end{align}
with the understanding that $L_{\infty}=|\g|$.

First, from the Duhamel formula for the $v$ and $w$ equations in \eqref{veq}-\eqref{weq}, together with Lemma~\ref{LEM:linXsb}, we have
\begin{align*}
    &
    \| v\|_{X^{s, \frac12+\dl_0}_T} + \|w\|_{X^{s, \frac12 + \dl_0}_T}
    \\
    &
    \les T^{\dl_0} \big( \| v(0)\|_{H^s} + \|w(0)\|_{H^s} + \| \ind_{[0,T]} \NN_v(u) \|_{X^{s, -\frac12+2\dl_0}} + \| \ind_{[0,T]} \NN_w(u) \|_{X^{s, -\frac12+2\dl_0}} 
        \big)
    \\
    &
    \les T^{\dl_0} N^s_T(u). 
\end{align*}

\noi
By combining Lemma~\ref{LEM:BXsb}, Lemma~\ref{LEM:GXsb}, Lemma~\ref{LEM:uinfoINLS}, \eqref{veq}-\eqref{weq}, Proposition~\ref{PROP:tri2} and Lemma~\ref{LEM:easyterms}, we obtain
\begin{align}
\begin{split}
N^{\s}_{T}(u) & \leq  C_1 (1+\|u_0\|_{H^{s_0}})^{a_1} \|u_0\|_{H^{\s}}  \\
&+C_{2}\big\{ 
( T^{\ta_1} (1+L_{h})+M^{-\ta_2}) Q_{2}(N_{T}^{s_0}(u)) 
+ T^{\ta_3} M (1+L_h) Q_{3}(N_{T}^{s_0}(u)) 
\big\}
N_{T}^{\s}(u)
, 
\end{split} \label{apriori1}
\end{align}
where $\s\in \{s_0,s,s_1\}$, and
for some constants $a_1, \ta_1, \ta_2, \ta_3>0$ and $C_1, C_2>0$ and where $Q_{2}, Q_3$ are some non-negative polynomials, where $Q_3$ has no constant term.

 Whilst the constants depend on $\{s_0,s,s_1\}$, we only use these three regularities and can thus take the maximum of the given constants over $\{s_0,s,s_1\}$.
We first put $\s=s_0$ in \eqref{apriori1}.
We then choose $M\gg 1$ depending on $\|u_0\|_{H^{s_0}}$ 
\begin{align}
C_{2}M^{-\ta_2}Q_{2}(4R) <\tfrac{1}{4} \quad \text{where} \quad R: = C_1 (1+\|u_0\|_{H^{s_0}})^{a_1} \|u_0\|_{H^{s_0}}. \label{M}
\end{align}
Then, given this choice of $M$, we choose $T_{\ast}=T_{\ast}(M, L_h)>0$
so that 
\begin{align}
C_2 (1+L_h) \big( T_*^{\ta_1} Q_2(4R) + T_*^{\ta_3} M  Q_3(4R) \big) < \tfrac14.
 \label{Tast}
\end{align}

\noi
This verifies that 
$N_{T}^{s_0}(u) \leq 2R$, which is the apriori bound at regularity $\s=s_0$. By using this information in \eqref{apriori} at regularity $\s=s_1$ and reducing $T=T(\|u_0\|_{H^{s_0}})>0$ if necessary, we obtain
\begin{align}\label{apriori2}
\| u\|_{L^{\infty}_{T'}H^{s_1}} \leq N_{T'}^{s_1}(u) \leq 2C_1 (1+\|u_0\|_{H^{s_0}})^{a_1} \|u_0\|_{H^{s_1}},
\end{align}
for any $0<T'\leq T$
which implies that the maximal time of existence for these solutions is bounded from below by $T_{\ast}=T_\ast(\|u_0\|_{H^{s_0}})$.
Note that $T_{\ast}$ also depends on $L_{h}$ but can be chosen uniformly in $h$ only depending on $\sup_{1\leq h \leq \infty}L_{h}$ in view of \eqref{GGconv}.

As for the uniqueness and continuity of the flow map, we consider differences of $H^{\infty}(\R)$ solutions and derive a difference estimate. Given two such solutions $u_1,u_2$ to \eqref{INLS2} with initial data $u_j(0)\in H^{\infty}(\R)$, we consider the difference $U:=u_1-u_2$ and $W:=w_1-w_2$, where $w_j : = \P_{-,\text{hi}}u_j$. 
We also define the primitives $F_{j}= F_{j}[u_j]$ as in \eqref{F} and the attendant gauged variables $v_{j}=v_{j}[u_j]$ as in \eqref{gauge} with corresponding difference $V:=v_1-v_2$.
By the previous analysis, we have control on $N_{T}^{\s}(u_j)$, $j=1,2$, for $\s\in \{s_0,s\}$ and for time $0<T\leq T_{\ast}$. In particular, by the Duhamel formulas for $v_j, w_j$ in \eqref{veq}-\eqref{weq}, this provides control on the $X^{\s,\frac 12+\dl_0}_{T}$ norms of $v_j,w_j$.
We then obtain difference estimates for the norms appearing in \eqref{NTu}, where we additionally use Lemma~\ref{LEM:diffests}. 
We note that unlike the situation for the Benjamin-Ono equation in \cite{MP}, the difference estimate does not assume that the data $u_j(0)$, $j=1,2$, agree on low frequencies.  Indeed, by Cauchy-Schwarz and \eqref{F} we see that 
\begin{align*}
\|F_{1}-F_{2}\|_{L^{\infty}_{T,x}} \les \| |u_1|^2 - |u_2|^2\|_{L^{\infty}_{T}L^1_{x}} \les (\|u_1\|_{L^{\infty}_{T}L^2_x} +\|u_2\|_{L^{\infty}_{T}L^2_x})\|U\|_{L^{\infty}_{T}L^2_x}.
\end{align*}
At the end of this procedure, we obtain 
\begin{align}
\| J^{\s} U\|_{L^{\infty}_{T}L^2_x \cap \wt{L^{4}_{T,x}}}  + \|V\|_{X^{\s,\frac12+\dl_0}_{T}} + \|W\|_{X^{\s,\frac12+\dl_0}_{T}}  \leq C_5(\|u_1(0)\|_{H^{s_0}}, \|u_2(0)\|_{H^{s_0}})  \| U(0)\|_{H^{\s}}, \label{LipbdT}
\end{align}
for any $0<T\leq T_{0}$ and $\s\in \{s_0,s\}$.

For the existence of solutions in $H^{s}(\R)$, we fix $u_0\in H^{s}(\R)$ and consider the sequence of approximations $u_{0,j} = \mathcal{F}^{-1}\{ \ind_{[-j,j]} \ft u_0\}$ with corresponding $H^{\infty}(\R)$ solutions $\{u_j\}_{j\in \N}$. By the previous results, these all belong to $C([0,T_{\ast}];H^{\infty}(\R))$, where $T_\ast = T_\ast(\|u_0\|_{H^{s_0}})>0$. 
Since $u_{0,j}$ converges to $u_0$ in $H^s(\R)$,  we may choose $M$ in \eqref{M} uniformly in $j \in \N$ and moreover, the sequence is equicontinous in $H^{s}(\R)$ and thus uniformly tight on the Fourier side. This property guarantees that we may choose $M$ uniformly in $j\in \N$ in order to obtain smallness for the second term on the right-hand side of \eqref{uHsbddiff}. By \eqref{LipbdT}, we see that the sequence $\{u_j\}_{j\in \N}$ is then Cauchy in the norm appearing in $N_{T}^{s}$ and hence converges to a limit $u$ there, which satisfies \eqref{INLS2} in the distributional sense and has $u\vert_{t=0}=u_0$. This completes the proof of the local well-posedness in Theorem~\ref{THM:main}.

Finally, we consider the infinite depth limit as $h\to \infty$.

\begin{proof}[Proof of Theorem~\ref{THM:limit}]
We only give sketch of the argument, and refer to similar full details in \cite{CFL1}. See also \cite{CLOP}.  
Given $u_0 \in H^s(\R)$ and a net $\{u_{0,h}\}_{1 \le h <\infty}$ in $H^s(\R)$ with $u_{0,h} \to u_0$ in $H^s(\R)$, we denote by $u_h$ and $u_\infty$ the global solutions to \eqref{INLS2} and \eqref{CCM}, respectively, with $u_h \vert_{t=0} = u_{0,h}$ and $u_\infty \vert_{t=0} = u_0$, constructed in Theorem~\ref{THM:GWP}. Moreover, for $1 \le h \le \infty$, we write $F_h = F_h[u_h]$ for the primitives as in \eqref{F}, $w_h$ and $v_h$ for the gauged variables in \eqref{gauge}, and $U_h = u_h - u_\infty$, $V_h = v_h - v_\infty$, and $W_h = w_h - w_\infty$. 
Repeating the process for the apriori bounds in \eqref{apriori2}, we get
\begin{align}
\sup_{1 \le h \le \infty} N^s_{T_*}(u_h)
\le  2 C_1 (1 + \|u_0\|_{H^{s_0}})^{a_1} \|u_0\|_{H^s}, 
\notag
\end{align}
where $T_* = T_*(\|u_0\|_{H^{s_0}})>0$ as in \eqref{Tast}, and where this choice can be made uniformly in $h$.
Then, by repeating the process to obtain the difference estimates in \eqref{LipbdT}, we get
\begin{align}\label{conv}
\begin{split}
    & \| J^s U_h \|_{L^\infty_{T_*} L^2_x \cap \wt{L^4_{T_*,x}} } + \| V_h \|_{X^{s, \frac12+\dl_0}_{T_*}} + \|W_h\|_{X^{s, \frac12+\dl_0}_{T_*}} 
    \\
    &
    \le C_2 (\|u_0\|_{H^{s_0}}) \big( \| u_0 - u_{0,h} \|_{H^s} 
 + \max_{p\in\{2,4\}}\|\QQ_h - \QQ_\infty\|_{L^p \to L^p} \big), 
    \end{split}
\end{align}
for some constant $C_2 = C_2 (\|u_0\|_{H^{s_0}})>0$ depending only on $\|u_0\|_{H^s}$, uniform in $1 \le h <\infty$, and
where $\QQ_h$ is as in \eqref{Qh}, and we recall that $\QQ_\infty= - i \be \H\dx  + i \g \Id$.
 Thus, from the convergence of the initial data, the fact that $\QQ_h - \QQ_\infty = -i \be \GG_h$, where $\GG_h$ is as in \eqref{Lh}, and \eqref{GGconv}, 
 we conclude that $u_h$ solving \eqref{INLS2} converges to the solution $u_\infty$ to \eqref{CCM} in $C([0,T_*]; H^s(\R))$ as $h\to\infty$. Given a target time $T \gg1$, since \eqref{conv} for $h\gg1$ guarantees that $\|U_h (T_*) \|_{H^s} \le \|u_0\|_{H^s}$, we can iterate the convergence argument to obtain convergence over the full interval $[0,T]$. 
\end{proof}

\section{Conservation laws}\label{SEC:GWP}

In this section, we state our new Lax pair and use it to obtain low-regularity apriori bounds for \eqref{INLS}.
Given $0<h\leq \infty$, we define 
\begin{align}
\Pih = \tfrac{1}{2}(1+i\mathcal{T}_{h})
\label{Pi+h}
\end{align}
with $\TT_{\infty}= \H$. Indeed, when $h=\infty$, \eqref{Hilbert} implies that $\Pi_{+,\infty}=\P_{+}$.
For finite $h>0$, $\Pih$ has a singularity at the zero frequency that is even more severe. 
With \eqref{Pi+h}, we can also write \eqref{INLS} (with $\g=0$) as
\begin{align}
\dt u + i\dx^2 u =  2\be u \Pih\dx(|u|^2), 
\label{INLSG}
\end{align}
where $\be\in \{\pm 1\}$. We recall that $\be=1$ corresponds to the defocusing case and $\be=-1$ to the focusing, and again encompassing CCM~\eqref{CCM} when $h=\infty$.
We now state the new Lax pair for \eqref{INLSG} for all $0<h\leq \infty$.

\begin{proposition}[Lax pair]
For any $0<h\leq \infty$, $u(t)$ solves~\eqref{INLSG} on the line if and only if the operators
\begin{equation}
\lax_{u;h} = -i\partial_x +\be u\Pih\cj{u} 
\quad\text{and}\quad
\peter_{u;h} = -i\partial_x^2  +2\be u\partial_x\Pih \cj{u} 
\label{Lax}
\end{equation}
on $L^2(\R)$ satisfy
\begin{equation}
\tfrac{d}{dt}\lax_{u;h} = [\peter_{u;h},\lax_{u;h}] .
\label{lax}
\end{equation}
\end{proposition}

When $h=\infty$, this resembles the well-known Lax pair for CCM in the literature for the special case $u\in L^2_+(\R)$, while for $h<\infty$ this appears to be new. 
Indeed, the Lax operator in~\cite{Pel3} is a $2\times 2$ operator-valued matrix, similar to that of NLS.

\begin{proof}
Let $f\in H^{\infty}(\R)$ be a test function. Let $\be = \mp 1$ denote the sign of the nonlinearity in \eqref{INLSG}.
We compute 
\begin{align*}
[\mathcal{P}_{u;h}, \L_{u;h}] f& = -i\be \dx^2[ u \Pih(\cj{u}f)]-2i\be u\dx \Pih(\cj{u}f') +2u\dx \Pih[ |u|^2 \Pih(\cj{u}f)] \\
& \hphantom{XX} + 2\be i \dx [ u\dx\Pih(\cj{u}f)] +i\be u \Pih( \cj{u}f'') -2u \Pih(|u|^2 \dx \Pih(\cj{u}f)).
\end{align*}
Consider first the terms which we are quadratic in $u$. By developing them further, we have
\begin{align}
 & {-i}\be \dx^2[ u \Pih (\cj{u}f)]-2i\be u\dx \Pih(\cj{u}f')+ 2\be i \dx [ u\dx \Pih(\cj{u}f)] +i\be u \Pih( \cj{u} f'')\notag  \\
 & = -i\be\big\{  \dx^2[ u \Pih(\cj{u}f)] +2 u\dx \Pih(\cj{u} f') -2\dx [ u\dx \Pih(\cj{u}f)] -u\Pih(\cj{u} f'')\big\} \notag \\
 & = -i\be\big\{    (\dx^2 u) \Pih(\cj{u}f)  + 2(\dx u ) \dx\Pih(\cj{u}f)  + u \Pih\dx^2( \cj{u}f)    +2 u\dx \Pih(\cj{u} f')\notag  \\
 &\hphantom{XXXXX} -2(\dx u ) \dx \Pih (\cj{u}f) -2u \dx^2 \Pih(\cj{u}f) - u\Pih(\cj{u}f'')
 \big\}  \notag \\
 & =  -i\be\big\{    (\dx^2 u) \Pih(\cj{u}f) -u\dx^2 \Pih(\cj{u}f) +2 u\dx \Pih(\cj{u} f') -u\Pih(\cj{u}f'') \big\} \notag \\
 & = -i\be\big\{    (\dx^2 u) \Pih(\cj{u}f) - u \Pih( (\dx^2 \cj{u})f)\big\}. \label{quad}
\end{align}
Now focus on the quartic in $u$ terms in the commutator:
\begin{align*}
2u \dx \Pih[ |u|^2  \Pih(\cj{u}f)]  -2u \Pih[ |u|^2 \dx \Pih(\cj{u}f)] = 2u \Pih[\dx( |u|^2)  \Pih(\cj{u}f)] .
\end{align*}
We claim that 
\begin{align}
\begin{split}
2u \Pih[\dx( |u|^2)  \Pih(\cj{u}f)]  =  2u \big[ \Pih\dx(|u|^2)  \cdot \Pih(\cj{u}f) +\Pih( \cj{\Pih\dx(|u|^2)} \cdot \cj{u}f) \big],
\end{split} \label{cubic}
\end{align}
whereupon, combining with \eqref{quad}, we will have shown that 
\begin{align*}
[\mathcal{P}_{u;h}, \L_{u;h}]  
= \be ( -i\dx^2 u +2\be u \Pih\dx(|u|^2)) \Pih \cj{u} + \be u \Pih \cj{ ( -i\dx^2 u +2\be u \Pih\dx(|u|^2))} 
\end{align*}
which completes the proof.

We expand the left-hand side of \eqref{cubic} using \eqref{Pi+h}: 
\begin{align*}
\text{LHS} \eqref{cubic}&=  \tfrac{1}{2}  u( 1+i\TT_{h}) (\dx(|u|^2)) (1+i\TT_{h})(\cj{u}f) \\
& =  \tfrac{1}{2}  u\Big\{ (\dx (|u|^2)) \cj{u}f + i \big[ \TT_{h}(  (\dx |u|^2) \cj{u}f) + \dx(|u|^2) \TT_{h}(\cj{u}f)\big] -  \TT_{h}[ \dx( |u|^2) \TT_{h}(\cj{u}f)] \Big\}.
\end{align*}
Similarly, for the right-hand side of \eqref{cubic}, using \eqref{Pi+h} we have: 
\begin{align*}
\text{RHS} \eqref{cubic} & = \tfrac{1}{2}u \Big\{ 2\dx(|u|^2) \cj{u}f + i \{ \TT_{h}(\dx(|u|^2)\cj{u}f)+\dx(|u|^2)\TT_{h}(\cj{u}f)\}  \\ 
& \hphantom{XXXX}+ \TT_{h}[ \TT_{h}\dx(|u|^2)\cdot \cj{u}f] -\TT_{h}\dx(|u|^2)\cdot \TT_{h}(\cj{u}f) \Big\}
.
\end{align*}
Taking the difference of these, we find: 
\begin{align*}
&\text{RHS}\eqref{cubic} - \text{LHS}\eqref{cubic} \\
& = \tfrac{1}{2}u \Big\{ \dx(|u|^2) \cj{u}f + \TT_{h}[ \TT_{h}\dx(|u|^2)\cdot \cj{u}f] -\TT_{h}\dx(|u|^2)\cdot \TT_{h}(\cj{u}f)  +\TT_{h}[ \dx( |u|^2) \TT_{h}(\cj{u}f)]  \Big\}.
\end{align*}
Using \eqref{Tilb1} and noting $M_{\dx(|u|^2)} = 0$, we see that 
\begin{align}
\text{RHS}\eqref{cubic} - \text{LHS}\eqref{cubic} = 0 ,
\notag
\end{align}
hence verifying \eqref{cubic}.
\end{proof}

Our above verification of \eqref{lax} hinged on the Cotlar-type identity \eqref{Tilb1}. It is conceivable that (a slightly modified version of) the Lax pair \eqref{Lax} would also work in the periodic case. However, \eqref{Tilb1} requires modifications due to the presence of the zero frequency, which cannot be avoided by imposing a mean-zero constraint: $\int_{\T}udx $ is not a conservation law for \eqref{INLS}. Consequently, we were not able at this point to give a Lax pair for the periodic \eqref{INLS} for finite~$h$.

In the next result, we make sense of the Lax operator $\L_{u;h}$ as a self-adjoint operator on $L^2(\R)$. To this end, we need to assume that $u\in H^{\frac 14}(\R)$. For the case of CCM~\eqref{CCM} in $L^2_{+}(\R)$, this restriction on the potentials can be weakened to $u\in L^2_{+}(\R)$.

\begin{proposition}[Lax operator] \label{PROP:Lax}
Fix $0<h\leq \infty$.  Given $u\in H^{\frac14}(\R)$, the operator
\begin{equation*}
\lax_{u;h} f = -i\dx f +\be u\Pih(\cj{u}f)
\end{equation*}
with domain $H^1$ is self-adjoint.  Moreover, there is a constant $C\geq 1$ so that for $\kk\in\R$ satisfying 
\begin{equation}
\kappa \geq C \big( 1 + \|u\|_{H^\frac14} \big)^4, 
\label{k0}
\end{equation}
we have
\begin{equation}
\tfrac12 ( \lax_0^2 + \kappa^2 ) \leq \lax_{u;h}^2 + \kappa^2 \leq \tfrac32 (\lax_0^2 + \kappa^2)
\label{mono}
\end{equation}
as quadratic forms, where $\L_0 : =  \L_{0;h} = - i \dx$. 
\end{proposition} 

\begin{proof}
For $\kappa>0$, let $R_0(\kappa) = (|\dx|+\kappa)^{-1} .$ 
We write
\begin{equation}
\Pih f = \tfrac12 i \mathcal{J}_h f + \tfrac{1}{2}(1+i\mathcal{K}_h)f , \label{Pihdecomp}
\end{equation}
where  $\mathcal{K}_{h}$ is defined as in \eqref{Kh} and $\J_{h}$ is the integral operator 
\begin{align}
\mathcal{J}_{h}f(x) &= \frac{1}{2h} \text{p.v.} \int_{\R} \sgn\bigg( \frac{\pi(x-y)}{2h} \bigg) f(y)dy.  \label{Jh}
\end{align}

 Using Sobolev embedding and Lemma~\ref{LEM:Kh}, this yields
\begin{equation}
\| u(1+i\mathcal{K}_h)\cj{u} R_0 f \|_{L^2}
\leq \|u\|_{L^4}^2 \big( 1 + \| \mathcal{K}_h \|_{L^4\to L^4}\big) \|R_0 f\|_{L^\infty}
\lesssim \kappa^{-\frac12} \|u\|_{H^{\frac14}}^2 \|f\|_{L^2} .
\label{op 1}
\end{equation}

\noi
For the contribution of $\J_h$, we have
\begin{equation}
\| \mathcal{J}_h f \|_{L^\infty} \lesssim \tfrac{1}{h} \|f\|_{L^1},
\label{JLinfty}
\end{equation}
which again is uniform in $1\leq h<\infty$,
and so
\begin{equation}
\| u \J_h\cj{u} R_0 f \|_{L^2}
\lesssim \tfrac{1}{h} \|u\|_{L^2}^2 \| R_0f \|_{L^2} 
\lesssim \tfrac{1}{h \kk} \|u\|_{L^2}^2 \|f\|_{L^2} .
\label{op 2}
\end{equation}

Collecting the previous two steps, we see that we may choose $\kappa\geq 1$ sufficiently large, as in \eqref{k0}, so that
\begin{equation*}
\| u\Pih\cj{u} g \|_{L^2}
\leq \| u\Pih\cj{u}R_0 \|_{\op} \| (|\dx|+\kappa) g \|_{L^2}
\leq \tfrac12 \big( \|\lax_0g\|_{L^2} + \kappa\|g\|_{L^2}  \big) .
\end{equation*}
Self-adjointness then follows from the Kato--Rellich Theorem (see Th.~X.12 in \cite{ReedSimonVol2}).

From~\eqref{op 1} and~\eqref{op 2}, we also see that there is a choice of $C\geq 1$ so that~\eqref{k0} ensures
\begin{equation*}
\| u\Pi_{+,h}\cj{u} R_0 \|_{\op} \leq \tfrac{1}{10} .
\end{equation*}
This in turn guarantees that
\begin{align*}
\big| \big\langle f,(\lax_{u;h}^2+\kappa^2)f \big\rangle - \big\langle f, (\lax_0^2+\kappa^2) f \big\rangle \big|
&\leq 2\| u\Pih\cj{u}f \|_{L^2} \| \lax_0f\|_{L^2} + \|u\Pih\cj{u}f\|_{L^2}^2 \\
&\leq \tfrac{1}{5} \| (|\dx|+\kappa) f\|_{L^2} \| |\dx|f \|_{L^2} + \tfrac{1}{100} \| (|\dx|+\kappa) f\|_{L^2}^2 \\
&\leq \tfrac{21}{100} \| (|\dx|+\kappa) f\|_{L^2}^2 \\
&= \tfrac{21}{100} \langle f, (|\dx|+\kappa)^2 f \rangle \\
&\leq \tfrac{21}{50} \langle f, (\lax_0^2 + \kappa^2) f \rangle .
\end{align*}
The claim~\eqref{mono} then follows.
\end{proof}

We note that the Peter operator $\peter_{u;h}$ in \eqref{Lax} is special because~\eqref{INLSG} may be written as
\begin{equation}
\partial_tu = \peter_{u;h} u .
\label{the one true Peter op}
\end{equation}
Combining this with~\eqref{lax}, it is straightforward to verify that the ``polynomial'' quantities
\begin{align}
E_{k}^{h}(u) = \jb{u, \L^{k}_{u;h}u} \quad \text{for} \quad u\in H^{\infty}(\R)
\label{energies}
\end{align}
are conserved for any $k\geq 0$.  Moreover, these functionals extend continuously to~$H^{\frac k 2 }(\R)$; see Proposition~\ref{PROP:poly} below for details.

More generally, the following lemma implies that for arbitrary functions $f:\R\to\R$, the quantity
\begin{equation*}
t\mapsto \big\langle u(t) , f(\lax_{u(t);h})\,u(t) \big\rangle
\end{equation*}
is automatically conserved.  As we will discuss further below, by selecting a function $f$ that is not a polynomial we will be able to address the intermediate regularities $s \in (\frac14, 1]$ in Theorem~\ref{THM:GWP}.

\begin{lemma}
Let $u(t)$ be a global $H^\infty(\R)$ solution of \eqref{INLSG}.  For any $t_0\in\R$ and any $\psi_0\in L^2(\R)$, there exists a unique $C_tL^2\cap C^1_tH^{-2}$ solution to the initial value problem
\begin{equation}
\tfrac{d}{dt} \psi(t) = \peter_{u(t);h}\psi(t) \quad\text{with}\quad \psi(t_0)=\psi_0
\label{U1}
\end{equation}
and it is global in time.  Moreover, for each $t\in\R$ the mapping $\mathcal{U}(t;t_0) : \psi_0 \to \psi(t)$ is unitary on $L^2$,
\begin{equation}
u(t) = \mathcal{U}(t;t_0)u(t_0), \quad\text{and}\quad \lax_{u(t);h} = \mathcal{U}(t;t_0)\lax_{u(t_0);h} \mathcal{U}^*(t;t_0) .
\label{U2}
\end{equation}
Finally, if $\psi_0\in H^\infty(\R)$, then so too is $\psi(t)$ for all $t\in\R$.
\end{lemma}

The proof of this lemma is elementary.  Indeed, the existence and uniqueness of solutions to~\eqref{U1} follows from a Bona--Smith regularization and a contraction mapping argument.  As $\peter_{u;h}$ is anti-selfadjoint, we deduce that the $L^2$ norm is conserved; this allows us to extend solutions of~\eqref{U1} globally in time, and demonstrates that $\mathcal{U}(t;t_0)$ is unitary.  A standard persistence of regularity argument then shows that $\psi(t)$ is as smooth as $\psi_0$ is.  Lastly, the two identities in~\eqref{U2} follow from the property~\eqref{the one true Peter op} satisfied by $\peter_{u;h}$ and the Lax pair relation~\eqref{lax}.  We refer to \cite[Prop.~2.3]{KLV} for the full details in the case of CCM.

We are now equipped to prove our a-priori estimates:
\begin{proof}[Proof of Theorem~\ref{THM:GWP}]
The main point is to establish \eqref{apriori}. Then, by iterating Theorem~\ref{THM:main}, we obtain the global well-posedness of solutions to \eqref{INLS} with small $L^2$-norm initial data.
Consider the quantity
\begin{equation*}
F(\lax_0^2+\kappa^2) := \big\langle u, (\lax_0^2+\kappa^2)^{\frac14} u \big\rangle .
\end{equation*}
By Loewner's Theorem, the function $x\mapsto x^{\frac14}$ on $(0,\infty)$ is operator monotone (see \cite{Simon} for details).  In particular, the relation~\eqref{mono} implies
\begin{equation*}
F\big(\tfrac12 (\lax_0^2 + \kappa^2)\big) \leq F( \lax_{u;h}^2 + \kappa^2 ) \leq F\big(\tfrac32(\lax_0^2 + \kappa^2)\big) ,
\end{equation*}
and so
\begin{equation*}
\tfrac{1}{C} F (\lax_0^2 + \kappa^2) \leq F( \lax_{u;h}^2 + \kappa^2 ) \leq C F(\lax_0^2 + \kappa^2)
\end{equation*}
for some constant $C>0$.  On the other hand, from~\eqref{U2} we see that
\begin{equation*}
\big\langle u(t), (\lax_{u(t);h}^2+\kappa^2)^{\frac14} u(t) \big\rangle = \big\langle u(0), (\lax_{u(0);h}^2+\kappa^2)^{\frac14} u(0) \big\rangle .
\end{equation*}
Combing the previous two steps, we find
\begin{align}
\| u(t)\|_{H^{\frac14}}^2
&\leq \big\langle u(t) , (\lax_0^2 + \kappa^2)^{\frac14}  u(t) \big\rangle 
\nonumber\\
&\leq C \big\langle u(t) , (\lax_{u(t);h}^2 + \kappa^2)^{\frac14}  u(t) \big\rangle 
\nonumber\\
&= C \big\langle u(0) , (\lax_{u(0);h}^2 + \kappa^2)^{\frac14}  u(0) \big\rangle 
\nonumber\\
&\leq C^2 \big\langle u(0) , (\lax_{0}^2 + \kappa^2)^{\frac14}  u(0) \big\rangle 
\nonumber\\
&\leq C^2 \| u(0) \|_{H^{\frac14}}^2 + C^2 \kappa^{\frac12} \| u(0) \|_{L^2}^2 .
\label{ap 2}
\end{align}

Now we use a bootstrap argument.  Given $u(0)$, consider
\begin{equation*}
\kappa = C \big( 1 + 2C\|u(0)\|_{H^{\frac14}} \big)^4 .
\end{equation*}
Here, $C\geq 1$ is larger than both the constant from the previous paragraph and the constant from~\eqref{k0}.  For this $\kappa$ and any time interval $[0,T]$ on which
\begin{equation}
\|u(t)\|_{H^{\frac14}} \leq 2C\|u(0)\|_{H^{\frac14}} ,
\label{boot 1}
\end{equation}
we know that the monotonicity relation~\eqref{mono} holds.  Then, ~\eqref{ap 2} demonstrates that 
\begin{equation*}
\|u(t)\|_{H^{\frac14}}^2
\leq C^2 \| u(0) \|_{H^{\frac14}}^2 + C^{\frac52} \big( 1 + 2C\|u(0)\|_{H^{\frac14}} \big)^2 \| u(0) \|_{L^2}^2 .
\end{equation*}
Taking $\| u(0) \|_{L^2}^2 \ll C^{-\frac52}$, we deduce
\begin{equation}
\|u(t)\|_{H^{\frac14}} \leq \tfrac32 C\|u(0)\|_{H^{\frac14}} .
\label{boot 2}
\end{equation}
Comparing this with~\eqref{boot 1}, we conclude that \eqref{boot 2} holds for all $t\in\R$.  This proves the $s = \frac14$ case of~\eqref{apriori}.

Now that we know
\begin{equation*}
\sup_{t\in\R} \|u(t)\|_{H^{\frac14}} < \infty ,
\end{equation*}
we may fix $\kappa$ sufficiently large so that the monotonicity relation~\eqref{mono} holds for all $t\in\R$.  An argument parallel to~\eqref{ap 2} using the quantity
\begin{equation*}
F_s(\lax_0^2+\kappa^2) = \big\langle u, (\lax_0^2+\kappa^2)^{s} u \big\rangle 
\end{equation*}
then demonstrates that~\eqref{apriori} holds for any $\frac14 < s \leq 1$.  Note that by Loewner's Theorem, the restriction $s\leq 1$ is necessary for the function $x\mapsto x^s$ to be operator monotone.
\end{proof}

Lastly, we establish the convergence of the polynomial conservation laws for INLS \eqref{INLS} (with $\g=0$) to the polynomial conservation laws of CCM \eqref{CCM}. 

\begin{proposition}\label{PROP:poly}
Let $0<h\leq \infty$, $s\geq -\frac 12$, and $u\in W^{\max(s,0),4}(\R)$. Then, 
\begin{align}
\| \L_{u;\infty}\|_{H^{s+1}(\R)\to H^{s}(\R)} &\les 1+ \|u\|_{W^{\max(s,0),4}(\R)}^{2}. \label{Ebd}
\end{align}
Furthermore, there exists $\ta=\ta(s)>0$ such that
\begin{align}
\| \L_{u;h}-\L_{u;\infty}\|_{H^{s+1}(\R)\to H^{s}(\R)} &\les h^{-\ta}(\|u\|_{W^{\max(s,0),4}(\R)}^{2}+\|u\|_{H^{\max(s,0)}}^2).
\label{Econv} 
\end{align}
with implicit constant uniform over $1\leq h \leq \infty$.
Moreover, if $E_{k}^{h}(u)$, $k\geq 0$, is the generalised energy in \eqref{energies}, it holds that
\begin{align}
 \lim_{h\to \infty} E_{k}^{h}(u)   = E_{k}^{\infty}(u) \label{Econv2}
\end{align}
for any $u\in H^{\frac k2}(\R)$.
\end{proposition}

\begin{proof}
We first establish \eqref{Econv2} assuming \eqref{Ebd} and \eqref{Econv}. Note that \eqref{Ebd} and \eqref{Econv} imply
\begin{align}
\sup_{1\leq h \leq \infty} \| \L_{u;h}\|_{H^{s+1}\to H^{s}} \les 1+\|u\|_{W^{\max(s,0),4}}^{2}+\|u\|_{H^{\max(s,0)}}^{2}. 
\notag
\end{align}
for any $s\geq -\frac 12$.
It follows that 
\begin{align}
\sup_{1\leq h\leq \infty} \| \L_{u;h}^{j} u\|_{H^{s}} \les  (1+\|u\|_{W^{s+j-1,4}}^{2j}+\|u\|_{H^{s+j-1}}^{2j}) \| u\|_{H^{s+j}} \label{laxHsbd}
\end{align} 
for any $j\in \N$ and $s\geq -\frac 12$.  
By a direct computation, we see that $E_{j}^{h}(u) = E_{j}^{\infty}(u)$ for $j=0,1$. Indeed, $E_{0}^{h}(u)=\int |u|^2 dx$ and $E_{1}^{h}(u)$ is the momentum in \eqref{momentum}.
 Thus, to prove \eqref{Econv2}, it suffices to consider $k\geq 2$. Consider first the case when $k\in 2\mathbb{N}$. Then
 \begin{align*}
E_{k}^{h}(u) = \jb{ \L_{u;h}^{k/2}u, \L_{u;h}^{k/2}u},
\end{align*}
so that by \eqref{laxHsbd},
\begin{align*}
|E_{k}^{h}(u)  - E_{k}^{\infty}(u)|  & \les \Big(\|\L_{u;h}^{k/2} u\|_{L^2}+\|\L_{u;\infty}^{k/2} u\|_{L^2} \Big) \|( \L_{u;h}^{k/2} -\L_{u;\infty}^{k/2})u\|_{L^2} \\
& \les \tfrac{1}{h^\ta} (1+\|u\|_{H^{k/2}})^{2k}  \|u\|_{H^{k/2}}^2 \to 0
\end{align*}
as $h\to \infty$. 

Now consider the case when $k\in 2\N+1$ so that we may write $k=2\l+1$ for some $\l\in \N$. Then, we have
\begin{align*}
|E_{k}^{h}(u)  - E_{k}^{\infty}(u)| 
 & \leq \big|\jb{ \L_{u;h}^{\l}u, \L_{u;h} \L_{u;h}^{\l}u} - \jb{ \L_{u;\infty}^{\l}u, \L_{u;\infty} \L_{u;\infty}^{\l}u} \big| \\
 & \leq  \big|\jb{ (\L_{u;h}^{\l}-\L_{u;\infty}^{\l})u, \L_{u;h} \L_{u;h}^{\l}u}\big|+\big|\jb{ \L_{u;\infty}^{\l}u,( \L_{u;h}-\L_{u;\infty}) \L_{u;h}^{\l}u}\big| \\
 &  \hphantom{XX} + \big|\jb{ \L_{u;\infty}^{\l}u, \L_{u;\infty} (\L_{u;h}^{\l}-\L_{u;\infty}^{\l})u}\big|.
 \end{align*}
We just consider the first term since similar estimates will control the remaining two terms. By \eqref{laxHsbd} and \eqref{Econv}, we have
\begin{align*}
\big|\jb{ (\L_{u;h}^{\l}-\L_{u;\infty}^{\l})u, \L_{u;h}^{\l+1}u}\big|  & \les  \| (\L_{u;h}^{\l}-\L_{u;\infty}^{\l})u\|_{H^{\frac 12}}\| \L_{u;h}^{\l+1}u\|_{H^{-\frac 12}}  
\\
&
\les \tfrac{1}{h^\ta} (1+\|u\|_{H^{\l+\frac 12}})^{2k} \|u\|_{H^{\l+\frac 12}}^2
=\tfrac{1}{h^\ta} (1+\|u\|_{H^{\frac k2}})^{2k} \|u\|_{H^{\frac k2}}^2 \to 0  ,
\end{align*}

\noi
as $h\to\infty$. Repeating the argument for the remaining terms, we obtain \eqref{Econv2} for $k\in 2\N+1$ and thus for all $k\in \N$.

It remains to establish \eqref{Ebd} and \eqref{Econv}. We begin with \eqref{Ebd}. By \eqref{Lax}, we have
\begin{align*}
\| \L_{u;\infty} f\|_{H^{s}} \leq \|f\|_{H^{s+1}} + \|u \P_{+}(\cj{u}f)\|_{H^{s}}.
\end{align*}
It thus suffices to estimate the second term on the right-hand side above.  
If $s=0$, we already controlled this term for $f\in H^1$ in the proof of Proposition~\ref{PROP:Lax}.
If $s<0$, we use duality: with $\s=-s$ so $0<\s\leq \frac 12$, by Sobolev embedding, we have
\begin{align}
\begin{split}
\|u \P_{+}(\cj{u}f)\|_{H^{s}} &= \sup_{\|g\|_{L^2}\leq 1} \bigg| \int \jb{\dx}^{-\s}g \cdot u \P_{+}(\cj{u}f) dx \bigg| \\
& \leq  \sup_{\|g\|_{L^2}\leq 1}\| \jb{\dx}^{-\s} g\|_{L^{ \frac{2}{1-2\s} }} \|u\|_{L^4} \|uf \|_{L^{\frac{4}{1+4\s}}} \\
& \les \sup_{\|g\|_{L^2}\leq 1}\|g\|_{L^2} \|u\|_{L^4}^{2} \|f\|_{L^{\frac{1}{\s}}} \\
& \les \|u\|_{L^4}^{2}\|f\|_{H^{\frac 12}}.
\end{split} \label{sneg}
\end{align}
This proves \eqref{Ebd} for $s<0$.

Now we consider the case when $s>0$. By the fractional Leibniz rule, the boundedness of $\P_{+}$ on $L^4(\R)$, and Sobolev embedding, we have 
\begin{align}
\|u \P_{+}(\cj{u}f)\|_{H^{s}} & \les \|u\|_{W^{s,4}} \|\P_{+}(\cj{u}f)\|_{L^4} + \|u\|_{L^4}\| \P_{+}(\cj{u}f)\|_{W^{s,4}}  \les \|u\|_{W^{s,4}}^{2} \|f\|_{H^{s+1}}. \label{Ebd3}
\end{align}
This completes the proof of \eqref{Ebd}.

We now turn to \eqref{Econv}.
Using \eqref{Lax}, \eqref{Pihdecomp}, and the triangle inequality, we have
\begin{align}
\| \L_{u;h}-\L_{u;\infty}\|_{H^{s+1}\to H^s} &\leq \tfrac{1}{2}\| u(\mathcal{K}_{h}-\H)\cj{u}\|_{H^{s+1}\to H^s} + \tfrac{1}{2}\|u \mathcal{J}_{h} \cj{u}\|_{H^{s+1}\to H^s}. \label{Econv3}
\end{align}
We consider the first term on the right-hand side of \eqref{Econv3}, which we split further as
\begin{align*}
\|u(\mathcal{K}_{h}-\H)\P_{\leq \frac{1}{h^{1/2}}} \cj{u}\|_{H^{s+1}\to H^s}  + \|u(\mathcal{K}_{h}-\H)\P_{\geq  \frac{1}{h^{1/2}}}\cj{u}\|_{H^{s+1}\to H^s} = : \I + \II.
\end{align*}
When $s<0$, following the computation in \eqref{sneg} and using Bernstein's inequality and Lemma~\ref{LEM:Kh}, we have
\begin{align*}
\|u(\mathcal{K}_{h}-\H)\P_{\leq \frac{1}{h^{1/2}}} [\cj{u} f]\|_{H^s}  & \les \|u\|_{L^4} \| \P_{\leq \frac{1}{h^{1/2}}} (\cj{u}f)\|_{L^{\frac{4}{1-4s}}} \les h^{-\frac{3+4s}{8}}\|u\|_{L^{4}} \|u\|_{L^2}\|f\|_{L^2}.
\end{align*}
If $s\geq 0$, additionally using the fractional Leibniz rule, we have
\begin{align*}
\| u(\mathcal{K}_{h}-\H)\P_{\leq \frac{1}{h^{1/2}}} [\cj{u} f]\|_{H^s} &  \les \|u\|_{W^{s,4}} \| (\mathcal{K}_{h}-\H)\P_{\leq \frac{1}{h^{1/2}}}[\cj{u}f]\|_{W^{s,4}} \\
& \les  \|u\|_{W^{s,4}}\|\P_{\leq \frac{1}{h^{1/2}}}[\cj{u}f]\|_{L^{4}} \\
& \les h^{-\frac 18}\|u\|_{W^{s,4}} \|u\|_{L^{4}}\|f\|_{L^4} \\
&\les h^{-\frac 18}\|u\|_{W^{s,4}}^{2} \|f\|_{H^{s+1}}.
\end{align*}

\noi 
Combining these estimates, we have shown that there exists $\ta>0$ such that
\begin{align}
\I \les   h^{-\ta} ( \|u\|_{L^4}\|u\|_{L^2}\ind_{s<0}+ \|u\|_{W^{s,4}}^{2} \ind_{s\geq 0}) 
.
\label{Ibd}
\end{align}
Now we consider $\II$, for which we claim that
\begin{align}
\II \les h^{-\ta} \| u\|_{W^{\max(s,0),4}}^{2}. \label{IIbd}
\end{align}
By following the computations in \eqref{sneg} and \eqref{Ebd3}, it is enough to show that
\begin{align}
\| (\mathcal{K}_{h}-\H)\P_{\geq  \frac{1}{h^{1/2}}}\|_{L^{p}\to L^p} \les h^{-\frac 12}.
\label{KhHconv}
\end{align}
for any $\frac43 \leq p\leq 4$.
Let 
\begin{align*}
m(\xi) = \big[ {-i}[ \coth(h\xi) - \sgn(h\xi)] +\tfrac{i}{h\xi}\big] (1-\eta_{h^{-1/2}}(\xi))
\end{align*}
which is the Fourier multiplier associated with $(\mathcal{K}_{h}-\H)\P_{\geq  \frac{1}{h^{1/2}}}$ 
Then,  we have $|\dx^{\al} m(\xi)| \les h^{-\frac 12}|\xi|^{-\al}$ for any $\al \in \N$ and so \eqref{KhHconv} follows from the Mikhlin-H\"{o}rmander theorem \cite[Theorem 6.2.7]{Grafakos}.

It remains to control the second term on the right-hand side of \eqref{Econv3}. We define 
\begin{align}
\Pblo \mathcal{J}_{h} f :  =\mathcal{J}_{h} f -\Pbhi  \mathcal{J}_{h} f  . 
\label{PbloJh}
\end{align}
The operator $\mathcal{J}_{h}$ is a Fourier multiplier operator with multiplier $(ih \xi)^{-1}$, which is singular at the origin. For $F\in L^1(\R)$ and $\al>0$,  $D^{\al} \J_{h} F$ is a tempered distribution in the sense of \cite[Definition 1.26]{BCD} and thus the inhomogeneous Littlewood-Paley decomposition converges in the sense of tempered distributions: 
\begin{align*}
\sum_k \dot{\P_{k}}[ D^{\al} \J_{h} F ] = D^{\al} \J_{h} F.
\end{align*}
Similar to \eqref{Dalcomp}, we then have
\begin{align*}
D^{\al} \Pblo \J_{h} F= \sum_{k\leq 2} \dot{\P_{k}}[ D^{\al} \Pblo \J_{h} F ]
\end{align*}
and thus using \eqref{Jh} and arguing as in \eqref{DeLinfty}, we have 
\begin{align}
\| D^{\al} \Pblo \J_{h} F\|_{L^{\infty}_{x}}  \les \sum_{k<2} 2^{k\al} \| m_{1,\al}\|_{L^1} \|\J_{h}F\|_{L^{\infty}} \les \tfrac{1}{h} \|F\|_{L^1}. \label{DalJh}
\end{align}

We now write
\begin{align*}
\| u \mathcal{J}_{h} \cj{u}\|_{H^{s+1}\to H^s} \leq \| u \Pblo \mathcal{J}_{h} \cj{u}\|_{H^{s+1}\to H^s}+\| u \Pbhi\mathcal{J}_{h} \cj{u}\|_{H^{s+1}\to H^s}.
\end{align*} 
Consider first the contribution from $\Pbhi$. If $s\leq 0$, we have 
\begin{align*}
\| u\Pbhi \J_{h} (\cj{u}f)\|_{H^{s}} \leq \| u\Pbhi \J_{h} (\cj{u}f)\|_{L^2} 
&\leq \|u\|_{L^2} \|\Pbhi \J_{h} (\cj{u}f)\|_{L^{\infty}} \\
& \les \tfrac{1}{h}  \|u\|_{L^2} \| \cj{u}f\|_{L^{2-}} \\
&\les \tfrac{1}{h}\| u\|_{L^2}^2 \|f\|_{H^{\frac 12}},
\end{align*}

\noi 
which is sufficient. If $s>0$, then we use the fractional Leibniz rule as in \eqref{Ebd3} and simply note that $\| \Pbhi \mathcal{J}_{h}\|_{L^4 \to L^4} \les \frac 1h$. This implies the bound 
\begin{align*}
\| u \Pbhi\mathcal{J}_{h} \cj{u}\|_{H^{s+1}\to H^s} \les \tfrac{1}{h} \| u\|_{W^{s,4}}^2.
\end{align*}
Next, for the contribution from $\Pblo \mathcal{J}_{h}$, we also consider the cases $s\leq 0$ and $s>0$. If $s\leq 0$, we use the physical side formula \eqref{Jh}, the definition \eqref{PbloJh},  Bernstein's inequality, and \eqref{JLinfty}, to obtain
\begin{align}
\| u\Pblo \J_{h} (\cj{u}f)\|_{H^{s}} &\leq \| u\Pblo \J_{h} (\cj{u}f)\|_{L^2}  \notag \\
& \les  \| u\|_{L^2} \| \Pblo \J_{h} (\cj{u}f)\|_{L^{\infty}} \notag \\
& \les  \| u\|_{L^2} \big( \tfrac{1}{h} \|u\|_{L^2}\|f\|_{L^2} + \| \Pbhi \J_h (\cj{u}f)\|_{L^{\infty}} \big) \notag \\
& \les \tfrac{1}{h} \| u\|_{L^2}^2 \|f\|_{H^{\frac 12}} . \label{Jh1}
\end{align}

\noi
If $s>0$, then by the fractional Leibniz rule, similar computations as in \eqref{Jh1} and using \eqref{DalJh}, we have
\begin{align*}
\| D^{s} [ u \Pblo \J_{h}( \cj{u}f)]\|_{L^2} & \les \| D^{s}u\|_{L^2} \|\Pblo \J_{h}( \cj{u}f)\|_{L^{\infty}} + \|u\|_{L^2} \|D^{s} \Pblo \J_{h}( \cj{u}f)\|_{L^{\infty}}  \les \tfrac{1}{h}\|u\|_{H^{s}}^{2} \|f\|_{L^2}.
\end{align*} 
Thus, we have shown that 
\begin{align}
\| u \mathcal{J}_{h}\cj{u}\|_{H^{s+1}\to H^s} \les \tfrac{1}{h} (  \|u\|_{W^{\max(s,0),4}}^{2}+\|u\|_{H^{\max(s,0)}}^{2}).
\notag
\end{align}
Combining this with \eqref{Econv3}, \eqref{Ibd}, and \eqref{IIbd} then proves \eqref{Econv}.
\end{proof}

\begin{ackno}\rm 
A.C.~was supported by support CNRS-INSMI through a grant “PEPS Jeunes chercheurs et jeunes chercheuses 2025”.
J.F. was partially supported by the ARC project FT230100588.
T.L. was supported by an AMS-Simons Travel Grant.
\end{ackno}

\end{document}